\documentclass[a4paper, 12pt]{article}

\usepackage{LatexDefinitions}
\usepackage{extarrows}
\usepackage[all]{xy}
\usepackage{bbold}
\usepackage{algorithm}
\usepackage[font=small,labelfont=bf]{caption}

\usepackage{xcolor} 
\definecolor{ocre}{RGB}{243,102,25} 
\definecolor{royalblue}{RGB}{0,78,156}

\usepackage{subcaption}

\usepackage[colorinlistoftodos,prependcaption,textsize=tiny]{todonotes}

\newcommand*{\Cont}{\mathcal C}
\newcommand*{\Contc}{\Cont_\textup c} 
\newcommand{\bnorm}[1]{\left\lVert#1\right\rVert} 
\newcommand{\norm}[1]{\lVert#1\rVert} %
\newcommand{\babs}[1]{\left\lvert #1\right\rvert}
\newcommand{\inv}{^{-1}}
\newcommand{\pair}[2]{\langle#1,#2\rangle}
\newcommand{\bpair}[2]{\left\langle#1,#2\right\rangle}
\newcommand{\setgiven}[2]{\left\{#1\,\middle|\,#2\right\}}
\newcommand{\dd}[2]{\frac{\mathrm{d} #1}{\mathrm{d} #2}}
\newcommand{\textweakstar}{weak\nobreakdash-\(\ast\)}

\newcommand*{\TV}{\mathrm{TV}} 

\newcommand{\OT}{\mathrm{OT}}
\newcommand{\g}{\mathbf{g}}
\newcommand{\tg}{\Tilde{\mathbf{g}}}
\newcommand{\Lip}{\mathrm{Lip}} 

\newcommand{\supp}{\mathrm{supp}}

\newcommand*{\calB}{\mathcal B}

\newcommand*{\calE}{\mathcal E}
\newcommand*{\calM}{\mathcal M}

\newcommand*{\calO}{\mathcal O}
\newcommand*{\calR}{\mathcal R}

\newcommand*{\calX}{\mathcal X}
\newcommand*{\calY}{\mathcal Y}
\newcommand*{\Hil}{\mathcal H}
\newcommand*{\cone}{K_+}
\newcommand*{\prm}{\mathcal P}

\newcommand{\mudot}{\dot{\mu}}

\newcommand{\betadot}{\dot{\beta}}

\newcommand{\fmumu}{{f_{\mu,\mu}}}

\newcommand{\fmunu}{{f_{\mu,\nu}}}
\newcommand{\gmunu}{{g_{\mu,\nu}}}
\newcommand{\fnumu}{{f_{\nu,\mu}}}

\newcommand{\fnunu}{{f_{\nu,\nu}}}

\newcommand*{\Pall}{\mathrm{Adm}}

\newcommand*{\A}{A}

\newcommand*{\B}{B}

\newcommand*{\dS}{\mathsf{d}_S}
\newcommand{\dSM}[1]{\mathsf{d}_{S, #1}}
\newcommand*{\vdot}{\cdot}

\newcommand{\revbegin}{\begin{color}{teal}}
	\newcommand{\revend}{\end{color}}

\title{Approximation and computation \\ of the geodesic Sinkhorn distance}
\author{Hugo Lavenant\thanks{Bocconi University, Department of Decision Sciences and BIDSA, Via Sarfatti 25, 20136 Milan, Italy.}, Jonas Luckhardt\thanks{Georg-August-Universität Göttingen, Institute for Computer Science, Goldschmidstraße 7, 37077, Göttingen, Germany. \\ \emph{E-mail addresses}: hugo.lavenant@unibocconi.it, jonas.luckhardt@uni-goettingen.de, schmitzer@cs.uni-goettingen.de}, Bernhard Schmitzer\addtocounter{footnote}{-1}\addtocounter{Hfootnote}{-1}\footnotemark\addtocounter{footnote}{1}\addtocounter{Hfootnote}{1} \\
	\date{\today}
}
\begin{document}
	
	
	\maketitle
	
	\begin{abstract}
		In [H. Lavenant, J. Luckhardt, G. Mordant, B. Schmitzer, L. Tamanini, The Riemannian geometry of Sinkhorn divergences. \emph{Ann. Inst. H. Poincaré Anal. Non Linéaire} 43 (2026)] we introduced a Riemannian metric $\dS$ on the space of probability distributions obtained from entropic optimal transport, specifically from the Sinkhorn divergence $S_\varepsilon$. In the present work we discuss how to approximate and compute $\dS$. Spatially, we prove Gromov--Hausdorff convergence of the metric and convergence of geodesics for increasingly fine Eulerian discretization of the base space. Temporally, we show $\Gamma$-convergence of the chain discretization $N \sum_{k=0}^{N-1} S_\varepsilon(\mu_k, \mu_{k+1})$ to the energy functional defining $\dS$. We deduce and implement numerical schemes to compute approximations of $\dS$.

        \medskip


        \medskip

	\end{abstract}

    \newpage
    
	\setcounter{tocdepth}{2}
	
	\tableofcontents

	\section{Introduction}
	
	\subsection{Context: Sinkhorn divergence and its geometry}
	
	\paragraph{Optimal transport.}
	With optimal transport, we can lift a distance $d$ over a space $X$ into a meaningful distance on probability measures \(\prm(X)\).
	This theory has several applications in pure and applied mathematics, economics and data science: see e.g. the textbooks \cite{Villani2009,Santambrogio2015,PeyreCuturi,galichon2018optimal,chewi2025statistical} and references therein.
	Given a cost function \(c \colon X \times X \to \R\), the Kantorovich formulation of optimal transport between two measures $\mu, \nu \in \prm(X)$ is given by
	\begin{equation} \label{eq:OT}
		\OT(\mu, \nu) \coloneqq \inf_{\pi \in \Pi(\mu, \nu)} \iint_{X \times X} c(x,y) \diff \pi(x,y),
	\end{equation}
	where we optimize over the set \(\Pi(\mu,\nu)\) of all \emph{transport plans}, which are probability measures \(\pi \in \prm(X \times X)\) with first marginal \(\mu\) and second marginal \(\nu\). 
	When \((X,d)\) is a metric space and \(c(x,y) =  d(x,y)^p\) for some \(p \in [1,\infty)\), this construction leads to the famous Wasserstein distance 
	\begin{equation} \label{eq:Wasserstein_distance}
		W_p(\mu,\nu) \assign \OT(\mu,\nu)^\frac{1}{p}.
	\end{equation}
	This is a metric on the space of probability measures, 
	which metrizes the \textweakstar~convergence of probability measures (in the noncompact case, some conditions on $p$-th moments are required)~\cite[Sec.~6]{Villani2009}.

	\paragraph{The Riemannian flavour of $W_2$-transport.} 
	For the squared distance cost on \(\R^d\), the space of probability measures with the \(W_2\)-metric resembles a Riemannian manifold. 
	The famous Benamou--Brenier formulation of \(W_2\) optimal transport~\cite[Prop.~1]{BB2000} shows that for measures \(\mu,\nu\)  with finite second moment,
	\begin{equation} \label{eq:Benamou_Brenier}
		W_2(\mu,\nu)^2 = \inf \int_0^1 \int_{\R^d} \norm{v_t(x)}^2 \diff \mu_t(x) \diff t,
	\end{equation}
	where the infimum ranges over all paths \((\mu_t)_t\) of probability densities that connect \(\mu\) to \(\nu\) and time dependent velocity fields \(v \colon [0,1] \times \R^d \to \R^d\), satisfying the continuity equation $\partial_t \mu_t + \ddiv(v_t \mu_t) = 0$
	in a weak sense. As the optimization in $v$ forces it to be curl free, the reformulation~\eqref{eq:Benamou_Brenier} formally resembles a Riemannian energy functional associated to the metric tensor defined as (see e.g.~\cite{Otto2001})
	\begin{equation}
		\label{eq:metric_tensor_ot}
		\g_\mu^0(\ddiv(v\mu),\ddiv(v\mu)) = \norm{v}_{L^2(\R^d,\mu)}^2 \qquad \text{if } v = \nabla \phi.
	\end{equation}

	\paragraph{Entropic regularization and Sinkhorn divergence.}
	Despite being convex, numerically solving the optimal transport problem~\eqref{eq:OT} can be challenging due to its high dimensionality (after discretization). The objective in \eqref{eq:OT} is linear in \(\pi\), hence not strongly convex. One can smoothen the problem by adding a strongly convex regularization term to the right-hand side of~\eqref{eq:OT}. A popular choice is the Kullback--Leibler divergence 
	\begin{equation}
		\label{eq:def_kl}
		\KL(\pi \, | \, \rho) \coloneqq 
		\begin{cases}
			\displaystyle{\int_X h \left( \dd{\pi}{\rho} \right) \diff \rho} &\tn{if } \pi \ll \rho,\, \rho \geq 0,\\
			+ \infty &\tn{else,}
		\end{cases}
	\end{equation}
	with \(h(s) \coloneqq s \log(s) - s + 1\) for \(s > 0\), \(h(0) \coloneqq 1\), \(h(s) \coloneqq + \infty\) for \(s < 0\). This leads to the entropic optimal transport problem
	\begin{equation} \label{eq:OTeps}
		\OT_\varepsilon(\mu, \nu) \coloneqq \inf_{\pi \in \Pi(\mu, \nu)} \left\{ \iint_{X \times X} c \diff \pi + \varepsilon \KL(\pi \, | \, \mu \otimes \nu) \right\}
	\end{equation}
	with a regularization parameter \(\varepsilon > 0\). 
	As \(\varepsilon \to 0\) we have (see, for instance,~\cite{Leonard2012,Carlier2017,Carlier2023} and references therein)
	\begin{equation*}
		\OT_\varepsilon(\mu,\nu) \to \OT(\mu,\nu),
	\end{equation*} 
	which makes \(\OT_\varepsilon\) a good substitute for \(\OT\). \(\OT_\varepsilon\) also enjoys a better sampling complexity than \(\OT\) \cite{Genevay:Samples,MenaWeed}. 
	The theory of convex optimization yields a corresponding dual problem to~\eqref{eq:OTeps} with the same optimal value~\cite{Nutz:IntroEOT}, see also Section~\ref{sec:sub:dual_OTeps} below. The dual problem is a maximization problem over a pair of functions \(f,g \in \Cont(X)\).
	It can be solved efficiently for instance via the renowned Sinkhorn algorithm~\cite{Sinkhorn1964,Cuturi2013,PeyreCuturi}, which performs alternating optimization on \(f\) and \(g\). We call the optimal \(f\) and \(g\) \emph{Schrödinger potentials}.

	An issue that arises from entropic regularization is that \(\OT_\varepsilon\) is not a metric. Generally \(\OT_\varepsilon(\mu,\mu) > 0\) and \(\inf_{\nu} \, \OT_\varepsilon(\mu,\nu)\) is attained not at \(\mu\) itself, but at a somewhat contracted version of \(\mu\) ~\cite[Sec.~3.3.2]{Feydy2020}. As proposed in~\cite{Genevay:Loss}, this problem can be partially addressed with a simple de-biasing, leading to the Sinkhorn divergence:
	\begin{equation} \label{eq:Seps}
		S_\varepsilon(\mu,\nu) \assign \OT_\varepsilon(\mu,\nu) - \frac{1}{2} \OT_\varepsilon(\mu,\mu) - \frac{1}{2} \OT_\varepsilon(\nu,\nu).
	\end{equation}
	Clearly \(S_\varepsilon(\mu,\mu) = 0\) for all probability measures \(\mu\). Symmetry of \(\OT\), \(\OT_\varepsilon\), and \(S_\varepsilon\) is equivalent to symmetry of \(c\). Positive definiteness is nontrivial but holds in case of a suitably regular cost function \(c\) as shown in \cite{Feydy2018}.
	
	\begin{theorem}[{\cite[Thm.~1]{Feydy2018}}]
		\label{theorem:FeydyEtAl}
		If \(k_c=\exp(-c/\varepsilon)\) is a positive definite, universal, jointly Lipschitz continuous kernel over a compact space \(X\), then \(S_\varepsilon\) defines a symmetric, positive definite and smooth loss function that is convex in each of its input variables. It also `metrizes' the \textweakstar~convergence.
	\end{theorem}
	
	The combination of differentiability, positive definiteness, and computational feasibility makes \(S_\varepsilon\) a suitable candidate for the loss function in machine learning tasks that involve learning a probability distribution: the motivation in~\cite{Genevay:Loss} was generative modelling. 
	Positive definite universal kernels include the Gaussian \(\exp(-\frac{1}{2\varepsilon}\norm{x-y}_2^2)\) and the Laplacian kernel \(\exp(-\frac{1}{\varepsilon}\norm{x-y}_1)\) on domains in \(\R^d\). For an overview of more such kernels, we refer to~\cite[Table~3.1]{Muandet2017}.

	\paragraph{The Riemannian geometry of Sinkhorn divergences.}
	While a great loss function, \(S_\varepsilon\), or more precisely its square root, does not satisfy the triangle inequality \cite[Sec.~7.1]{RGSD}. This is the starting point of our previous work~\cite{RGSD}. 
	Since the Sinkhorn divergence is positive definite, its Hessian at \(\mu=\nu\) can be used as a metric tensor \(\g_\mu\): we introduced in \cite{RGSD} 
	\begin{equation*}
		\g_\mu(\mudot,\mudot) \assign \lim_{t \to 0} \frac{S_\varepsilon(\mu,\mu + t \mudot)}{t^2},
	\end{equation*}
	where say for simplicity $\mudot$ is a signed measure such that $\mu + t \mudot$ is a probability measure for $t$ small enough.
	The precise expression of $\g_\mu$ can be found below in~\eqref{eq:expansion_Sinkhorn_def}. We use this metric tensor to induce a Riemannian metric $\dS$ by minimizing the action functional $E$ over a suitable class of paths.
	Namely, inspired by~\eqref{eq:Benamou_Brenier}, for \(\mu, \nu \in \prm(X)\) we set
	\begin{align} \label{eq:intro:dS}
		E((\mu_t)_t) & \assign \int_0^1 \g_{\mu_t}(\mudot_t,\mudot_t) \diff t, & \dS(\mu, \nu) & \assign \inf_{(\mu_t)_t \in \Pall(\mu, \nu)} \, E((\mu_t)_t)^\frac{1}{2},
	\end{align}
	where $\Pall(\mu, \nu)$ denotes the set of admissible paths connecting $\mu$ to $\nu$.
	This construction is explored rigorously in \cite{RGSD} and we summarize the main results that are relevant for the present article in Section~\ref{sec:sub:RGSD}.
	Most importantly, under the same regularity assumptions as in Theorem~\ref{theorem:FeydyEtAl}, \(\dS\) is indeed a metric that metrizes the \textweakstar~topology on \(\prm(X)\), minimizers in~\eqref{eq:intro:dS} exist and we call these curves geodesics.

	\paragraph{A key change of variables.}
	Both in our article \cite{RGSD} as well as in~\cite{Hardion2025} studying gradient flows in the geometry of Sinkhorn divergences, a key role is played by a change of variables \(\mu \mapsto \beta_\mu \assign \exp(-\fmumu/\varepsilon)\), where \(\fmumu\) is the entropic self-transport potential of \(\mu\), see Definition~\ref{def:alpha_beta_variables}. It embeds $\prm(X)$ into a segment of the unit sphere in the Reproducing Kernel Hilbert Space (RKHS) \(\Hil_c\) induced by the cost kernel \(k_c = \exp(-c/\varepsilon)\). Such a change of variables also plays a prominent role in the present work. 

    \begin{remark}
    All our analysis is performed for a fixed $\varepsilon > 0$, thus we actually mean that $\dS = \mathsf{d}_{S_\varepsilon}$, which depends on $\varepsilon$, defines a Riemannian-like metric for all $\varepsilon > 0$. As $\varepsilon \to 0$ some results on the convergence of $\g_\mu$ to $\g^0_\mu$ as defined in~\eqref{eq:metric_tensor_ot} hold (see \cite[Sec.~4.4]{RGSD} for formal arguments and \cite{mordant2026convergence} for a positive result in a particular case), but we will not focus on it in the present work. 
    \end{remark}
	
	\subsection{Outline and contributions}
	
	While \cite{RGSD} defined the metric $\dS$, this paper focuses on how one can approximate it, both from a theoretical and numerical point of view. In formula~\eqref{eq:intro:dS} we can approximate the space variable (that is, substitute $X$ with a simpler, e.g. finite, subset of it), or the temporal variable (that is, substitute the temporal integral with some form of Riemann sum). We study both these approximations, and explain how they lead to implementable algorithms in order to compute $\dS$ and its geodesics. Along the way we develop new estimates relating $\dS$, $S_\varepsilon$ and the metric tensor, and discuss a concept of approximate geodesics for our metric. As in \cite{RGSD,Hardion2025} we restrict $X$ to be a compact space. 
	
	\paragraph{Gromov--Hausdorff convergence for spatial discretization.}
	Let \((x_m)_m\) be a sequence that is dense in \(X\). For \(M \in \N\) set \(X_M = \{x_1, x_2, \dots, x_M\}\). We denote by \(\dSM{M}\) the geodesic restriction of \(\dS\) to paths in \(\prm(X_M)\). Importantly, even if $X_M$ is a finite space, $\dSM{M}$ is still a geodesic distance on $\prm(X_M)$. In Section~\ref{sec:Gromov_Hausdorff} we show Gromov--Hausdorff convergence \((\prm(X_M),\dSM{M}) \to (\prm(X),\dS)\). Specifically, we find a projection \(R_M\) of \(\prm(X)\) to \(\prm(X_M)\) (Definition~\ref{def:cone_projection}), such that
	\begin{equation*}
		\lim_{M \to \infty} \sup_{\mu,\nu \in \prm(X)} \, \babs{\dS(\mu,\nu) - \dSM{M}(R_M(\mu),R_M(\nu))} = 0,
	\end{equation*}  
	see Theorem~\ref{theorem:unif_conv_RM}.
	It implies that geodesics in \(\dSM{M}\) converge to geodesics of \(\dS\) as \(M \to \infty\) (Corollary~\ref{cor:GH_convergence_geodesics}). The projection \(R_M\) is induced by an orthogonal projection in the RKHS \(\Hil_c\) which we study in Section~\ref{sec:projection}. This projection $R_M$ is also a key technical tool for the study of the temporal discretization. 
	
	Note that \(\dSM{M}\) remains a geodesic distance on \(\prm(X_M)\) for any finite \(M\), whereas the Wasserstein-2 metric does not even have any finite length, non-constant continuous paths in \(\prm(X_M)\) \cite{GigliMaas-GromovHausdorff2013}. That means that even if the Gromov--Hausdorff convergence \((\prm(X_M), W_2) \to (\prm(X),W_2)\) as \(M \to \infty\) holds, convergence of geodesics is meaningless. As an alternative, \cite{GigliMaas-GromovHausdorff2013,GlKoMa2018} consider a carefully designed dynamic variant of Wasserstein transport on graphs and show convergence to \(W_2\). For our metric, no such effort is necessary: \(\dS\) and \(\dSM{M}\) are defined in one and the same way, and the proof of Gromov-Hausdorff convergence is arguably simpler.

	\paragraph{Local expansion of the Sinkhorn divergence and its geodesic distance.}
	
	Given how the metric tensor is defined and the definition of the distance $\dS$ we expect to have
	\begin{equation}
		\label{eq:to_prove_approx}
		S_\varepsilon(\mu,\nu) \approx \g_\mu(\nu - \mu, \nu - \mu) \approx \tg_\mu(\beta_\nu - \beta_\mu, \beta_\nu - \beta_\mu) \approx \dS^2(\mu,\nu) \qquad \text{ as } \nu \to \mu.
	\end{equation}
	Here $\tg_\mu$ corresponds to the metric tensor in the $\beta$ variable, as defined below in~\eqref{eq:def:tgmu}. We make this intuition quantitative: the whole challenge is to understand in which norms the small $o$'s can be taken. As we prove in Theorem~\ref{theorem:expansion_dS}, in all equalities but the first one in~\eqref{eq:to_prove_approx}, the $\approx$ can be substituted by an $o(\dS^2(\mu,\nu))$. On the other hand for the Sinkhorn divergence we prove in Theorem~\ref{theorem:estimates_with_Wc} that $S_\varepsilon(\mu,\nu) \geq \dS^2(\mu,\nu) + o(\dS^2(\mu,\nu))$, but for the opposite inequality we only obtain an upper bound of $S_\varepsilon(\mu,\nu)$ by $\dS^2(\mu,\nu)$ with a remainder in a stronger norm, namely a one-Wasserstein distance. 
	As a corollary, we study approximate geodesics with respect to $\dS$ at small scales. This serves to gain a better intuition for the geometry induced by $\dS$: both vertical interpolation (in $\prm(X)$) and arc interpolation (in $\Hil_c$) give paths with approximately minimal energy.  
	
	The core of the proofs in this section is continuity estimates of the metric tensors with respect to the base point: the approach is close to \cite[Sec.~3 \& Sec.~4]{RGSD}, but the study is finer than this previous work.

	\paragraph{Chain discretization and \texorpdfstring{$\Gamma$}{Gamma}-convergence for temporal discretization.}
	Given the approximation~\eqref{eq:to_prove_approx}, it makes sense to discretize the temporal integral in~\eqref{eq:intro:dS} with the Sinkhorn divergence $S_\varepsilon$.
	Doing this with time steps of size \(\Delta t = 1/N\) yields the time-discrete functional
	\begin{equation} \label{eq:intro:Sinkhorn_chain}
		\ol{E}^N(\mu_0,\mu_1,\dots,\mu_N) \assign N \cdot \sum_{k=0}^{N-1} S_\varepsilon(\mu_k, \mu_{k+1})
	\end{equation}
	for an \((N+1)\)-tuple \((\mu_k)_{k=0}^N\) in \(\prm(X)\).
	The Sinkhorn chain~\eqref{eq:intro:Sinkhorn_chain} is reminiscent of the setting studied in~\cite{RumpfWirth2015} for discretizing the action of curves on Riemannian manifolds.
	Importantly, as $S_\varepsilon$ is intensively used as a loss function in machine learning, there are mature libraries available to evaluate and optimize the energy $\ol{E}^N$ in~\eqref{eq:intro:Sinkhorn_chain} \cite{Feydy2018,cuturi2022optimal}. 
	
	In Theorem~\ref{theorem:gamma_convergence} we show that $\ol{E}^N$ $\Gamma$-converges to $E$ as $N\to\infty$ on the set of admissible paths for \eqref{eq:intro:dS} under a suitable identification of $(N+1)$-tuples and paths. It implies
	\begin{equation} \label{eq:intro:chain_convergence}
		\dS^2(\mu,\nu) = \lim_{N \to \infty} \inf \, \setgiven{N \cdot \sum_{k=0}^{N-1} S_\varepsilon(\mu_k,\mu_{k+1})}{\mu_0,\ldots,\mu_N \in \prm(X), \, \mu_0=\mu,\, \mu_N = \nu}
	\end{equation}
	with convergence of a suitable interpolation of a minimizing sequence to a geodesic curve (\Cref{cor:chain_energy_approx}), thus providing a theoretical foundation to our proposed numerical scheme for approximating the distance \(\dS\) and its geodesics. 
	
	The \(\Gamma\)-liminf is relatively easy once we have the lower bound $S_\varepsilon(\mu,\nu) \geq \g_\mu(\nu - \mu, \nu - \mu) + o(\dS^2(\mu,\nu))$. The $\Gamma$-limsup is harder because we do not have the analogue upper bound for $S_\varepsilon$, so we cannot directly follow the techniques of~\cite{RumpfWirth2015} for the finite dimensional case. To circumvent this issue, we first prove the $\Gamma$-limsup when the base space $X$ is finite, as many smoothness issues disappear in this case. We combine this with our projection $R_M$ of Section~\ref{sec:projection} onto a finite base space to handle the general case.

	A careful reading of our proofs reveals that we can prove a joint limit of spatial and temporal discretization. That is, if in~\eqref{eq:intro:chain_convergence} we restrict competitors to $\prm(\{x_1, x_2, \dots, x_M\})$ where $(x_m)_m$ is a sequence dense in $X$, then the joint limit $N, M \to \infty$ also converges to the energy $E$, without any restriction on the rate at which $N,M$ grow (Remark~\ref{rmk:joint_limit}). Similar joint discretization results have been proved for discretization of the classical Benamou-Brenier energy~\eqref{eq:Benamou_Brenier} \cite{lavenant2021unconditional}.
	
	\paragraph{Numerical scheme.} 
	We approximate \(\dS(\mu,\nu)\) with~\eqref{eq:intro:chain_convergence}. For the measures, we implement Eulerian, or Lagrangian discretization (i.e.~\(\mu_t = \sum_{i=1}^M m_i(t) \delta_{x_i}\), or \(\mu_t = \sum_{i=1}^M m_i \delta_{x_i(t)}\)). In Eulerian discretization, optimization over each individual chain link is a convex Sinkhorn barycenter problem which can be solved using~\cite[Alg.~1]{Janati20}. We perform alternating optimization over the even and odd links. For numerical efficiency, we propose to begin with a low number of midpoints and refine the chain at convergence until the desired depth is reached. 
	
	Both approaches capture the prominent splitting effects of our metric. Convergence of the score is relatively fast. However, visually, the curves require many iterations to approach their limit curve. This is, in part, because locally undirected deviations (akin to blur) are barely noticed by the Sinkhorn divergence.

	\section{Preliminaries} \label{sec:preliminaries}
	
	\subsection{Setting and notation.} 
	
	\begin{assumption*}[Valid throughout the whole article] 
		We adopt the setting of \cite{RGSD}. Let \(X\) be a compact metrizable space. We fix \(\varepsilon > 0\) and a symmetric, continuous, non-negative cost function \(c \in \Cont(X \times X; [0,\infty))\) that induces a positive definite universal kernel \(k_c(x,y) = \exp(-c(x,y)/\varepsilon)\). 
	\end{assumption*}
	
	\noindent Though not explicitly written, most of the quantities we introduce depend on $\varepsilon$. The limit $\varepsilon \to 0$ is not our concern here. 
	
	\begin{remark}
		As shifting $c$ by a constant does not change $S_\varepsilon$ (hence it does not change $\g_\mu$ either), the non-negativity assumption on \(c\) is purely cosmetic and only simplifies a few multiplicative constants.     
	\end{remark}
	
	The spaces of Borel probability measures and Borel signed measures on $X$ are denoted by $\prm(X)$ and $\mathcal{M}(X)$ respectively, and is always endowed with the topology of \textweakstar~convergence. In particular $\prm(X)$ is compact. 
	The space of continuous functions is $\Cont(X)$, with supremum norm $\| \vdot \|_\infty$. We denote the dual pairing between a space and its dual space by $\langle \vdot,\vdot \rangle$ and write $\langle \vdot,\vdot \rangle_\Hil$ for the inner product on a Hilbert space $\Hil$.
	
	On a product space, we denote by $\proj^1$ and $\proj^2$ the projections onto the first and second coordinates respectively. We write $f_\# \mu$ for the push-forward of the measure $\mu$ by the measurable map $f$, defined as $(f_\# \mu)(A) = \mu(f^{-1}(A))$.

	We write $\ones_X$ for the constant function equal to $1$ over $X$. If $(Y, \| \vdot \|_Y)$ is a space of functions defined over $X$ which contains $\ones_X$ we define $Y/\R$ as the set of $y \in Y$ quotiented by the equivalence relation $y \sim y'$ if $y - y' \in \R \ones_X$. 
	We denote by $Y^*$ the topological dual of $Y$, and $Y^*_0$ the subset of $\sigma \in Y^*$ such that $\langle \sigma, \ones_X \rangle = 0$. The pairing $\langle \sigma, y\rangle$ between an element $\sigma \in Y^*_0$ and a function $y \in Y/\R$ is defined without ambiguity.
	
	By convention, all curves $(\mu_t)_t$, $(\beta_t)_t$, etc. are indexed by the time parameter $t \in [0,1]$.
	We use Landau's notations of ``small $o$'' and ``big $\calO$'', where we will underline over which parameter they can be taken uniformly.

	\subsection{The entropic dual problem} \label{sec:sub:dual_OTeps}
	We recall some basic facts about entropic optimal transport as already collected in~\cite[Sec.~2]{RGSD}.
	The entropic optimal transport problem, as defined in~\eqref{eq:OTeps}, is a constrained convex optimization problem. A dual formulation in the sense of convex analysis is given by
	\begin{equation} \label{eq:entropic_dual}
		\OT_\varepsilon(\mu, \nu) = \sup_{f, g \in \Cont(X)} \, \pair{\mu}{f} + \pair{\nu}{g} - \varepsilon \bpair{\mu \otimes \nu}{\exp \left( \frac{1}{\varepsilon} (f \oplus g - c) \right) - 1}.
	\end{equation}
	Here $f \oplus g$ denotes the function $X \times X \ni (x,y) \mapsto f(x)+g(y)$.
	The dual problem admits maximizers \(\fmunu\), \(\gmunu\), which we call \emph{Schrödinger potentials}: they are solutions to
	\begin{equation} \label{eq:Schroedinger_system}
		\fmunu = T_\varepsilon(\gmunu,\nu), \qquad
		\gmunu = T_\varepsilon(\fmunu, \mu),
	\end{equation}
	where $T_\varepsilon(f,\mu)(y) \coloneqq - \varepsilon \log \int_X \exp \left( \frac{1}{\varepsilon} (f(x) - c(x,y)) \right) \diff \mu(x)$ is an operator from $\Cont(X) \times \prm(X)$ to $\Cont(X)$.
	The Schrödinger potentials are unique in the space \(\Cont(X) \times \Cont(X) / \sim\) with the relation \((f,g)~\sim~(f~+~\lambda\ones_X,~g~-~\lambda\ones_X)\) for all \(\lambda \in \R\). In particular the sum $\fmunu \oplus \gmunu$ is unique. 
	By~\cite[Thm.~1.2]{DiMarino2019} one can choose the constant \(\lambda\) in such a way that
	\begin{equation} \label{eq:sup_bound_potentials}
		\norm{\fmunu}_\infty, \norm{\gmunu}_\infty \leq \frac{3}{2} \norm{c}_\infty.
	\end{equation}
	For Schrödinger potentials \(\fmunu\), \(\gmunu\) the unique optimal minimizer to the primal problem \eqref{eq:OTeps} is given by $\pi_{\mu,\nu} = k_{\mu,\nu} \cdot \mu \otimes \nu$, where
	\begin{equation} \label{eq:kmunu}
		k_{\mu,\nu} (x,y) \assign \exp \left( \frac{1}{\varepsilon}\left(\fmunu(x) + \gmunu(y) - c(x,y) \right) \right).    
	\end{equation}

	\subsection{The Riemannian geometry of Sinkhorn divergences} \label{sec:sub:RGSD}
	In this section we recall the main definitions and results of \cite{RGSD} that will be used throughout the present article. We start by introducing kernel operators which play a major role. We introduce them in a slightly greater generality than in \cite{RGSD}, which only deals with the case $\mu = \nu$ in the definition below.
	
	\begin{definition}[Integral operators induced by transport kernels {\cite[Def.~3.3]{RGSD}}] \label{def:kernel_operators}
		Given the kernels $k_{\mu,\nu}$ associated with measures $\mu,\nu \in \prm(X)$, \eqref{eq:kmunu}, and $k_c = \exp(-c/\varepsilon)$, we define the kernel operators \(K_{\mu,\nu}\), \(H_{\mu,\nu}\) and \(H_c\) by 
		\begin{align}
			\label{eq:KerneOps}
			\begin{split}
				&K_{\mu,\nu} \colon \Cont(X) \to \Cont(X), \quad K_{\mu,\nu} [\phi] (y) = \int_X \phi(x) k_{\mu,\nu}(x,y) \diff \mu(x),\\
				&H_{\mu,\nu} \colon \calM(X) \to \Cont(X), \quad H_{\mu,\nu} [\eta] (y) = \int_X k_{\mu,\nu}(x,y) \diff \eta(x) = \pair{\eta}{k_{\mu,\nu}(\vdot,y)}, \\
				&H_c \colon \calM(X) \to \Cont(X), \quad H_c[\eta](y) = \int_X k_c(x,y) \diff \eta(x) = \pair{\eta}{k_c{(\vdot,y)}}.
			\end{split}
		\end{align}
		When \(\nu = \mu\), we simplify the notation and write \(k_\mu\), \(K_\mu\) and \(H_\mu\) respectively.
	\end{definition}
	
	The three operators are continuous and compact \cite[Prop.~3.6]{RGSD}, and we will extend their domain of definition below. It is \(K_{\mu,\nu}[\ones_X] = \ones_X\) by the Schrödinger system~\eqref{eq:Schroedinger_system}. On the quotient space \(\Cont(X)/\R\) with its norm $\norm{\phi}_{\Cont/\R} = \inf_{\lambda \in \R} \, \norm{\phi - \lambda \ones_X}_\infty$ 
	the operator \(K_{\mu,\nu}\) is a contraction with 
	\begin{equation} \label{eq:contraction_Kmunu}
		\norm{K_{\mu,\nu}[\phi]}_{\Cont/\R} \leq q \cdot \norm{\phi}_{\Cont/\R} \quad \tn{for all} \quad \phi \in \Cont(X),
	\end{equation}
	where \(q = 1 - \exp(-4 \norm{c}_\infty/\varepsilon) \in (0,1)\) \cite[Prop.~3.8]{RGSD} (phrased for the diagonal \(K_\mu\) in~\cite{RGSD}, but the proof remains the same). As a consequence of~\eqref{eq:contraction_Kmunu}, we obtain existence of the inverse operator \((\id - K_{\mu,\nu} K_{\nu,\mu})\inv\) on \(\Cont(X)/\R\) \cite[Thm.~3.9]{RGSD}.

	\paragraph{The expansion of the Sinkhorn divergence}
	Take \((\mu_t)_{t \in (-\tau,\tau)}\) valued in \(\prm(X)\), which we assume to be continuously differentiable in $\mathcal{M}(X)$ endowed with the \textweakstar{}~topology. That means that, for any $\phi \in \Cont(X)$, the function $t \mapsto \langle \mu_t, \phi \rangle$ is differentiable, that the derivative is linear in $\phi$, and that the map $(t,\phi) \in (- \tau, \tau) \times \Cont(X) \to \langle \dot{\mu}_t,\phi \rangle \in \R$ is jointly continuous. Then \cite[Thm. 3.4]{RGSD} proves that, with $\mu = \mu_0$ and $\mudot = \mudot_0$, 
	\begin{equation}
		\label{eq:expansion_Sinkhorn}
		\lim_{t \to 0} \frac{S_\varepsilon(\mu, \mu_t)}{t^2} =  \frac{\varepsilon}{2} \bpair{ \mudot }{ (\id - K_\mu^2)^{-1} H_\mu \left[ \mudot \right] }.
	\end{equation}
    The right hand side is taken as the definition of the metric tensor: at least if $\mudot \in \calM_0(X)$,
    \begin{equation}
    \label{eq:expansion_Sinkhorn_def}
     \g_\mu(\mudot,\mudot) \assign  \frac{\varepsilon}{2} \bpair{ \mudot }{ (\id - K_\mu^2)^{-1} H_\mu \left[ \mudot \right] }.  
    \end{equation}
	
	\begin{remark}
		\label{rmk:diff_setting}
		In \cite{RGSD} this expansion is extended to a larger class of paths if $c$ is more regular. Namely, assume $X$ is the closure of a bounded open set in \(\R^d\) and that all mixed partial derivatives \(\partial_x^{\alpha_1} \partial_y^{\alpha_2} c(x,y)\) exist and are continuous for all multiindices \(\alpha_1,\alpha_2 \in \N_0^d\) with \(\abs{\alpha_1}, \abs{\alpha_2} \leq m\). With $\Cont^m(X)$ the space of $m$-times continuously differentiable functions, take a path \((\mu_t)_{t \in (-\tau,\tau)}\) valued in \(\prm(X)\), and continuously differentiable in the space $\Cont^m(X)^*$ endowed with the \textweakstar~topology (the previous case corresponding to $m=0$). Again with $\mu = \mu_0$ and $\mudot = \mudot_0$, \cite{RGSD} proves that~\eqref{eq:expansion_Sinkhorn} still holds in this case.     
	\end{remark}

	\paragraph{The reproducing kernel Hilbert spaces $\Hil_c$ and $\Hil_\mu$.}
	The natural smoothness condition for admissible paths in the definition of the metric \(\dS\) in~\eqref{eq:intro:dS} is described in terms of a reproducing kernel Hilbert space (RKHS). For convenience of the reader, we briefly recall the basic concept here. For more details on RKHS, we refer to~\cite{Micchelli2006, Sriperumbudur2011}. The following is a brief version of~\cite[Sec.~4.1,~App.~B]{RGSD}.
	
	Let \(k \colon X \times X \to [0,\infty)\) be a continuous non-negative and symmetric universal kernel. Then the RKHS \(\Hil \subseteq \Cont(X)\) induced by \(k\) is the closure of \(\vecspan\setgiven{k(x,\vdot)}{x \in X}\) with respect to the inner product defined as $\langle k(x,\vdot), k(y,\vdot) \rangle_{\Hil} = k(x,y)$ for $x,y \in X$, and extended by linearity. This forms a space of functions on \(X\). Under the assumption of universality that we made, the space \(\Hil\) is dense in \(\Cont(X)\) in supremum norm. There holds \(\phi(x) = \pair{\phi}{k(x,\vdot)}_{\Hil}\), this is the reproducing property. The embedding \(\Hil \xhookrightarrow{} \Cont(X)\) is continuous, hence \(\calM(X) \xhookrightarrow{} \Hil^*\). 
	
	\begin{definition}
		We write \(\Hil_c\) and \(\Hil_\mu\) for the RKHS induced by the kernels \(k_c\) and \(k_\mu\).     
	\end{definition}
	
	\noindent The map $\phi \mapsto \phi \cdot \exp \left( \fmumu / \varepsilon \right)$ is an isometry from $\Hil_c$ onto $\Hil_\mu$ (\cite[Prop.~4.7]{RGSD}). Very importantly \(H_c\) and \(H_\mu\), defined in \eqref{eq:KerneOps} as acting on signed measures, are the Riesz embedding operators from \(\calM(X)\) into \(\Hil_c\) and \(\Hil_\mu\) respectively. Indeed, they satisfy the relation 
	\begin{equation*}
		\pair{H_c[\eta]}{\phi}_{\Hil_c} = \pair{\eta}{\phi}
	\end{equation*}
	for any \(\eta \in \calM(X)\) and \(\phi \in \Hil_c\) and likewise for \(H_\mu\). In particular $H_\mu$ can be defined for any input in $\Hil_\mu^*$, and takes values in $\Hil_\mu$, and likewise for \(H_c\).
	
	\begin{lemma}[{\cite[App.~B]{RGSD}}]
		\label{lemma:appendixBRGSD}
		The embedding $\Hil_c \to \Cont(X)$ is weak-to-norm continuous, and the map $H_c : \calM(X) \to \Hil_c$ is \textweakstar-to-norm continuous.  
	\end{lemma}
	
	The metric tensor $\g_\mu(\mudot,\mudot)$, defined in~\eqref{eq:expansion_Sinkhorn_def}, can be extended to a larger space: the completion of $\calM_0(X)$ with respect to \(\g_{\mu}\) is \(\Hil^*_{\mu,0}\) the set of linear forms on $\Hil_\mu$ which vanish on the function $\ones_X$ (see \cite[Sec.~4.2]{RGSD}). The latter \(\Hil^*_{\mu,0}\) can be thought of as some admissible tangent space to $\prm(X)$ at $\mu$.

	\paragraph{The metric introduced by the RKHS}
	We may not have a canonical metric on \(X\). However, any cost function \(c\) of the regularity we assume throughout induces its own metric \(d_c\) on the space \(X\) defined as \cite[Eq. (4.21)]{SteinwartChristmann2008} 
	\begin{equation}
		\label{eq:def_dc}
		d_c(x,y) \assign \norm{H_c[\delta_x] - H_c[\delta_y]}_{\Hil_c} = \sqrt{k_c(x,x) + k_c(y,y) - 2 k_c(x,y)}. 
	\end{equation}
	By~Lemma~\ref{lemma:appendixBRGSD}, \(\norm{H_c[\vdot]}_{\Hil_c}\) metrizes the \textweakstar~topology on \(\prm(X)\), so that $d_c$ metrizes the original topology on $X$.
	When \(c = d^2\) for a metric \(d\) on \(X\), then \(d_c = \sqrt{2}{\varepsilon} d + o(d) \) as $d \to 0$, so that the two metrics are asymptotically the same for small distances up to a multiplicative factor.

	\paragraph{The change of variables in a common Hilbert space}
	One difficulty when working with $\g_\mu$ is that the tangent space \(\Hil^*_{\mu,0}\) changes with $\mu$, making it difficult to describe the regularity of a path. We tackle this issue by using the following changes of variables.
	
	\begin{definition} \label{def:alpha_beta_variables}
		For any \(\mu \in \prm(X)\) denote
		\begin{equation*} 
			\A(\mu) \assign \alpha_\mu \assign \exp \left( \frac{1}{\varepsilon} \fmumu \right) \mu \in \calM(X), \quad
			\B(\mu) \assign \beta_\mu \assign H_c[\alpha_\mu] = \exp \left( - \frac{1}{\varepsilon} \fmumu \right) \in \Hil_c.
		\end{equation*}
	\end{definition} 
	
	\noindent The change of variables \(\mu \leftrightarrow \alpha_\mu\) already plays an important role in~\cite{Feydy2018}, and was a key tool in \cite{RGSD,Hardion2025}. By construction \(\alpha_\mu = \mu / \beta_\mu\).

	\begin{theorem}[{\cite[Thm.~4.8]{RGSD}}] \label{theorem:alpha_beta_homeo}
		The maps $A \colon \prm(X) \to \calM(X)$ and \(\B \colon \prm(X) \to \Hil_c\) are respectively \textweakstar-to-\textweakstar{}~and \textweakstar-to-weak homeomorphisms onto their images. Moreover, the image of $B$ is $\calB~\assign~\setgiven{\beta~\in~H_c[\calM_+(X)]}{\norm{\beta}_{\Hil_c} = 1}$. Weak convergence and norm convergence agree on \(\calB\), and the set $\calB$ is weakly and norm compact in $\Hil_c$.
	\end{theorem}
	
	Geometrically $\calB$ is the intersection between a convex cone $H_c[\calM_+(X)]$ and the unit sphere of $\Hil_c$. 
	The image of $A$ can also be characterized, but we will not need it in this work. 
	With this change of variables, a path \((\mu_t)_t\) in \(\prm(X)\) gives a corresponding path \((\beta_t)_t\) in \(\calB\), where we simplify the notation with \(\beta_t = \beta_{\mu_t}\).
	The map $B^{-1}$ pulls back $\g$ into a metric tensor on $\Hil_c$. We give it a specific name: for $\mu \in \prm(X)$ and $\betadot \in \Hil_c$ with $\pair{\betadot}{\beta_\mu}_{\Hil_c} = 0$, we define the quadratic form $\tg_\mu(\betadot, \betadot)$ by 
	\begin{equation} \label{eq:def:tgmu}
		\tg_\mu(\betadot, \betadot) = \frac{\varepsilon}{2} \left( \norm{\betadot}_{\Hil_c}^2 + 2 \bpair{\betadot / \beta_\mu}{ (\id - K_{\mu})\inv [\betadot / \beta_\mu] }_{L^2(X,\mu)} \right). 
	\end{equation}
	When \(\pair{\dot{\beta}}{\beta_\mu}_{\Hil_c} = 0\), we have \(\pair{\betadot/\beta_\mu}{\ones_X}_{L^2(X,\mu)} = \pair{\alpha_\mu}{\betadot} = \pair{\betadot}{\beta_\mu}_{\Hil_c} = 0\), making \(\tg_\mu\) well defined despite the inverse of \(\id - K_\mu\) only being defined in \(\Cont(X) / \R\). We state that $\tg$ indeed corresponds to the pull-back of $\g$ by $B^{-1}$.

	\begin{theorem}[{\cite[Thm.~4.12,~Lem.~4.13]{RGSD}}] \label{theorem:metric_tensor_in_betadot}
		Let \((\mu_t)_t\) be a path in \(\prm(X)\), such that the associated path \((\beta_t)_t = (B(\mu_t))_t\) is differentiable in $\Hil_c$ at $t_0$. Then $\pair{\betadot_{t_0}}{\beta_{t_0}}_{\Hil_c} = 0$, the path $(\mu_t)_t$ is differentiable at $t_0$ in \(\Hil_{\mu_{t_0}}^*\), and \(H_{\mu_{t_0}}[\mudot_{t_0}] = (\id + K_{\mu_{t_0}})[\betadot_{t_0}/\beta_{t_0}]\) as well as
		\begin{equation*} 
			\g_{\mu_{t_0}}(\mudot_{t_0},\mudot_{t_0}) = \tg_{\mu_{t_0}}(\betadot_{t_0}, \betadot_{t_0}).
		\end{equation*} 
	\end{theorem}
	
	Importantly, $\tg_{\mu}$ is defined as a quadratic form on a subspace of the common Hilbert space $\Hil_c$, making the analysis much easier when working with this variable.
	We recall a key estimate: in the variable $\beta$, the metric tensor is equivalent, up to constants, to the flat metric tensor on $\Hil_c$. Specifically, \cite[Prop.~4.15]{RGSD} states that there exists some \(C > 0\) only depending on \(\norm{c}_\infty\) and \(\varepsilon\), such that for all $\mu \in \prm(X)$ and all $\betadot$ with $\langle \betadot,  \beta_\mu \rangle_{\Hil_c} = 0$, 
	\begin{equation} \label{eq:tg_equiv_Hc}
		\frac{\varepsilon}{2} \norm{\betadot}_{\Hil_c}^2 \leq \tg_{\mu}(\betadot,\betadot) \leq C \frac{\varepsilon}{2}  \norm{\betadot}_{\Hil_c}^2.
	\end{equation}

	\paragraph{The metric $\dS$.}
	To define the Sinkhorn metric, we need to specify the space of admissible curves. In the variable $\beta$, it corresponds to a Hilbert-valued Sobolev space. We recall the basics, see e.g.~\cite{Kreuter2015}. We write \(L^2((0,1);\Hil_c)\) for the space of measurable functions \(u \colon (0,1) \to \Hil_c\) for which \(\norm{u(\vdot)}_{\Hil_c}\) is 2-integrable, up to equality almost everywhere. It carries the obvious inner product. We say that a function \(u \in L^2((0,1);\Hil_c)\) has the distributional derivative \(v \in L^2((0,1);\Hil_c)\) and write \(v = u'\) if for all \(\phi \in \Contc^\infty((0,1);\R)\) it holds $ \int_0^1 u(t) \phi'(t) \diff t = - \int_0^1 v(t) \phi(t) \diff t$ as Bochner integrals. The first Sobolev space is given by 
	\begin{equation*}
		\Hil^1((0,1);\Hil_c) = \setgiven{u \colon (0,1) \to \Hil_c}{u \textup{ has a distr.~derivative } u' \textup{ and } u, u' \in L^2((0,1);\Hil_c)}.
	\end{equation*}
	It is a Hilbert space with the inner product \(\pair{u_1}{u_2}_{\Hil^1} = \pair{u_1}{u_2}_{L^2} + \pair{u_1'}{u_2'}_{L^2}\).

	\begin{definition}[{\cite[Def.~5.1]{RGSD}}] \hfill \label{def:distance}
		\begin{enumerate}
			\item 
			A path \((\mu_t)_{t \in [0,1]}\) is admissible when the corresponding path \((\beta_t)_{t \in (0,1)} = (\B(\mu_t))_{t \in (0,1)}\) belongs to the Sobolev space \(\Hil^1((0,1);\Hil_c)\). We denote the set of admissible paths as \(\Pall\) and write \(\Pall(\mu, \nu)\) for the subset of \(\Pall\) with fixed start and end points \(\mu, \nu \in \prm(X)\).
			\item 
			The energy of a path \((\mu_t)_t \in \Pall\) is defined as
			\begin{equation*}
				E((\mu_t)_t) \assign \int_0^1 \tg_{\mu_t}(\betadot_t,\betadot_t) \diff t = \int_0^1 \g_{\mu_t}(\mudot_t,\mudot_t) \diff t.
			\end{equation*}    
			\item 
			For \(\mu,\nu \in \prm(X)\) we define
			\begin{equation} \label{eq:def_dS}
				\dS(\mu,\nu) \assign \left( \inf_{(\mu_t)_t\in\Pall(\mu,\nu)} E((\mu_t)_t) \right)^\frac{1}{2}.
			\end{equation}
		\end{enumerate}
	\end{definition}
	
	This is nothing but mimicking the classical definition of Riemannian distance on a manifold. In the second point, the two integral expressions coincide thanks to Theorem~\ref{theorem:metric_tensor_in_betadot}.

	\begin{theorem}[{\cite[Thm.~5.2 \& Thm.~5.4]{RGSD}}] \label{theorem:dS_metric}
		The function \(\dS\) is a metric on \(\prm(X)\) that metrizes the \textweakstar~topology. The variational problem~\eqref{eq:def_dS} defining \(\dS\) has a minimizer. 
		Further, for all \(\mu, \nu \in \prm(X)\), there holds, for \(C > 0\) as in~\eqref{eq:tg_equiv_Hc}, 
		\begin{equation*}
			\sqrt{\frac{\varepsilon}{2}} \, \norm{ \beta_\mu - \beta_\nu }_{\Hil_c} \leq \dS(\mu,\nu) \leq \frac{\pi}{2} \sqrt{\frac{\varepsilon}{2} C} \, \norm{ \beta_\mu - \beta_\nu }_{\Hil_c}.
		\end{equation*}
	\end{theorem}

	\noindent We call minimizers in~\eqref{eq:def_dS} \emph{Riemannian geodesics} for $\dS$. The proof of existence of geodesics relies on some coercivity of the functional $E$, together with its lower semicontinuity. We state explicitly the first result, which is an easy consequence of Sobolev embeddings. 
	
	\begin{theorem}[{Proof of Thm.~5.4 in \cite{RGSD}, Thm~B.2 in \cite{Hardion2025}}]
		\label{theorem:coercivity}
		Assume $(\mu^n_t)_t$ is a sequence of curves in $\Pall$, and call $(\beta^n_t)_t = (B(\mu^n_t)_t)$. If $\sup_n E((\mu^n_t)_t) < \infty$ (or equivalently $\sup_n \| \betadot^n_t \|_{L^2((0,1);\Hil_c)} < \infty$ given~\eqref{eq:tg_equiv_Hc}), then up to extraction of a subsequence $(\beta^n_t)_t$ converges to a limit $(\beta_t)_t$ strongly in $\Cont([0,1];\Hil_c)$ and weakly in $\Hil^1((0,1);\Hil_c)$. As a consequence, $(\betadot^n_t)_t$ converges weakly to $(\betadot_t)_t$ in $L^2((0,1);\Hil_c)$, and $(\mu^n_t)_t$ converges in $\Cont([0,1],(\prm(X),\dS))$ to $(B\inv(\beta_t))_t$.     
	\end{theorem}

	Eventually we comment on a few canonical paths in our geometry that we will use frequently.
	We will also use the following paths frequently throughout this paper.
	
	\begin{definition} \label{def:vert_arc}
		Given \(\mu, \nu \in \prm(X)\) we define
		\begin{itemize}
			\item the vertical interpolation by \(\mu^V_t \assign (1-t) \mu + t \nu\),
			\item the arc interpolation \(\mu^A_t\) as \(\B\inv(\beta^A_t)\), where \((\beta^A_t)_t\) is the arc connecting \(\beta_\mu\) and \(\beta_\nu\) on the unit sphere of \(\Hil_c\): elementary trigonometry yields, with $\cos(\theta) = \langle \beta_\mu, \beta_\nu \rangle$,
			\begin{equation}
				\label{eq:explicit_formula_arc}
				\beta^A_t = \frac{\sin((1-t) \theta)}{\sin(\theta)} \beta_\mu + \frac{\sin(t \theta)}{\sin(\theta)} \beta_\nu. 
			\end{equation}
		\end{itemize}
	\end{definition}
	
	\noindent On $\R^d$ we can also consider horizontal motion, in particular constant speed translations of measures, i.e.~\(\mu_t = (\id + t v)_\# \mu\) for \(\mu \in \prm(X)\) and some fixed \(v \in \R^d\). For the squared distance cost function, these curves are geodesics of \(\dS\) \cite[Thm.~5.16]{RGSD}.

	\section{Gromov--Hausdorff convergence of spatial dis\-cret\-ization} \label{sec:Gromov_Hausdorff}

	Any $Z \subseteq X$ finite subset of our base space comes with its own intrinsic metric \(\dSM{Z}\) on \(\prm(Z) \subseteq \prm(X)\). We can either see it by applying all our theory with $Z$ in place of $X$, Theorem~\ref{theorem:dS_metric} shows that $(\prm(Z), \dSM{Z})$ is always a geodesic space. Or we can see $\prm(Z)$ as a subset of $\prm(X)$, and in Definition~\ref{def:distance} restrict the space of admissible paths to \( \setgiven{(\mu_t)_t \in \Pall(\mu,\nu)}{\forall t \colon \mu_t \in \prm(Z)}\). These two constructions are equivalent: this is because the values of $f_{\mu,\nu}$ and $g_{\mu,\nu}$ on $\supp(\mu)$ and $\supp(\nu)$ respectively do not depend on the ambient space $X$, only on the restriction of $c$ to $\supp(\mu) \cup \supp(\nu)$. Thus, if $\mu \in \prm(Z) \subseteq \prm(X)$ and $\mudot \in \calM_0(Z) \subseteq \calM_0(X)$, then the value of $\g_\mu(\mudot,\mudot)$ only depends on the restriction of $c$ to $Z$.

	The main result of this section is that finer and finer discretizations $Z$ of \(X\) yield convergence of $\dSM{Z}$ to \(\dS\). 
	Specifically, throughout this section, let \((x_m)_{m}\) be a sequence that is dense in \(X\), and for \(M \in \N\) set \(X_M \assign \{x_1, x_2, \dots, x_M\}\). Denoting \(\dSM{M} \assign \dSM{X_M}\), we will show the convergence of the metric space $(\prm(X_M), \dSM{M})$ to $(\prm(X),\dS)$. 
	Specifically we build explicitly a surjective map $R_M$ from $\prm(X)$ onto $\prm(X_M)$ which almost preserves distances. Given our geometry is almost Hilbertian once we look at $\beta_\mu = B(\mu) \in \Hil_c$, we build the map $R_M$ in the geometry of $\Hil_c$. Recalling that $\alpha_\mu = A(\mu)$ as in Definition~\ref{def:alpha_beta_variables} shares the same support as $\mu$, an element $\mu \in \prm(X)$ belongs to $\prm(X_M)$ if and only if $\alpha_\mu \in \mathcal{M}_+(X_M)$, where $\mathcal{M}_+(X_M)$ is the set of non-negative measures on $X_M$. 
	Applying the operator $B$, we obtain that $\mu \in \prm(X_M)$ if and only if
	\begin{equation*}
		\beta_\mu = H_c[\alpha_\mu] \in H_c[\calM_+(X_M)] = \vecspan_{\R_+} \setgiven{k_c(x,\vdot)}{x \in X_M} \assignRe \cone^M.
	\end{equation*}
	This set \(\cone^M\) is a convex cone of finite dimension, hence closed in $\Hil_c$. The orthogonal projection $P_M$ onto \(\cone^M\) is well behaved, however we are not guaranteed that $P_M(\beta_\mu)$ belongs to $\calB$, as the latter also encodes a constraint of unit norm. We simply normalize $P_M(\beta_\mu)$, to keep us in $\cone^M$. To recover a probability measure in $\prm(X_M)$, we apply the inverse change of variables \(\B\inv\) to the normalization of $P_M(\beta_\mu)$. It is summarized in the following definition.

	\begin{definition}
		\label{def:cone_projection}
		With $P_M : \Hil_c \to \Hil_c$ the orthogonal (metric) projection onto $\cone^M \assign H_c[\calM_+(X_M)]$, we define $R_M : \prm(X) \to \prm(X_M)$ by 
		\begin{equation*}
			R_M(\mu) = \B\inv \left( \frac{P_M(\B(\mu))}{\norm{P_M(\B(\mu))}_{\Hil_c}} \right).    
		\end{equation*}
	\end{definition}
	
	\noindent By construction, as $B(\mu) \in \cone^M$ if $\mu \in \prm(X_M)$, we have $R_M(\mu) = \mu$ for $\mu \in \prm(X_M)$. Our main result is the following.

	\begin{theorem}
		\label{theorem:unif_conv_RM}
		Uniformly for all $\mu, \nu \in \prm(X)$, there holds
		\begin{equation*}
			\lim_{M \to \infty} \dS(R_M(\mu),\mu) = 0, \qquad \lim_{M\to \infty} \dSM{M}(R_M(\mu), R_M(\nu)) = \dS(\mu,\nu). 
		\end{equation*}
	\end{theorem}
	
	We recall that the Gromov--Hausdorff distance between two metric spaces \((\calX,d_\calX)\) and \((\calY,d_\calY)\) is
	\begin{equation*}
		d_{\mathrm{GH}}((\calX,d_\calX), (\calY,d_\calY)) \assign \frac{1}{2} \inf_{\calR} \sup_{\substack{(x,y) \in \calR \\ (x',y') \in \calR}} \, \abs{d_\calX(x,x') - d_\calY(y,y')},
	\end{equation*}
	where the infimum ranges over all subsets \(\calR \subseteq \calX \times \calY\) with surjective projections \(\proj_1(\calR) = \calX\), \(\proj_2(\calR) = \calY\) \cite[Sec.~7.3]{Burago2001}. 
	As $R_M$ is surjective, using for $\mathcal{R}$ the graph of $R_M$, Theorem~\ref{theorem:unif_conv_RM} shows
	\begin{equation*}
		\lim_{M \to \infty} d_{\mathrm{GH}}((\prm(X),\dS), (\prm(X_M),\dSM{M})) = 0,  
	\end{equation*}
	hence the Gromov-Hausdorff convergence of $(\prm(X_M),\dSM{M})$ to $(\prm(X),\dS)$.
	An easy implication is the convergence of the geodesics in $X_M$ to the ones in $X$, at least along subsequences.  
	
	\begin{corollary}
		\label{cor:GH_convergence_geodesics}
		For any $M \geq 1$, assume that we have \((\mu^M_t)_t\) in $\Pall$ taking values in $\prm(X_M)$, and that $\mu^M_0$, $\mu^M_1$ converge respectively to $\mu$, $\nu$ as $M \to \infty$. We also assume that \((\mu^M_t)_t\) is asymptotically geodesic, in the sense \(E((\mu^M_t)_t) - \dSM{M}(\mu^M_0, \mu^M_1)^2 \to 0\) as \(M \to \infty\). Then, up to extraction of a subsequence, the curve \((\mu^M_t)_t\) converges in $\Cont([0,1]; (\prm(X),\dS))$ to a curve $(\mu_t)_{t}$ which is a geodesic between $\mu$ and $\nu$, and it holds
		\begin{equation*}
			\dS^2(\mu,\nu) = E((\mu_t)_t) = \lim_{M \to \infty} E((\mu^M_t)_t).
		\end{equation*}
	\end{corollary}

	The rest of this section is dedicated to the proofs of Theorem~\ref{theorem:unif_conv_RM} and Corollary~\ref{cor:GH_convergence_geodesics}. We first prove some additional result on the energy functional, study carefully the operator $R_M$, and eventually move to the actual proofs.
	
	\subsection{Regularity of the energy functional} 

Due to the identification of admissible paths \((\mu_t)_t \in \Pall\) with their corresponding paths \((\beta_t)_t \in \Hil^1((0,1);\Hil_c)\), we can naturally let \(\Pall\) inherit the weak or norm topology of \(\Hil^1((0,1);\Hil_c)\). In our previous work we already proved lower semi-continuity of the energy.

\begin{theorem}[{\cite[Proof of Thm.~4.5]{RGSD}}]
	\label{theorem:Energy_lsc}
	The energy \(E\) is lower semi-continuous in the weak topology of \(\Hil^1((0,1);\Hil_c)\).
\end{theorem}

\noindent We now prove the following related (and easier) result.

\begin{proposition} \label{prop:Energy_cont}
	The energy \(E\) is continuous in the norm topology of \(\Hil^1((0,1);\Hil_c)\).
\end{proposition}

\begin{proof}
	Take a sequence $(\mu^n_t)_t$ such that the associated curve $(\beta^n_t)_t$ converges strongly in $\Hil^1((0,1);\Hil_c)$ to a limit curve $(\beta_t)_t = (B(\mu_t))_t$. In particular $(\betadot^n_t)_t$ converges strongly in $L^2((0,1);\Hil_c)$ to $(\betadot_t)_t$. So up to subsequence, convergence of $\betadot^n_t$ to $\betadot_t$ holds in $\Hil_c$ for a.e. $t$. From \cite[Proposition 4.14]{RGSD}, it implies that, for a.e. $t$, we have $\tg_{\mu^n_t}(\betadot^n_t, \betadot^n_t) \to \tg_{\mu_t}(\betadot_t, \betadot_t)$. In addition $\tg_{\mu^n_t}(\betadot^n_t, \betadot^n_t)$ is upper bounded up to a multiplicative constant by $\| \betadot^n_t \|_{\Hil_c}^2$, and the latter is uniformly integrable from the convergence of $(\betadot^n_t)_t$ in $L^2((0,1);\Hil_c)$. Thus we can switch limit and integration and obtain 
	\begin{equation*}
		E((\mu^n_t)_t) = \int_0^1 \tg_{\mu^n_t}(\betadot^n_t, \betadot^n_t) \diff t \to \int_0^1 \tg_{\mu_t}(\betadot_t, \betadot_t) \diff t = E((\mu_t)_t). 
	\end{equation*}
	As the limit does not depend on the subsequence extracted from $(\betadot^n_t)_t$, a standard argument shows that this holds along the whole sequence $(\mu^n_t)_t$. 
\end{proof}

\subsection{Projection to a finite base space}\label{sec:projection}

In this subsection we study in detail the operator $R_M$ as introduced in Definition~\ref{def:cone_projection}, which not only serves as the correspondence between \(\prm(X)\) and \(\prm(X_M)\) to prove Gromov--Hausdorff convergence, but also plays a crucial role later for proving the limsup condition for \(\Gamma\)-convergence. To simplify notations, recalling that $P_M$ is the projection on the cone $\cone^M$, we also introduce $\tilde{R}_M : \B \to  \cone^M \cap \calB$ by  
\begin{equation*}
	\tilde{R}_M(\beta) = \frac{P_M(\beta)}{\| P_M(\beta) \|_{\Hil_c}},
\end{equation*}
such that $R_M = B\inv \circ \tilde{R}_M \circ B$. The map $\tilde{R}_M$ actually corresponds to the metric projection onto $\calB$ \cite[Thm.~8.1]{bauschke2018projecting}.  
With $d_c$ as in~\eqref{eq:def_dc}, we define
\begin{equation} \label{eq:dM}
	\delta_M = d_c(X,X_M) = \max_{x \in X} \min_{y \in X_M} d_c(x,y).
\end{equation}
It is \(\lim_{M \to \infty} \delta_M = 0\) by density of the sequence \((x_m)_{m}\).

\begin{lemma} \label{lemma:PM_properties}
	The projection \(P_M\)  satisfies
	\begin{enumerate}
		\item \label{item:PM_contraction} \(\norm{P_M(\beta) - P_M(\beta')}_{\Hil_c} \leq \norm{\beta - \beta'}_{\Hil_c}\) for all \(\beta, \beta' \in \Hil_c\),
		\item \label{item:PM_orthogonal} \(\pair{P_M(\beta)}{\beta - P_M(\beta)}_{\Hil_c} = 0\) for all \(\beta \in \Hil_c\),
		\item \label{item:PM_orthogonal_decomp} \(\norm{\beta}_{\Hil_c}^2 = \norm{P_M(\beta)}_{\Hil_c}^2 + \norm{\beta - P_M(\beta)}_{\Hil_c}^2\) for all \(\beta \in \Hil_c\),
		\item \label{item:PM_bound_1} \(\norm{P_M(\beta)}_{\Hil_c} > 0\) uniformly for \(\beta \in \calB\) and \(M \in \N\). 
		\item \label{item:PM_bound_dM} \(\norm{\beta - P_M(\beta)}_{\Hil_c} \leq C \cdot \delta_M\) for all \(\beta \in \calB\) with $\delta_M$ as in \eqref{eq:dM}, and some \(C > 0\) 
		independent of \(M\). In particular, for the same constant $C$,  \(\norm{\beta - \tilde{R}_M(\beta)}_{\Hil_c} \leq 2 C \cdot \delta_M\) for all \(\beta \in \calB\).
	\end{enumerate}
\end{lemma}

\begin{proof}
	\emph{Points 1 to 3.} They are valid for any orthogonal projection onto a convex cone in Hilbert space, see e.g. \cite[Lem.~2]{Ingram1991}.
	
	\emph{Point 4.} 
	With Point~\ref{item:PM_orthogonal_decomp} it is \(\norm{P_M(\beta)}_{\Hil_c} \geq \norm{P_1(\beta)}_{\Hil_c} = \pair{\beta}{k_c(x_1,\vdot)}_{\Hil_c} = \beta(x_1) \geq \exp(-\frac{3}{2} \norm{c}_\infty / \varepsilon)\), where the last inequality is~\eqref{eq:sup_bound_potentials}.
	
	\emph{Point 5.} For \(x \in X\) let \(\proj_{X_M}(x) \in \argmin_{y \in X_M} d_c(x,y)\) be (a measurable selection of) the nearest neighbour projection. Then
	\begin{equation*}
		\norm{k_c(x,\vdot) - k_c(\proj_{X_M}(x),\vdot)}_{\Hil_c} = d_c(x,\proj_{X_M}(x)) \leq \delta_M.
	\end{equation*}
	For any \(\beta=\beta_\mu = H_c[\alpha_\mu] \in \calB \subseteq \Hil_c\), 
	the candidate \(\alpha_\mu^M = (\proj_{X_M})_\# \alpha_\mu \in \calM_+(X_M)\) yields the bound 
	\begin{align*}
		\norm{\beta_\mu - P_M(\beta_\mu)}_{\Hil_c} & \leq
		\norm{\beta_\mu - H_c[\alpha_\mu^M]}_{\Hil_c} \\
		& = \bnorm{\int_X \left( k_c(x,\vdot) - k_c(\proj_{X_M}(x),\vdot) \right) \diff \alpha_\mu(x)}_{\Hil_c} 
		\leq \delta_M \cdot \norm{\alpha_\mu}_\TV.
	\end{align*}
	This gives the claim, as \(\norm{\alpha_\mu}_\TV \leq \norm{\exp(\fmumu/\varepsilon)}_\infty \leq \exp(\frac{3}{2} \norm{c}_\infty / \varepsilon)\) by~\eqref{eq:sup_bound_potentials}. The claim for $\tilde{R}_M$ follows from
	\begin{align*}
		\| \tilde{R}_M(\beta) - \beta \|^2_{\Hil_c} &= 2 \left( 1 - \frac{\langle \beta, P_M(\beta) \rangle_{\Hil_c}}{\| P_M(\beta) \|_{\Hil_c} } \right) \\
		&=  2 (1 - \| P_M(\beta) \|_{\Hil_c}) \leq 2 (1 - \norm{P_M(\beta)}_{\Hil_c}^2) = 2 \| \beta - P_M(\beta) \|_{\Hil_c}^2
	\end{align*}
	with the second equality coming from Point \ref{item:PM_orthogonal} and the last one from Point~\ref{item:PM_orthogonal_decomp}.   
\end{proof}

\begin{proposition} \label{prop:proj_path_conv_H1}
	If \((\beta_t)_t \in \Hil^1((0,1);\calB)\) is given, then $(\tilde{R}_M(\beta_t))_t$ converges in norm to $(\beta_t)_t$ in \(\Hil^1((0,1);\Hil_c)\) as $M \to \infty$.  
\end{proposition}

\begin{proof}
	We write $\beta^M_t = \tilde{R}_M(\beta_t)$.
	With Point~\ref{item:PM_bound_dM} of Lemma~\ref{lemma:PM_properties} we obtain convergence of \(\beta^M_t \) to \( \beta_t\) for all $t$, hence norm convergence in \(L^2((0,1);\Hil_c)\) as all curves take values in the unit sphere. 
	
	Next we prove that \((\beta^M_t)_t \in \Hil^1((0,1);\Hil_c)\):
	As $P_M$ is a contraction, we obtain that \((P_M(\beta_t))_t \in \Hil^1((0,1);\Hil_c)\) for any \(M \in \N\). We can even say more: 
	by Point~\ref{item:PM_bound_1} of Lemma~\ref{lemma:PM_properties} we can bound \(\norm{P_M(\beta_t)}_{\Hil_c} \geq C > 0\) (uniformly in \(t\)). Using the Lipschitz function \(F \colon \Hil_c \to \Hil_c,~ F(\phi) = \min \{ 1/C , 1/\norm{\phi}_{\Hil_c} \} \cdot \phi\) in~\cite[Cor.~3.14]{Kreuter2015} yields that \((\beta^M_t)_t \in \Hil^1((0,1);\Hil_c)\). In particular, we also have \((\norm{P_M(\beta_t)}_{\Hil_c})_t \in \Hil^1((0,1);\R)\). 
	
	We first claim that we have 
	\begin{equation}
		\label{eq:to_prove_convergence_curve}
		\limsup_{M \to \infty} \int_0^1 \norm{ \betadot^M_t }_{\Hil_c}^2 \diff t \leq \int_0^1 \norm{ \betadot_t }_{\Hil_c}^2 \diff t,
	\end{equation}
	that will be enough to yield the conclusion. We compute with the chain rule
	\begin{equation} 
		\label{eq:betaMdot}
		\betadot^M_t = \frac{ \frac{\diff}{\diff t} P_M(\beta_t)}{\norm{P_M(\beta_t)}_{\Hil_c}} -  \langle \tfrac{\diff}{\diff t} P_M(\beta_t), P_M(\beta_t)  \rangle_{\Hil_c} \frac{P_M(\beta_t)}{\norm{P_M(\beta_t)}^3}.
	\end{equation}
	Since \(P_M\) is a contraction, we must have \(\norm{\frac{\diff}{\diff t} P_M(\beta_t)}_{\Hil_c} \leq \norm{\betadot_t}_{\Hil_c}\). On the other hand, we have $\langle \dd{}{t} P_M(\beta_t), P_M(\beta_t) - \beta_t \rangle_{\Hil_c} + \langle  P_M(\beta_t), \dd{}{t} P_M(\beta_t) - \betadot_t \rangle_{\Hil_c} = 0$, obtained by  differentiating Point~\ref{item:PM_orthogonal} of Lemma~\ref{lemma:PM_properties} in time. As moreover $\langle \betadot_t, \beta_t \rangle_{\Hil_c} = 0$ whenever $\betadot_t$ exists, using again that $P_M$ is a contraction,
	\begin{align*}
		\left| \langle \tfrac{\diff}{\diff t} P_M(\beta_t), P_M(\beta_t)  \rangle_{\Hil_c} \right|  & = \left| \langle \tfrac{\diff}{\diff t} P_M(\beta_t) - \betadot_t, \beta_t - P_M(\beta_t) \rangle_{\Hil_c} \right| \\
        & \leq \left( \| \tfrac{\diff}{\diff t} P_M(\beta_t)  \|_{\Hil_c} + \| \betadot_t \|_{\Hil_c} \right) \|  \beta_t - P_M(\beta_t) \|_{\Hil_c} \leq 2 C \delta_M \| \betadot_t \|_{\Hil_c}    
	\end{align*}
	with \(C\) from Point~\ref{item:PM_bound_dM} of Lemma~\ref{lemma:PM_properties}. The same point shows that \(1 \geq \norm{P_M(\beta_t)}_{\Hil_c} \geq 1 - C \delta_M \to 1\) as \(M \to \infty\), uniformly in \(t\). From~\eqref{eq:betaMdot} we deduce 
	\begin{equation}
		\label{eq:bound_betadot_nonasymp}
		\| \betadot^M_t \|_{\Hil_c} \leq \left( \frac{1}{1 - C \delta_M} + \frac{2 C \delta_M}{(1 - C \delta_M)^2} \right) \| \betadot_t \|_{\Hil_c}. 
	\end{equation}
	The limit of the multiplicative prefactor is $1$ as $\delta_M \to 0$. As $(\betadot_t)_t$ is in $L^2((0,1);\Hil_c)$, we can use Fatou's lemma for limsup and obtain~\eqref{eq:to_prove_convergence_curve}.

	We now conclude using~\eqref{eq:to_prove_convergence_curve}. We have that \((\beta^M_t)_t \to (\beta_t)_t\) strongly in $L^2((0,1);\Hil_c)$. Moreover, from~\eqref{eq:to_prove_convergence_curve}, we see that \((\beta^M_t)_t \) is bounded in $\Hil^1((0,1);\Hil_c)$. By the Banach-Alaoglu theorem and the $L^2$ convergence, we deduce that the convergence of \((\beta^M_t)_t \) to \( (\beta_t)_t\) holds weakly in $\Hil^1((0,1);\Hil_c)$. Combined with the strong $L^2$ convergence and~\eqref{eq:to_prove_convergence_curve}, we have $\limsup_{M \to \infty} \| (\beta^M_t)_t \|_{\Hil^1((0,1);\Hil_c)} \leq \| (\beta_t)_t \|_{\Hil^1((0,1);\Hil_c)}$,
	which is enough to upgrade weak convergence to strong convergence \cite[Prop.~3.32]{Brezis2011}. 
\end{proof}

\begin{corollary}
	\label{crl:cv_energy_M}
	If $(\mu_t)_t \in \Pall$, then $E((R_M(\mu_t))_t) \to E((\mu_t)_t)$ as $M \to \infty$.    
\end{corollary}

\begin{proof}
	It follows from the convergence in $\Hil^1((0,1);\Hil_c)$ of Proposition~\ref{prop:proj_path_conv_H1} and the continuity of $E$ in Proposition~\ref{prop:Energy_cont}.
\end{proof}

\subsection{Proof of Gromov--Hausdorff convergence}

In these proofs, to simplify notations we write $A \lesssim B$ if there exists a constant $C > 0$ depending only on $\| c \|_\infty$, $X$ and $\varepsilon$ (but not on $\mu$ nor $M$) such that $A \leq C B$. We start with comparisons up to multiplicative constants.

\begin{lemma}
	\label{lemma:comp_ds_dsm}
	There holds, for all $\mu, \nu \in \prm(X_M)$,
	\begin{equation*}
		\dS(\mu,\nu) \leq \dSM{M}(\mu,\nu) \lesssim \dS(\mu,\nu). 
	\end{equation*}
\end{lemma}

\begin{proof}
	The first inequality is a direct consequence of the definition of $\dSM{M}$. For the second one, we use that both $\dS(\mu,\nu)$ and $\dSM{M}(\mu,\nu)$ are comparable up to constants which do not depend on $M$ to $\| \beta_\mu - \beta_\nu \|_{\Hil_c}$ (see Theorem~\ref{theorem:dS_metric}), and the square of the latter is $2 - 2 \langle \alpha_\mu \otimes \alpha_\nu, \exp(-c) \rangle$, thus is identical whether we use $X$ or $X_M$ as base space.
\end{proof}

\begin{lemma}
	\label{lemma:comp_dsmrm_ds}
	There holds, for all $\mu,\nu \in \prm(X)$, 
	\begin{equation*}
		\dSM{M}(R_M(\mu), R_M(\nu)) \lesssim \dS(\mu,\nu).
	\end{equation*}
\end{lemma}

\begin{proof}
	Consider $(\mu_t)_t$ the geodesic in $\prm(X)$ joining $\mu$ to $\nu$. Then $(R_M(\mu_t))_t$ is an admissible curve joining $R_M(\mu)$ to $R_M(\nu)$. Thus from the estimate~\eqref{eq:bound_betadot_nonasymp} and the comparison~\eqref{eq:tg_equiv_Hc}, 
	\begin{equation*}
		\dSM{M}(R_M(\mu),R_M(\nu))^2 \leq E((R_M(\mu_t))_t)  \lesssim \| (\betadot^M_t)_t  \|^2_{L^2} 
		\lesssim \| (\betadot_t)_t  \|^2_{L^2} \lesssim \dS^2(\mu,\nu). \qedhere
	\end{equation*}
\end{proof}

\noindent We are now ready to prove the main results of this section.

\begin{proof}[{Proof of Theorem~\ref{theorem:unif_conv_RM}}]
	For the first claim, we start from the estimate in Theorem~\ref{theorem:dS_metric} combined with Point~\ref{item:PM_bound_dM} in Lemma~\ref{lemma:PM_properties}:
	\begin{equation}
		\label{eq:precise_estimate_RM}
		\dS(\mu,R_M(\mu)) \lesssim \| B(\mu) - \tilde{R}_M(B(\mu)) \|_{\Hil_c} \lesssim \delta_M
	\end{equation}
	The conclusion $\dS(\mu,R_M(\mu)) \to 0$ uniformly follows as $\delta_M \to 0$.

	Then we move to the second estimate. We first prove pointwise convergence. Take $\mu, \nu \in \prm(X)$, and $(\mu_t)_t$ a constant-speed geodesic between $\mu$ and $\nu$. We want to estimate the distance between $R_M(\mu)$ and $R_M(\nu)$, and for that we use $(R_M(\mu_t))_t$ as an admissible curve joining $R_M(\mu)$ to $R_M(\nu)$. Using Corollary~\ref{crl:cv_energy_M}, as $M \to \infty$, 
	\begin{equation*}
		\dSM{M}(R_M(\mu),R_M(\nu))^2 \leq E((R_M(\mu_t))_t) \to E((\mu_t)_t) = \dS^2(\mu,\nu).
	\end{equation*}
	Combined with the trivial upper bound $\dSM{M}(R_M(\mu),R_M(\nu)) \geq \dS(R_M(\mu),R_M(\nu))$ (see Lemma~\ref{lemma:comp_ds_dsm}), and as $R_M(\mu) \to \mu$ and $R_M(\nu) \to \nu$ by the first point, we obtain the other inequality asymptotically. Thus $\dSM{M}(R_M(\mu),R_M(\nu)) \to \dS(\mu,\nu)$ for any $\mu, \nu$. 
	
	It remains to upgrade this pointwise convergence to uniform convergence. For this we remark with the triangle inequality and Lemma~\ref{lemma:comp_dsmrm_ds} that
	\begin{equation*}
		|\dSM{M}(R_M(\mu), R_M(\nu)) - \dSM{M}(R_M(\mu'), R_M(\nu'))| \lesssim \dS(\mu,\mu') + \dS(\nu,\nu').
	\end{equation*}
	With $h_M(\mu,\nu) \assign \dSM{M}(R_M(\mu), R_M(\nu))$, it proves that $(h_M)_M$ is equi-continuous over the compact set $\prm(X)^2$. As the functions $h_M$ take values in a bounded subset of $\R$ independent of $M$ (again with Lemma~\ref{lemma:comp_dsmrm_ds}), by the Arzelà–Ascoli theorem $(h_M)_M$ converges uniformly to a limit function $h$, up to extraction. As we already know that $(h_M)_M$ converges pointwise to $h(\mu,\nu) \colon (\mu,\nu) \mapsto \dS(\mu,\nu)$, we can conclude that the convergence of $(h_M)_M$ to $h$ holds uniformly, and without the need to extract a subsequence. 
\end{proof}

\begin{proof}[{Proof of Corollary~\ref{cor:GH_convergence_geodesics}}]
	As $E((\mu^M_t)_t)$ is uniformly bounded, combining Theorem~\ref{theorem:coercivity} and Theorem~\ref{theorem:Energy_lsc} (that is, following the proof of \cite[Thm.~5.4]{RGSD}) we obtain the existence of a limit curve $(\mu_t)_{t}$, with convergence in $\Cont([0,1]; (\prm(X),\dS))$, such that $E((\mu_t)_t) \leq \liminf_{M \to \infty} E((\mu^M_t)_t)$. In particular, 
	\begin{equation}
		\label{eq:aux_to_cite}
		\dS^2(\mu,\nu) \leq E((\mu_t)_t) \leq \liminf_{M \to \infty} \, E((\mu^M_t)_t) =\liminf_{M \to \infty} \, \dSM{M}(\mu^M_0, \mu^M_1)^2.
	\end{equation}
	On the other hand, with the triangle inequality, 
	\begin{equation*}
		\dSM{M}(\mu^M_0, \mu^M_1) \leq \dSM{M}(\mu^M_0, R_M(\mu)) + \dSM{M}(R_M(\mu), R_M(\nu)) + \dSM{M}(R_M(\nu), \mu^M_1).
	\end{equation*}
	The second summand converges to $\dS(\mu,\nu)$ by Theorem~\ref{theorem:unif_conv_RM}, while the first and last summands converge to $0$ given Lemma~\ref{lemma:comp_ds_dsm} and the first part of Theorem~\ref{theorem:unif_conv_RM}. Thus $\liminf_M \dSM{M}(\mu^M_0, \mu^M_1)^2 \leq \dS(\mu, \nu)^2$, all inequalities in~\eqref{eq:aux_to_cite} are equalities.   
\end{proof}

\section{Local expansion of the Sinkhorn divergence and its geodesic distance}
\label{sec:local_expansion}

As already explained in the introduction, the goal of this section is to quantify the approximations~\eqref{eq:to_prove_approx}, which we recall here:
\begin{equation*}
	S_\varepsilon(\mu,\nu) \approx \g_\mu(\nu - \mu, \nu - \mu) \approx \tg_\mu(\beta_\nu - \beta_\mu, \beta_\nu - \beta_\mu) \approx \dS^2(\mu,\nu) \qquad \text{ as } \nu \to \mu.    
\end{equation*}
Specifically we want to substitute the ``$\approx$'s'' with small $o$'s which are uniform over the whole space of probability distributions.
Note that we already face a small issue: $\tg_\mu$ is a priori only defined for vectors $\betadot$ such that $\langle \betadot, \beta_\mu \rangle_{\Hil_c} = 0$, so that we cannot directly plug $\beta_\nu - \beta_\mu$ as an argument. We first extend the definition of $\tg_\mu$ as a quadratic form on the whole space $\Hil_c$ by projecting on a relevant subspace. Let us introduce, for $\phi \in \Cont(X)$, $Q_\mu[\phi] = \phi - \langle \mu,\phi \rangle$.
Clearly $Q_\mu$ maps $\Cont(X)$ into $\Cont(X)$, as moreover $Q_\mu[\ones_X] = 0$, it descends $Q_\mu$ into an operator from $\Cont(X)/\R$ to $\Cont(X)$. 

\begin{definition}
	\label{def:tg_extended}
	If $\mu \in \prm(X)$ and $\betadot \in \Hil_c$, then we define 
	\begin{equation}
		\label{eq:def_tilde_G}
		\tilde{G}_\mu[\betadot] \assign  H_c[ (\mu / \beta_\mu) Q_\mu (\id - K_\mu )\inv [\betadot / \beta_\mu] ], \quad \text{and} \quad \tg_{\mu}(\betadot,\betadot) \assign \frac{\varepsilon}{2} \| \betadot \|_{\Hil_c}^2 + \varepsilon \langle \betadot, \tilde{G}_\mu[\betadot] \rangle_{\Hil_c}.
	\end{equation}
\end{definition}

\noindent As $\tg$ was already defined in~\eqref{eq:def:tgmu}, we need to check that Definition~\ref{def:tg_extended} coincides with~\eqref{eq:def:tgmu} when $\betadot$ is orthogonal to $\beta_\mu$.  In this case, using \(\beta_\mu = H_c[\mu/\beta_\mu]\), we get $Q_\mu[\betadot/\beta_\mu] = \betadot/\beta_\mu$ because $\langle \mu, \betadot / \beta_\mu \rangle = \langle \mu / \beta_\mu,\betadot \rangle = \langle \betadot, \beta_\mu \rangle_{\Hil_c} = 0$. Moreover, note that $Q_\mu$ is self-adjoint in $L^2(X,\mu)$ and
\begin{multline*}
	\langle \betadot, \tilde{G}_\mu[\betadot] \rangle_{\Hil_c} 
	= \langle (\mu / \beta_\mu) Q_\mu (\id - K_\mu )\inv [\betadot / \beta_\mu], \betadot \rangle  \\
	= \langle  Q_\mu (\id - K_\mu )\inv [\betadot / \beta_\mu], \betadot / \beta_\mu \rangle_{L^2(X,\mu)} 
	= \langle  (\id - K_\mu )\inv [\betadot / \beta_\mu], \betadot / \beta_\mu \rangle_{L^2(X,\mu)},   
\end{multline*}
so that the two definitions match. We state our first main result: the expansion of the distance $\dS$.

Write $\mu^V_t = (1-t) \mu + t \nu$ for the vertical interpolation between $\mu$ and $\nu$ and denote the arc connecting \(\beta_\mu\) to \(\beta_\nu\) on the unit sphere of \(\Hil_c\) by \((\beta^A_t)_t\) with corresponding path \((\mu^A_t)_t\) in \(\prm(X)\).

\begin{theorem} \label{theorem:expansion_dS}
    For any reference point \(\ol{\mu} \in \prm(X)\) and measures $\mu, \nu \in \prm(X)$ we get 
    \begin{align}
        \label{eq:expansion_dS_beta}
		\dS^2(\mu,\nu) 
		& = \tg_{\ol{\mu}}(\beta_{\nu} - \beta_\mu,\beta_{\nu} - \beta_\mu) + o(\dS^2(\mu,\nu))
    \end{align}
    uniformly as $\max \{\dS(\mu,\ol{\mu}), \dS(\nu,\ol{\mu})\} \to 0$. Moreover, for any \(t \in [0,1]\) we have
    \begin{align}
		\label{eq:expansion_dS_mu}
		\dS^2(\mu,\nu) & = \g_{\mu^V_t}(\nu - \mu, \nu - \mu) + o(\dS^2(\mu,\nu))
	\end{align}
    as $\dS(\mu,\nu) \to 0$, uniformly over $\prm(X)^2$. 
	In addition, there holds
	\begin{equation}
		\label{eq:approximate_geodesics}
		\dS^2(\mu,\nu) = E((\mu^A_t)_t) + o(\dS^2(\mu,\nu)) = E((\mu^V_t)_t) + o(\dS^2(\mu,\nu)).
	\end{equation}
\end{theorem}

The more restrictive choice of foot point in~\eqref{eq:expansion_dS_mu} compared to~\eqref{eq:expansion_dS_beta} simply comes from expressing the metric tensor in the \(\mu\)-variable and not from the choice of the path itself.
Next, we state a similar result for $S_\varepsilon$, however in this case we are unable to obtain a two-sided $o(\dS^2(\mu,\nu))$ for the remainder. As a fallback we introduce 
\begin{equation*}
	W_c(\mu,\nu) = \min_{\pi \in \Pi(\mu,\nu)} \iint_{X \times X} d_c(x,y) \diff \pi(x,y)
\end{equation*}
the one-Wasserstein distance built on the distance $d_c$ as defined in~\eqref{eq:def_dc}. It metrizes the weak convergence as $d_c$ metrizes the topology on $X$.

\begin{theorem}
	\label{theorem:estimates_with_Wc}
	Uniformly over $\prm(X)^2$, we have
	\begin{align}
		S_\varepsilon(\mu,\nu) &\leq \dS^2(\mu,\nu) + o \left( W_c^2(\mu,\nu) \right) &\tn{as } W_c(\mu,\nu) \to 0, \label{eq:Seps_asymp_ubound_Wc} \\
		S_\varepsilon(\mu,\nu) &\geq \dS^2(\mu,\nu) + o(\dS^2(\mu,\nu)) &\tn{as } \dS(\mu,\nu) \to 0. \label{eq:Seps_asymp_lbound_dS}
	\end{align}
	Moreover, there exists a constant $C > 0$ depending only on $X$, $c$ and $\varepsilon$ such that for all \(\mu,\nu \in \prm(X)\) it holds
	\begin{equation} \label{eq:Seps_bound_Wc}
		S_\varepsilon(\mu,\nu) \leq C W_c^2(\mu,\nu).
	\end{equation}  
\end{theorem}

\noindent An upper bound of $S_\varepsilon$ by a Wasserstein distance was already proved in \cite[Rem.~5.11]{RGSD}, but under an assumption of differentiability on the cost function. Here $X$ can be any compact metrizable space, without a differential structure.

\begin{remark}
We believe that by adapting our techniques, in the differentiable setting of \cite{RGSD} (see Remark~\ref{rmk:diff_setting}), we can upgrade $o(W_c^2(\mu,\nu))$ to $o(\| \mu - \nu \|_{\Cont^m(X)^*}^2)$ in~\eqref{eq:Seps_asymp_ubound_Wc}. This is a better control if $m \geq 2$, though we still have $\dS(\mu,\nu) \leq C\| \mu - \nu \|_{\Cont^m(X)^*}$ \cite[Prop.~5.9]{RGSD}.
See Remark~\ref{rem:why_not_better}, after the proof, on why there is an obstruction in obtaining a remainder in $o(\dS^2(\mu,\nu))$ for the upper bound. 
\end{remark}

Lastly, we quote the work \cite[Proposition 16]{sejourne2019sinkhorn} which proved the lower bound
\begin{equation}
	\label{eq:Seps_lower_bound_Hc}
	S_\varepsilon(\mu,\nu) \geq  \frac{\varepsilon}{2} \norm{\beta_\nu - \beta_\mu}_{\Hil_c}^2,
\end{equation}
see also Lemma~\ref{lemma:Seps_KL_reformulation} below where we refine this lower bound. Before moving to the proofs we summarize all comparisons between distances up to multiplicative constants. Denoting $A \lesssim B$ a comparison $A \leq C B$ which holds up to a multiplicative $C$ depending only on $X$, $c$ and $\varepsilon$, and $A \asymp B$ if $A \lesssim B$ and $B \lesssim A$, Theorem~\ref{theorem:dS_metric}, Theorem~\ref{theorem:estimates_with_Wc} and~\eqref{eq:Seps_lower_bound_Hc} read
\begin{equation}
	\label{eq:bounds_distance}
	\norm{\beta_\nu - \beta_\mu}_{\Hil_c}^2 \; \asymp \; \dS^2(\mu,\nu) \; \lesssim \;  S_\varepsilon(\mu,\nu) \; \lesssim \; W_c(\mu,\nu)^2.
\end{equation}
The rest of this section is dedicated to the proof of Theorem~\ref{theorem:expansion_dS} and Theorem~\ref{theorem:estimates_with_Wc}. They rely on continuity estimates of the metric tensors in the relevant topologies.

\subsection{Continuity of the metric tensor}

If $Y,Z$ are Banach spaces we denote by $\BB(Y;Z)$ the set of bounded linear operators from $Y$ to $Z$.
We endow $\BB(Y;Z)$ with the topology of (operator) norm convergence. We denote by $\| \vdot \|_{Y \to Z}$ such operator norm, meaning $\| A \|_{Y \to Z} = \sup \{ \| Ay \|_Z \; | \; \| y \|_Y \leq 1 \}$, and operator convergence corresponds to convergence in operator norm. The operator norm is submultiplicative, and from the identity $\| A\inv-B\inv \| = \| A\inv (B - A) B\inv \| \leq \| A\inv \| \| B\inv \| \| A - B \|$ when the inverses exist, convergence in operator norm together with boundedness of the inverse implies convergence of the inverses in operator norm. 

Recall that $\tilde{G}_\mu$, characterizing $\tg_\mu$, is defined in~\eqref{eq:def_tilde_G}, while on the other hand we call
\begin{equation}
	\label{eq:def_Qmunu_simpler}
	G_{\mu,\nu} = (\id - K_{\mu,\nu} K_{\nu,\mu})\inv H_{\mu,\nu}.
\end{equation}
It is such that $\g_\mu(\sigma,\sigma) = \frac{\varepsilon}{2} \langle \sigma, {G}_{\mu,\mu}[\sigma] \rangle$, and as it will be clear later we will only need the ``off-diagonal'' operators for the proof of Theorem~\ref{theorem:estimates_with_Wc}. In this section we prove the continuity of $\mu \mapsto \tilde{G}_\mu$ and $(\mu,\nu) \mapsto G_{\mu,\nu}$ in appropriate operator norms. 

We start with $G_{\mu,\nu}$, and for that we need several preliminary results on $W_c$. 
As the $W_c$ distance has a dual formulation in terms of Lipschitz functions with respect to $d_c$, we introduce the corresponding spaces $\Lip_c(X)$ and $\Lip_c(X)/\R$ as subsets of $\Cont(X)$ and $\Cont(X)/\R$. We write
\begin{equation*}
	\Lip_c(\phi) \assign  \sup_{x \neq y} \frac{|\phi(x) - \phi(y)|}{d_c(x,y)}, \qquad \norm{\phi}_{\Lip_c} \assign  \| \phi \|_\infty + \Lip_c(\phi), \qquad \norm{\phi}_{\Lip_c/\R} \assign \Lip_c(\phi).
\end{equation*} 
The Lipschitz constant \(\Lip_c(\phi)\) defines a norm on \(\Lip_c(X) / \R\) which is equivalent to the quotient norm \(\inf_{\lambda \in \R} \norm{\phi - \lambda \ones_X}_{\Lip_c} \geq \Lip_c(\phi)\).  
We denote by $\Lip_c(X)^*$ the topological dual of $\Lip_c(X)$, and $\Lip_c(X)^*_0$ the subset made of linear forms vanishing on the constant function $\ones_X$. Define the associated norm of the latter as $\| \sigma \|_{\Lip_{c,0}^*} = \sup \setgiven{ \abs{\pair{\sigma}{\phi}} }{ \Lip_c(\phi) \leq 1 }$. By Kantorovich duality \cite[Remark 6.5]{Villani2009} we always have 
\begin{equation}
	\label{eq:Kantorovich_duality}
	W_c(\mu,\nu) = \| \mu - \nu \|_{\Lip_{c,0}^*} \geq \| \mu - \nu \|_{\Lip_{c}^*}.
\end{equation}
We can easily check the estimate
\begin{equation}
	\label{eq:Lipc_bound_product}
	\| \phi \psi \|_{\Lip_c} \leq \| \phi \| _{\Lip_c} \| \psi \| _{\Lip_c}.
\end{equation} 
This stability of the set $\Lip_c(X)$ by multiplication is actually why it is a useful extension of $\Hil_c$. In contrast, products of elements in $\Hil_c$ do not stay in $\Hil_c$.
For any \(\phi \in \Hil_c\), we get \(\abs{\phi(x) - \phi(y)} = \abs{\pair{\phi}{k_c(x,\vdot) - k_c(y,\vdot)}_{\Hil_c}} \leq \norm{\phi}_{\Hil_c} \cdot d_c(x,y)\). Moreover as we assume $c$ to be non-negative so that $k_c \leq 1$, for any $x \in X$ we have $|\phi(x)| = |\langle \phi, k_c(x,\vdot) \rangle_{\Hil_c}| \leq \| \phi \|_{\Hil_c} \sqrt{k_c(x,x)} \leq \| \phi \|_{\Hil_c}$. Hence
\begin{equation} \label{eq:Lipc_Hc_bound}
	\norm{\phi}_{\Lip_c/\R} \leq \norm{\phi}_{\Hil_c}, \qquad \norm{\phi}_{\Lip_c} \leq 2 \norm{\phi}_{\Hil_c}
\end{equation}
Using then the identity $\norm{\sigma}_{\Hil_{c}^*} = \sup \{ \langle \sigma, \phi \rangle \; | \; \| \phi \|_{\Hil_c} \leq 1 \}$, we can ``dualize'' the previous estimate and get
\begin{equation}
	\label{eq:Lip_Hc_bound_dual}
	\| \sigma \|_{\Hil_{c}^*} \leq 2 \| \sigma \|_{\Lip_{c}^*}.
\end{equation}


\begin{lemma}
	\label{lemma:continuous_potentials_Lipc}
	We can choose the representative $f_{\mu,\nu}$ of the Schrödinger potentials such that the maps $(\mu,\nu) \mapsto f_{\mu,\nu} \in \Lip_c(X)$ and $(\mu,\nu) \mapsto \exp(f_{\mu,\nu}/\varepsilon) \in \Lip_c(X)$ are continuous. 
\end{lemma}

\begin{proof}
	We can choose representatives such that $(\mu,\nu) \mapsto g_{\mu,\nu} \in \Cont(X)$ is continuous \cite[Prop.~B.1]{nutz2023stability}. 
	As $\exp(-\fmunu/\varepsilon) = H_c[\exp(\gmunu/\varepsilon) \nu]$ by the Schrödinger system~\eqref{eq:Schroedinger_system}, with Lemma~\ref{lemma:appendixBRGSD} we deduce that $\exp(-\fmunu/\varepsilon)$ depends continuously on $\mu,\nu$ in $\Hil_c$. With~\eqref{eq:Lipc_Hc_bound}, we deduce that it does in $\Lip_c(X)$. As moreover the Schrödinger potentials are uniformly bounded, we deduce that also $f_{\mu,\nu} \in \Lip_c(X)$ depends continuously on $\mu,\nu$. 
\end{proof}

As multiplication by $\beta_\mu$ is an isometry from $\Hil_\mu$ to $\Hil_c$, from~\eqref{eq:Lipc_bound_product} we can extend the estimate~\eqref{eq:Lipc_Hc_bound} to \(\phi \in \Hil_\mu\):  
\begin{equation} \label{eq:Lipc_Hmu_bound}
	\norm{\phi}_{\Lip_c} \leq \norm{1/\beta_\mu}_{\Lip_c} \cdot 2 \norm{\phi \beta_\mu}_{\Hil_c} \leq C \norm{\phi}_{\Hil_\mu}
\end{equation}
for some \(C > 0\).
We can now state and prove our first continuity result. 

\begin{proposition}
	\label{prop:metric_tensor_mu_smooth}
	The operator $G_{\mu,\nu}$ belongs to $\BB(\Lip^*_{c,0}(X);\Lip_c(X)/\R)$ and is continuous in operator norm as a function of $\mu,\nu$.
\end{proposition}

\begin{proof}
	We prove that $(\mu,\nu) \mapsto H_{\mu,\nu} \in \BB(\Lip_{c}(X)_0^*;\Lip_c(X)/\R)$ as well as $(\mu,\nu) \mapsto (\id - K_{\mu,\nu} K_{\nu,\mu})\inv \in \BB(\Lip_c(X)/\R;\Lip_c(X)/\R)$ are continuous operator-valued maps, as the conclusion follows by composition. 
	
	It is $H_{\mu,\nu} = \exp(f_{\nu,\mu}/\varepsilon) H_c \exp(f_{\mu,\nu}/\varepsilon)$. Multiplication by $\exp(\fmunu/\varepsilon)$ and \(\exp(\fnumu/\varepsilon)\) varies continuously with $\mu,\nu$ in $\BB(\Lip_c(X)^*; \Lip_c(X)^*)$ and $\BB(\Lip_{c}(X); \Lip_c(X))$ respectively by Lemma~\ref{lemma:continuous_potentials_Lipc} and~\eqref{eq:Lipc_bound_product}. The operator $H_c$ does not depend on $\mu,\nu$ and belongs to $\BB(\Lip_c(X)^*;\Lip_c(X))$ given $H_c \in \BB(\Hil_c^*;\Hil_c)$ and~\eqref{eq:Lipc_Hc_bound} and~\eqref{eq:Lip_Hc_bound_dual}. 
    We deduce that $H_{\mu,\nu} \in \BB(\Lip_{c}(X)^*;\Lip_c(X))$ varies continuously with $\mu,\nu$. Composing with the inclusion \(\Lip_c(X)^*_0 \to \Lip_c(X)^*\) and the projection $\Lip_c(X) \to \Lip_c(X)/\R$, together with the equivalence of \(\norm{\phi}_{\Lip_c/\R} = \Lip_c(\phi)\) to the quotient norm, we get our conclusion for $H_{\mu,\nu}$.   
	
	We turn to $(\id - K_{\mu,\nu} K_{\nu,\mu})\inv$. We only need to prove that $(\mu,\nu) \mapsto (\id - K_{\mu,\nu} K_{\nu,\mu})$ is continuous, and that $\| (\id - K_{\mu,\nu} K_{\nu,\mu})\inv \|_{\Lip_c/\R \to \Lip_c/\R}$ remains uniformly bounded. Note that $K_{\mu,\nu}[\phi] = H_{\mu,\nu}[\phi \mu]$. Multiplication by $\mu$ belongs to $\BB(\Lip_c(X);\Lip_c(X)^*)$ and varies continuously with $\mu$: if $\mu, \mu' \in \prm(X)$, using Kantorovich duality~\eqref{eq:Kantorovich_duality} and~\eqref{eq:Lipc_bound_product} 
	\begin{equation*}
		\| \phi \mu - \phi \mu' \|_{\Lip_c^*} = \sup_{\| \psi \|_{\Lip_c \leq 1}} \langle \phi (\mu - \mu'), \psi \rangle \leq \sup_{\| \psi \|_{\Lip_c \leq 1}} \| \phi \psi \|_{\Lip_c} \| \mu - \mu' \|_{\Lip_{c}^*} \leq \| \phi \|_{\Lip_c} W_c(\mu,\mu').
	\end{equation*}
	Thus by composition $K_{\mu,\nu} \in \BB(\Lip_c(X);\Lip_c(X))$ varies continuously with $\mu,\nu$. From this and by taking the quotient with constant functions we conclude that $(\mu,\nu) \mapsto (\id - K_{\mu,\nu} K_{\nu,\mu}) \in \BB(\Lip_c(X)/\R;\Lip_c(X)/\R)$ is continuous. Eventually, we need to uniformly bound the norm of $(\id - K_{\mu,\nu} K_{\nu,\mu})\inv$. We start by writing
	\begin{equation*}
		(\id - K_{\mu,\nu} K_{\nu,\mu})\inv = \id + K_{\mu,\nu} (\id - K_{\nu,\mu} K_{\mu,\nu})\inv K_{\nu,\mu}
	\end{equation*}         
    using the Neumann series as operators on the larger space \(\Cont(X)/\R\).
    Taking appropriate operator norms gives 
	\begin{align}
    \begin{split}
        \label{eq:boundedness_inverse}
		\| (\id - K_{\mu,\nu} &K_{\nu,\mu})\inv \|_{\Lip_c/\R \to \Lip_c/\R}   \\
		&\leq 1 + \norm{K_{\mu,\nu}}_{\Cont/\R \to \Lip_c/\R}  \|(\id - K_{\nu,\mu} K_{\mu,\nu})\inv\|_{\Cont/\R \to \Cont/\R} \|K_{\nu,\mu} \|_{\Lip_c/\R \to \Cont/\R} . 
    \end{split}
	\end{align}
	By the continuity we just proved and the compactness of $\prm(X)^2$, $\|K_{\nu,\mu} \|_{\Lip_c/\R \to \Lip_c/\R}$ is uniformly bounded. This extends immediately to \(\norm{K_{\nu,\mu}}_{\Lip_c/\R \to \Cont/\R}\). 
	We also have that $\|(\id - K_{\nu,\mu} K_{\mu,\nu})\inv\|_{\Cont/\R \to \Cont/\R}$ is bounded uniformly for all $\mu,\nu \in \prm(X)$ from~\eqref{eq:contraction_Kmunu}. Eventually, to bound $\|K_{\mu,\nu} \|_{\Cont/\R \to \Lip_c/\R}$, from~\eqref{eq:Lipc_bound_product} and~\eqref{eq:Lipc_Hc_bound},
	\begin{align*}
		\|K_{\mu,\nu}[\phi] \|_{\Lip_c} & \leq \| \exp(\fnumu/\varepsilon) \|_{\Lip_c} \| H_c[ \exp(\fmunu/\varepsilon) \phi \mu ]  \|_{\Hil_c} \\
		& \leq \| \exp(\fnumu/\varepsilon) \|_{\Lip_c} \| \exp(\fmunu/\varepsilon) \phi \mu  \|_{\mathrm{TV}} \leq C \| \phi \|_\infty, 
	\end{align*}
	where the constant depends only on the uniform bounds on the Schrödinger potentials in $\Cont(X)$ and $\Lip_c(X)$. We deduce that $\|K_{\mu,\nu}\|_{\Cont \to \Lip_c}$ is bounded uniformly, thus $\|K_{\mu,\nu}\|_{\Cont/\R \to \Lip_c/\R}$ as $K_{\mu,\nu}[\ones_X] = \ones_X$. Plugging this back in~\eqref{eq:boundedness_inverse} we deduce that the left hand side is bounded, which, together with the continuity of $\id - K_{\mu,\nu} K_{\nu,\mu}$, yields the conclusion.  
\end{proof}

Next we move to the continuity result for $\tilde{G}_\mu$. It is phrased in the space $\BB(\Hil_c;\Hil_c)$ which is the canonical one for this operator. 

\begin{proposition}
	\label{prop:metric_tensor_beta_smooth}
	The operator $\tilde{G}_{\mu}$ belongs to $\BB(\Hil_c;\Hil_c)$ and is continuous in operator norm as a function of $\mu$.
\end{proposition}

\begin{proof}
	Assume $\mu_n \to \mu$ weakly, write $\beta_n = \beta_{\mu_n}$ and $\tilde{G}_n = \tilde{G}_{\mu_n}$ and take $\betadot_n \in \Hil_c$ of unit norm such that $\| (\tilde{G}_n - \tilde{G}_\mu) \betadot_n \|_{\Hil_c} \geq \frac{1}{2} \| \tilde{G}_n - \tilde{G}_\mu \|_{\Hil_c \to \Hil_c}$. Up to extraction, $\betadot_n$ converges weakly in $\Hil_c$ to a limit $\betadot$. It follows from Lemma~\ref{lemma:appendixBRGSD} that $\betadot_n /\beta_n$ converges in $\Cont(X)$ to a $\betadot/\beta_\mu$. As proved in \cite[Proposition 4.14]{RGSD}, it implies that $(\id - K_{\mu_n} )\inv [\betadot_n / \beta_n]$ converges to $(\id - K_\mu )\inv [\betadot / \beta_\mu]$ as $n \to \infty$ in $\Cont(X)/\R$. We can pair weak convergence of measures with strong convergence of functions, so that $\alpha_{\mu_n} Q_{\mu_n} (\id - K_{\mu_n} )\inv [\betadot_n / \beta_n]$ converges weakly in $\calM(X)$ to $\alpha_\mu Q_\mu (\id - K_\mu )\inv [\betadot / \beta_\mu]$. Eventually, $H_c  : \calM(X) \to \Hil_c$ is \textweakstar-to-norm continuous (Lemma~\ref{lemma:appendixBRGSD}), thus $\tilde{G}_n \betadot_n \to \tilde{G}_\mu \betadot$ in norm in $\Hil_c$. The conclusion follows as also $\tilde{G}_\mu \betadot_n \to \tilde{G}_\mu \betadot$, so $\| (\tilde{G}_n - \tilde{G}_\mu) \betadot_n \|_{\Hil_c} \to 0$.      
\end{proof}

\subsection{Proof of the local expansion for the distance}

\begin{proof}[Proof of Theorem~\ref{theorem:expansion_dS}]
	We call $\gamma$ a modulus of continuity of $\mu \in (\prm(X),\dS) \mapsto \tg_\mu$ as a bilinear form on $\Hil_c$, by Definition~\ref{def:tg_extended} and Proposition~\ref{prop:metric_tensor_beta_smooth} it exists. In this proof, to simplify notations we write $A \lesssim B$ if there exists a constant $C > 0$ depending only on $c$, $X$ and $\varepsilon$ such that $A \leq C B$. 
	
	\medskip
	
	\emph{1st step: $(\mu^A_t)_t$ is asymptotically optimal.}
	Call $\varrho = \max \{\dS(\mu,\ol{\mu}), \dS(\nu,\ol{\mu})\}$. From the estimates in Theorem~\ref{theorem:dS_metric} and an easy geometric consideration for geodesics on the sphere, for any $t \in [0,1]$
	\begin{equation*}
		\dS(\ol{\mu},\mu^A_t) \lesssim \| \ol{\beta} - \beta_{t}^A \|_{\Hil_c} \lesssim \max \{ \| \ol{\beta} - \beta_\mu \|_{\Hil_c}, \| \ol{\beta} - \beta_\nu \|_{\Hil_c} \} \lesssim \varrho. 
	\end{equation*}
	Given the explicit expression of the geodesic on the sphere~\eqref{eq:explicit_formula_arc}, using that the angle $\theta$ between $\beta_\mu$ and $\beta_\nu$ satisfies $|\theta| \lesssim \dS(\mu,\nu)$, we have $\| \betadot^A_t \|_{\Hil_c} \lesssim \| \beta_\nu - \beta_\mu \|_{\Hil_c}$ and 
	\begin{equation*}
		\betadot^A_t = \beta_\nu - \beta_\mu + o(\dS(\mu,\nu)). 
	\end{equation*}
	as $\dS(\mu,\nu) \to 0$. Combining the expansion of $\mu^A_t$ and $\betadot^A_t$:
	\begin{align}
		\notag
		\dS^2(\mu,\nu) \leq E((\mu^A_t)_t) & = \int_0^1 \tg_{\mu^A_t}(\betadot^A_t, \betadot^A_t) \diff t \\
		\notag
		& = \int_0^1 \tg_{\ol{\mu}}(\betadot^A_t, \betadot^A_t) \diff t + \calO \left( \gamma(\varrho) \cdot \| \beta_\nu - \beta_\mu \|_{\Hil_c}^2 \right) \\
		\notag
		& = \int_0^1\tg_{\ol{\mu}}(\beta_\nu - \beta_\mu,\beta_\nu - \beta_\mu) \diff t + \calO \left( \gamma(\varrho) \cdot \dS^2(\mu,\nu) \right) + o(\dS^2(\mu,\nu)) \\
		\label{eq:zz_aux_upper_bound}
		& = \tg_{\ol{\mu}}(\beta_\nu - \beta_\mu,\beta_\nu - \beta_\mu) \diff t + o(\dS^2(\mu,\nu)),
	\end{align}
	where in the last equality we used $\dS(\mu,\nu) \leq 2 \varrho$.
	
	On the other hand let $(\beta_t)_t$ be the path in $\calB$ corresponding to a geodesic $(\mu_t)_t$ for $\dS$ between $\mu$ and $\nu$. We easily have
	\begin{equation*}
		\dS(\mu_t,\mu^A_t)  \leq \dS(\mu_t,\mu) + \dS(\mu,\ol{\mu}) + \dS(\ol{\mu},\mu^A_t)  \lesssim \dS(\mu,\nu) + 2 \varrho \lesssim \varrho. 
	\end{equation*}
	Using the continuity of $\mu \mapsto \tg_\mu(\cdot,\cdot)$, the convexity of the map $\betadot \mapsto \tg_{\ol{\mu}}(\betadot,\betadot)$ on \(\Hil_c\) and Jensen's inequality for Bochner integrals~\cite[Theorem~3.2]{Perlman1974}
	\begin{align}
		\notag
		\dS^2(\mu,\nu) = E((\mu_t)_t) & = \int_0^1 \tg_{\ol{\mu}}(\betadot_t,\betadot_t) \diff t + \calO\left( \gamma(\varrho) \cdot \norm{(\betadot_t)_t}_{L^2((0,1);\Hil_c)}^2\right) \\
		\notag
		& \geq \tg_{\ol{\mu}}\left(\int_0^1 \betadot_t \diff t,\int_0^1 \betadot_t \diff t\right) + \calO\left( \gamma(\varrho) \cdot \dS^2(\mu_0,\mu_1)\right) \\
		\label{eq:zz_aux_lower_bound}
		& = \tg_{\ol{\mu}}(\beta_\mu - \beta_\nu,\beta_\mu - \beta_\nu) +  o( \dS^2(\mu_0,\mu_1)).
	\end{align}
	Combining~\eqref{eq:zz_aux_upper_bound} and~\eqref{eq:zz_aux_lower_bound}, we obtain as $\varrho \to 0$
	\begin{equation*}
		\dS^2(\mu,\nu) = E((\mu^A_t)_t) + o( \dS^2(\mu,\nu)) = \tg_{\ol{\mu}}(\beta_\mu - \beta_\nu,\beta_\mu - \beta_\nu) +  o( \dS^2(\mu,\nu)).
	\end{equation*}
	This proves~\eqref{eq:expansion_dS_beta} as well as the first equality in~\eqref{eq:approximate_geodesics}.
	
	\medskip
	
	\emph{2nd step: $(\mu^V_t)_t$ is close to $(\mu^A_t)_t$}. Next we want to do the estimates in the $\mu$ variable. We will prove in this step that 
	\begin{equation}
		\label{eq:mudotV_mudotA_close}
		\sup_{t \in [0,1]} \| \mudot^V_t - \mudot^A_t \|_{\Lip_c^*} = o(\dS(\mu,\nu)). 
	\end{equation} 
	Recalling the explicit expression of the geodesic arc~\eqref{eq:explicit_formula_arc}, we write $u_t = \sin((1-t) \theta)/\sin(\theta)$ and $v_t = \sin(t \theta)/\sin(\theta)$, with $\theta$ being the angle between $\beta_\mu$ and $\beta_\nu$ and suppressing the dependence in $\theta$ in the notation. 
	From $\mu = \alpha_\mu \beta_\mu$, and $\alpha = H_c[\beta]$, 
	\begin{equation*}
		\mu^A_t = (u_t \alpha_\mu + v_t \alpha_\nu)(u_t \beta_\mu + v_t \beta_\nu) = u_t^2 \mu + v_t^2 \nu + u_t v_t (\mu + \nu) - u_t v_t (\alpha_\nu - \alpha_\mu)(\beta_\nu - \beta_\mu).
	\end{equation*}
	The remainder $(\alpha_\nu - \alpha_\mu)(\beta_\nu - \beta_\mu)$ is negligible in the space $\Lip_c(X)^*$. Indeed, using in particular~\eqref{eq:Lipc_bound_product} and~\eqref{eq:Lipc_Hc_bound}, 
	\begin{align*}
		\| (\alpha_\nu - \alpha_\mu)(\beta_\nu - \beta_\mu) \|_{\Lip_c^*} & \leq \| \alpha_\nu - \alpha_\mu \|_{\Lip_c^*} \; \| \beta_\nu - \beta_\mu \|_{\Lip_c} \\
		& \lesssim \| \alpha_\nu - \alpha_\mu \|_{\Lip_c^*} \; \| \beta_\mu - \beta_\nu \|_{\Hil_c} \lesssim  \|\alpha_\nu - \alpha_\mu \|_{\Lip_c^*} \; \dS(\mu,\nu). 
	\end{align*}
	As $\| \vdot \|_{\Lip_c^*}$ is nothing else than the bounded Lipschitz distance, it metrizes \textweakstar{} convergence over $\calM(X)$ (see e.g. \cite[Thm.~8.3.2]{bogachev2007measure}) and $\mu \mapsto \alpha_\mu$ is a homeomorphism (Theorem~\ref{theorem:alpha_beta_homeo}). Thus $\| \alpha_\nu - \alpha_\mu\|_{\Lip_c^*} = o(1)$ as $\dS(\mu,\nu) \to 0$, yielding 
	\begin{equation}
		\label{eq:p:dS_asymp:control_reminder}
		\| (\alpha_\nu - \alpha_\mu)(\beta_\nu - \beta_\mu) \|_{\Lip_c^*} = o(\dS(\mu,\nu)). 
	\end{equation}
	As $u_t = 1-t + o(\theta)$, $\dot{u}_t = -1 + o (\theta)$, $v_t = t + o(\theta)$ and $\dot{v}_t = 1 + o(\theta)$, with $|\theta| \lesssim \dS(\mu,\nu)$ we have in $\Lip_c(X)^*$: 
	\begin{equation*}
		\mudot^A_t = 2 \dot{u}_t u_t \mu + 2 \dot{v} v_t \nu + (\dot{u}_t v_t + u_t \dot{v_t})(\mu + \nu) + o(\dS(\mu,\nu)) = \nu - \mu + o(\dS(\mu,\nu)).     
	\end{equation*}
	The conclusion~\eqref{eq:mudotV_mudotA_close} follows as $\mudot^V_t = \nu - \mu$ for all $t$. 
	
	\medskip

        \emph{3rd step: $(\mu^V_t)_t$ is asymptotically optimal}. Write \(\mu_0 = \mu\), \(\mu_1 = \nu\). Proposition~\ref{prop:metric_tensor_mu_smooth} together with the compactness of $X$ yields $\g_\mu(\mudot,\mudot) \lesssim \| \mudot \|^2_{\Lip_c^*}$. Thus with~\eqref{eq:mudotV_mudotA_close} from the previous step applied at $t=0$ we obtain
        \begin{equation} \label{eq:p:expansion_dS_mu:betadot_0}
            \norm{\betadot^V_0 - \betadot^A_0}_{\Hil_c}^2 \asymp \tg_{\mu_0}(\betadot^V_0 - \betadot^A_0, \betadot^V_0 - \betadot^A_0) = \g_{\mu_0}(\mudot^V_0 - \mudot^A_0, \mudot^V_0 - \mudot^A_0) = o(\dS^2(\mu,\nu))
        \end{equation}
        We want to show that \((\beta^V_t)_t\) approximates the arc at all intermediate times as well.        
        We know that \(\betadot^A_0 = \beta_1 - \beta_0 + o(\norm{\beta_1 - \beta_0}_{\Hil_c})\) and \(\norm{\beta_1 - \beta_0}_{\Hil_c} \asymp \dS(\mu_0,\mu_1)\). Thus we can reformulate~\eqref{eq:p:expansion_dS_mu:betadot_0} as $\norm{\betadot^V_0 - (\beta_1 - \beta_0)}_{\Hil_c} = o(\norm{\beta_1 - \beta_0}_{\Hil_c})$
        uniformly as \(\norm{\beta_1 - \beta_0}_{\Hil_c} \to 0\), allowing only for slight deviations from the ``right'' direction. 
        We apply the same argument to the interpolation between $\mu^V_t$ and $\mu_1$, showing
        \begin{equation} \label{eq:p:expansion_dS_mu:betadot_t_o1}
           \left\| \betadot^V_t - \frac{1}{1-t} (\beta_1 - \beta^V_t) \right\|_{\Hil_c} = o \left( \frac{1}{1-t} \norm{\beta_1 - \beta^V_t}_{\Hil_c} \right)
        \end{equation} 
        as \( \frac{1}{1-t} \norm{\beta_1 - \beta^V_t}_{\Hil_c} \to 0\). 
        Write \(\beta^C_t = (1-t) \beta_0 + t \beta_1\) for the chord in \(\Hil_c\). Then \(\beta_1 - \beta_0 = \frac{1}{1-t} (\beta_1 - \beta^C_t)\). In particular, for any $t \in [0,1/2]$
        \begin{align*}
            \frac{\diff}{\diff t} \| \beta^V_t - \beta^C_t \|_{\Hil_c} & \leq \| \betadot^V_t - \betadot^C_t \|_{\Hil_c} \\
            & \leq  \bnorm{\betadot^V_t - \frac{1}{1-t} (\beta_1 - \beta^V_t)}_{\Hil_c} + \frac{1}{1-t} \bnorm{ \beta^V_t - \beta^C_t}_{\Hil_c} \\
            & \leq o(\norm{\beta_1 - \beta^V_t}_{\Hil_c}) + 2 \bnorm{ \beta^V_t - \beta^C_t}_{\Hil_c} \\
            & \leq o(\| \beta_1 - \beta_0 \|_{\Hil_c}) + (2+o(1)) \cdot \bnorm{ \beta^V_t - \beta^C_t}_{\Hil_c}.
        \end{align*}
        With Grönwall's inequality we deduce that, uniformly for $t \in [0,1/2]$, 
        \begin{equation*}
            \| \beta^V_t - \beta^C_t \|_{\Hil_c} = o(\| \beta_1 - \beta_0 \|_{\Hil_c}).
        \end{equation*}
        We deduce for the derivative, uniformly for $t \in [0,1/2]$, 
        \begin{equation*}
            \norm{\betadot^V_t - (\beta_1 - \beta_0)}_{\Hil_c} 
            \leq \norm{\betadot^V_t - \frac{1}{1-t} (\beta_1 - \beta^V_t)}_{\Hil_c} + \frac{1}{1-t} \norm{\beta^V_t - \beta^C_t}_{\Hil_c} 
            = o(\norm{\beta_1 - \beta_0}_{\Hil_c}).
        \end{equation*}
        The symmetric argument starting from \(\mu_1\) completes this bound for all \(t \in [0,1]\). That is, we have
        \begin{equation} \label{eq:p:expansion_dS_mu:approx}
            \norm{\beta^V_t - \beta^A_t}_{\Hil_c} = o(\norm{\beta_1 - \beta_0}_{\Hil_c}), \qquad \norm{\betadot^V_t - (\beta_1 - \beta_0)}_{\Hil_c} = o(\norm{\beta_1 - \beta_0}_{\Hil_c})
        \end{equation}
        as \(\norm{\beta_1 - \beta_0}_{\Hil_c} \to 0\), uniformly over \(t \in [0,1]\).

        With~\eqref{eq:p:expansion_dS_mu:approx} and~\eqref{eq:expansion_dS_beta} we can estimate the energy of the vertical interpolation as 
        \begin{align*}
            \g_{\mu^V_t}(\mudot^V_t,\mudot^V_t) 
            &= \tg_{\mu^V_t}(\betadot^V_t,\betadot^V_t) \\
            &= \tg_{\mu^V_t}(\beta_1 - \beta_0, \beta_1 - \beta_0) + o(\norm{\beta_1 - \beta_0}_{\Hil_c}^2) 
            = \dS^2(\mu_0,\mu_1) + o(\dS^2(\mu_0,\mu_1))
        \end{align*}
        uniformly in \(t \in [0,1]\) and thus also find \(E((\mu^V_t)_t) = \dS^2(\mu_0,\mu_1) + o(\dS^2(\mu_0,\mu_1))\). \qedhere
\end{proof}

\subsection{Proof of the local expansion for the Sinkhorn divergence}

We now proceed to the proof of Theorem~\ref{theorem:estimates_with_Wc}. We rely on two different rewritings of the Sinkhorn divergence, which give respectively the bounds~\eqref{eq:Seps_asymp_ubound_Wc} and~\eqref{eq:Seps_bound_Wc} on one hand, and~\eqref{eq:Seps_asymp_lbound_dS} on the other hand.

We start with the estimates involving \(W_c\). Recall $H_{\mu,\nu}$ and $K_{\mu,\nu}$ can be found in Definition~\ref{def:kernel_operators}.

\begin{proposition}
	\label{prop:rewriting_Sinkhorn}
	If $\mu,\nu \in \prm(X)$ and with $\mu_t = (1-t) \mu + t \nu$, 
	\begin{equation*}
		S_\varepsilon(\mu,\nu) = \frac{\varepsilon}{2} \int_0^1 \int_0^1  \pair{\nu - \mu}{(\id - K_{\mu_s,\mu_t} K_{\mu_t,\mu_s})\inv H_{\mu_s,\mu_t}[\nu - \mu]} \diff t \diff s. 
	\end{equation*}
\end{proposition}

\noindent We recognize the importance of considering the ``off-diagonal'' operators $G_{\mu,\nu}$ for $\mu \neq \nu$ as these are the ones which naturally appear when rewriting the Sinkhorn divergence.
The proof is a mild adaptation of the techniques in \cite{RGSD,Hardion2025}. We first state a slight extension of the result in \cite[Sec.~3]{RGSD}.  

\begin{proposition}[{\cite[Lem.~3.13,~Prop.~3.14]{RGSD}}] \label{prop:time_derivative_of_potentials} 
	Let \((\mu_t)_{t \in I}\) be weakly continuously differentiable in $\mathcal{M}(X)$. Then the potential \(f_{t,s} \assign f_{\mu_t,\mu_s}\) is continuously differentiable in \(\Cont(X)/\R\) with respect to the time parameters in \((t,s) \in I \times I\) and the derivatives satisfy
	\begin{align*}
		\begin{split} 
			&\frac{\partial f_{t,s}}{\partial s}  = - \varepsilon (\id - K_{\mu_s,\mu_t} K_{\mu_t,\mu_s})\inv H_{\mu_s,\mu_t} [ \mudot_s ] \in \Cont(X) /\R, \\ 
			&\frac{\partial f_{t,s}}{\partial t}  = \varepsilon K_{\mu_s,\mu_t} (\id - K_{\mu_t,\mu_s} K_{\mu_s,\mu_t})\inv H_{\mu_t,\mu_s} [ \mudot_t ] \in \Cont(X) /\R. 
		\end{split}
	\end{align*}
\end{proposition}

\begin{proof}
	The differentiability was already proved in \cite[Prop.~3.14]{RGSD}. The formulas are obtained by differentiating the identity \(f_{t,s} = T_\varepsilon(T_\varepsilon(f_{t,s},\mu_t),\mu_s)\) in time and reordering terms as in \cite[Prop.~3.14]{RGSD}, the only added complexity is that one has to keep track of the distinction between \(K_{\mu_t,\mu_s}\) and \(K_{\mu_s,\mu_t}\). See in particular \cite[Prop.~7.5]{Hardion2025} where the first equation is derived.   
\end{proof}

\noindent With that, we proceed to prove Proposition~\ref{prop:rewriting_Sinkhorn}. 

\begin{proof}[Proof of Proposition~\ref{prop:rewriting_Sinkhorn}]
	Let \(F(t,s) = \OT_\varepsilon(\mu_t,\mu_s)\), which of course is symmetric. From \cite[Lem~3.2]{RGSD} (see also \cite[Prop.~2]{Feydy2018}), the function $F$ is differentiable once with $\partial_t F(t,s) = \langle \mu - \nu, f_{t,s} \rangle$. 
	With Proposition~\ref{prop:time_derivative_of_potentials} we see that $\partial_t F(t,s)$ is continuously differentiable, and moreover
	\begin{equation*}
		\partial_{s} \partial_t F(t,s) = - \varepsilon \langle \nu - \mu, (\id - K_{\mu_s,\mu_t} K_{\mu_t,\mu_s})\inv H_{\mu_s,\mu_t} [ \nu - \mu ] \rangle
	\end{equation*}
	As $F$ is of class $C^2$, by the fundamental theorem of calculus
	\begin{align*}
		\frac{1}{2} ( F(1,0) + F(0,1) - F(0,0) - F(1,1) ) &=  - \frac{1}{2}\int_{0}^1 \int_{0}^1 \partial_{s} \partial_t F(t,s) \diff t \diff s,
	\end{align*}
	and the conclusion follows as we recognize $S_\varepsilon(\mu,\nu)$ in the right hand side. 
\end{proof}

\begin{proof}[Proof of Theorem~\ref{theorem:estimates_with_Wc}: estimates~\eqref{eq:Seps_asymp_ubound_Wc} and~\eqref{eq:Seps_bound_Wc}]
We begin with~\eqref{eq:Seps_asymp_ubound_Wc}. 
Call $\omega$ a modulus of continuity of $G_{\mu,\nu}$ as defined in~\eqref{eq:def_Qmunu_simpler}, when we endow for simplicity $\prm(X)^2$ with the Wasserstein distance $W_c$, meaning 
\begin{equation*}
    \| G_{\mu,\nu} - G_{\mu',\nu'} \|_{\Lip_{c,0}^* \to \Lip_{c}/\R} \leq \omega \left( W_c(\mu,\mu') + W_c(\nu,\nu') \right).  
\end{equation*}
Its existence is guaranteed by Proposition~\ref{prop:metric_tensor_mu_smooth}. From the explicit formula of Proposition~\ref{prop:rewriting_Sinkhorn} together with the definition of the metric tensor, writing $\mu_t = (1-t) \mu + t \nu$ we have 
\begin{align*}
    |S_\varepsilon(\mu,\nu) - \g_\mu(\nu-\mu,\nu-\mu)| & = \left| \frac{\varepsilon}{2} \int_0^1 \int_0^1  \pair{\nu - \mu}{\{ G_{\mu_s,\mu_t} - G_{\mu,\mu} \}[\nu - \mu]} \diff t \diff s \right| \\
    & \leq \frac{\varepsilon}{2} W_c(\mu,\nu)^2 \int_0^1 \int_0^1 \| G_{\mu_s,\mu_t} - G_{\mu,\mu}  \|_{\Lip_{c,0}^* \to \Lip_{c}/\R} \diff t \diff s  \\
    & \leq \frac{\varepsilon}{2} W_c(\mu,\nu)^2 \; \omega( W_c(\mu,\nu)), 
\end{align*}
where we have used the Kantorovich duality~\eqref{eq:Kantorovich_duality} in the first inequality. We exchange \(\g_\mu(\nu-\mu,\nu-\mu)\) for \(\dS^2(\mu,\nu)\) using Theorem~\ref{theorem:expansion_dS} and obtain $S_\varepsilon(\mu,\nu) = \dS^2(\mu,\nu) + o \left( W_c(\mu,\nu)^2 \right)$,
which of course implies~\eqref{eq:Seps_asymp_ubound_Wc}.

We move on to the third estimate~\eqref{eq:Seps_bound_Wc}: as $(\mu,\nu) \mapsto G_{\mu,\nu} \in \BB(\Lip_c(X)^*_0; \Lip_c(X)/\R)$ is continuous, it is uniformly bounded on $\prm(X)^2$, call $C$ a constant bounding its operator norm. Again from the explicit formula of Proposition~\ref{prop:rewriting_Sinkhorn} and the Kantorovich duality~\eqref{eq:Kantorovich_duality}, 
\begin{align*}
    S_\varepsilon(\mu,\nu) & = \frac{\varepsilon}{2} \int_0^1 \int_0^1  \pair{\nu - \mu}{ G_{\mu_s,\mu_t} [\nu - \mu]} \diff t \diff s \\
    & \leq \frac{\varepsilon}{2} \left( \sup_{(\mu',\nu') \in \prm(X)^2} \| G_{\mu',\nu'}  \|_{\Lip_{c,0}^* \to \Lip_{c}/\R} \right) \;  W_c(\mu,\nu)^2. 
\end{align*}
The conclusion follows as the constant depends only on $c$, $X$ and $\varepsilon$.
\end{proof}

\begin{remark}[Why not a better remainder?]
\label{rem:why_not_better}
If we want the remainder in $o(\dS^2(\mu,\nu))$ in these bounds, we should work in the variable $\beta$ to try to measure all norms in $\Hil_c$. If $(\mu_t)_t$ is any reasonable path joining $\mu$ to $\nu$ (e.g. arc interpolation), then Proposition~\ref{prop:rewriting_Sinkhorn} should be valid with $\nu - \mu$ substituted with $\mudot_t$ or $\mudot_s$. Using the link between $\mudot_t$ and $\betadot_t$ recalled in Theorem~\ref{theorem:metric_tensor_in_betadot}, we can expect $S_\varepsilon$ to be
\begin{equation*}
    \frac{\varepsilon}{2} \int_0^1 \int_0^1  \pair{H_{\mu_s}\inv(\id + K_{\mu_s})[\betadot_s/\beta_s]}{(\id - K_{\mu_s,\mu_t} K_{\mu_t,\mu_s})\inv H_{\mu_s,\mu_t}H_{\mu_t}\inv(\id + K_{\mu_t})[\betadot_t/\beta_t]} \diff t \diff s.
\end{equation*}
The key obstacle is that we have the composition $H_{\mu_s,\mu_t}H_{\mu_t}\inv$. While it nicely simplifies when $t = s$, this is not the case here with the ``off-diagonal'' operators needed to rewrite the Sinkhorn divergence. That is, we do not know how to make sense of $H_{\mu_s,\mu_t}H_{\mu_t}\inv$ in Banach spaces built directly on $\Hil_c$. Similar issues arise when studying the ``Sinkhorn JKO'' scheme~\cite[Rem.~7.8]{Hardion2025}.  
\end{remark}

The proof of the tighter lower bound~\eqref{eq:Seps_asymp_lbound_dS} requires a reformulation of the Sinkhorn divergence. This reformulation can be seen as a strengthening of the inequality~\eqref{eq:Seps_lower_bound_Hc}, and its proof actually follows closely the one of \cite[Proposition 16]{sejourne2019sinkhorn}.
For \(\mu, \nu \in \prm(X)\), recall  \(\pi_{\mu,\nu} =  \exp(\frac{1}{\varepsilon} (\fmunu \oplus \gmunu - c)) \mu \otimes \nu\) is the entropic optimal transport plan between them. Denote \(\hat{\pi}_{\mu,\nu} = \exp(\frac{1}{\varepsilon} (\fmumu \oplus \fnunu - c)) \mu \otimes \nu\). Thanks to the Schrödinger system~\eqref{eq:Schroedinger_system} we can compute its marginals: 
\begin{align} 
    \begin{split}\label{eq:pihat_marginal}
        &(\proj_1)_\# \hat{\pi}_{\mu,\nu} = \exp((\fmumu - \fnunu) / \varepsilon) \mu = \alpha_\mu \cdot \beta_\nu, \\
        &(\proj_2)_\# \hat{\pi}_{\mu,\nu} = \exp((\fnunu - \fmumu) / \varepsilon) \nu = \alpha_\nu \cdot \beta_\mu.
    \end{split}
\end{align}

\begin{lemma} \label{lemma:Seps_KL_reformulation}
    The Sinkhorn divergence rewrites as 
    \begin{equation*}
        S_\varepsilon(\mu,\nu) = \frac{\varepsilon}{2} \norm{\beta_\nu - \beta_\mu}_{\Hil_c}^2 + \varepsilon \KL(\pi_{\mu,\nu} \vert \hat{\pi}_{\mu,\nu}).
    \end{equation*}
    In particular, for any \(\mu \otimes \nu\)-integrable function \(h \colon X \times X \to \R\), we have
    \begin{equation} \label{eq:lemma:Seps_KL_reformulation:lower_bound}
        S_\varepsilon(\mu, \nu) \geq \varepsilon \left( \frac{1}{2} \norm{\beta_\nu - \beta_\mu}_{\Hil_c}^2 +  \iint_{X \times X} h \diff \pi_{\mu,\nu} - \iint_{X \times X} (\exp(h) - 1) \diff \hat{\pi}_{\mu,\nu} \right)
    \end{equation}
    with equality if and only if \(h = \frac{1}{\varepsilon} ( (\fmunu - \fmumu) \oplus (\gmunu - \fnunu) )\).
\end{lemma}

\noindent Note that we recover~\eqref{eq:Seps_lower_bound_Hc} thanks to the non-negativity of $\KL$. 

\begin{proof}
    Recalling that $\pi_{\mu,\nu} = k_{\mu,\nu} \cdot (\mu \otimes \nu)$ and given the explicit expression of $k_{\mu,\nu}$~\eqref{eq:kmunu}, we have the well-known identity $\OT_\varepsilon(\mu, \nu) = \int_X \fmunu \diff \mu + \int_X \gmunu  \diff \nu$. Using it for the self-transports, but the original formulation~\eqref{eq:OTeps} for the ``cross''-transport,
    \begin{align*}
        S_\varepsilon(\mu,\nu) & = \iint_{X \times X} c \diff \pi_{\mu,\nu} + \varepsilon \KL(\pi_{\mu,\nu}| \mu \otimes \nu) - \int_X \fmumu \diff \mu - \int_X \fnunu \diff \nu  \\
        & =  \iint_{X \times X} (c - (\fmumu \oplus \fnunu))  \diff \pi_{\mu,\nu} + \varepsilon \KL(\pi_{\mu,\nu}| \mu \otimes \nu) \\
        &= \varepsilon \iint_{X \times X} - \log \left( \frac{\diff \hat{\pi}_{\mu,\nu}}{\diff \mu \otimes \nu} \right) \diff \pi_{\mu,\nu} + \varepsilon \iint_{X \times X} \log \left( \frac{\diff {\pi}_{\mu,\nu}}{\diff \mu \otimes \nu} \right) \diff \pi_{\mu,\nu} \\
        &= \varepsilon \iint_{X \times X} \log \left( \frac{\diff \pi_{\mu,\nu}}{\diff \hat{\pi}_{\mu,\nu}} \right) \diff \pi_{\mu,\nu} 
        = \varepsilon \KL(\pi_{\mu,\nu} | \hat{\pi}_{\mu,\nu}) + \varepsilon - \varepsilon \hat{\pi}_{\mu,\nu}(X \times X).
    \end{align*}
    In the last line we used the definition of $\KL$ for measures which are not necessarily of unit mass, see~\eqref{eq:def_kl}.
    On the other hand, using that \(\beta_\mu,\beta_\nu\) lie in the unit sphere of $\Hil_c$ and \(\pair{H_c[\eta]}{H_c[\sigma]}_{\Hil_c} = \pair{\eta \otimes \sigma}{k_c}\) for any \(\eta, \sigma \in \calM(X)\) we have
    \begin{align*}
        \norm{ \beta_\mu - \beta_\nu }_{\Hil_c}^2 
        &= 2 - 2 \langle H_c[ \exp(\fmumu/\varepsilon) \mu, H_c[ \exp(\fnunu/\varepsilon) \nu \rangle_{\Hil_c} \\
        &= 2 - 2 \iint_{X \times X} \exp((\fmumu \oplus \fnunu - c) / \varepsilon) \diff \mu \otimes \nu 
        = 2 - 2 \hat{\pi}_{\mu,\nu}(X \times X). 
    \end{align*}
    The first result follows. The second result is a convexity inequality: for any $y$, we have $p \log p - p + 1 \geq y p - (e^y - 1)$, with equality if and only if $y = \log p$. Substituting for $y$ any \(\mu \otimes \nu\)-integrable function \(h \colon X \times X \to \R\), for $p$ the density $\diff \pi_{\mu,\nu}/\diff \hat{\pi}_{\mu,\nu}$, and integrating this inequality with respect to $\pi_{\mu,\nu}$ gives~\eqref{eq:lemma:Seps_KL_reformulation:lower_bound}. Note that boundedness of \(k_{\mu,\nu} = \frac{\diff {\pi}_{\mu,\nu}}{\diff \mu \otimes \nu}\) ensures that the middle summand in~\eqref{eq:lemma:Seps_KL_reformulation:lower_bound} is finite, while the last term might become \(- \infty\), which makes the lower bound trivial. As $\pi_{\mu,\nu} = k_{\mu,\nu} \cdot (\mu \otimes \nu)$ one can check that the function $h$ corresponds to the \(\log\)-density of $\pi_{\mu,\nu}$ with respect to $\hat{\pi}_{\mu, \nu}$, thus attaining equality. 
\end{proof}

If $\mu$ and $\nu$ are very close, the expansion of the Schrödinger potentials (see Proposition~\ref{prop:time_derivative_of_potentials}) and Theorem~\ref{theorem:metric_tensor_in_betadot} yield 
\begin{equation*}
    f_{\mu, \nu} - \fmumu \approx  \varepsilon (\id - K_\mu^2)\inv H_\mu[f_{\nu,\nu} - f_{\mu,\mu}] \approx  \varepsilon (\id - K_\mu)\inv [(\beta_\nu - \beta_\mu) / \beta_\mu],      
\end{equation*}
though we will not try to make this expansion rigorous. In this case the optimal function $h$ in~\eqref{eq:lemma:Seps_KL_reformulation:lower_bound} should be approximated by 
\begin{equation*} 
    h(x,y) \approx  \psi(y) - \psi(x), \qquad \psi(x) \assign  (\id - K_\mu)\inv [(\beta_\nu - \beta_\mu) / \beta_\mu],
\end{equation*}
which is an ansatz the reader can recognize below. 

\begin{proof}[Proof of Theorem~\ref{theorem:estimates_with_Wc}: estimate~\eqref{eq:Seps_asymp_lbound_dS}]
We apply the lower bound~\eqref{eq:lemma:Seps_KL_reformulation:lower_bound} from Lemma~\ref{lemma:Seps_KL_reformulation} with \(h = (-\psi) \oplus \psi\) for \(\psi = (\id - K_{\mu})\inv [(\beta_\nu - \beta_\mu) / \beta_\mu]\). Note that \(\psi\) is only defined in \(\Cont(X)/\R\) but \(h\) is unambiguous. 

Recall the boundedness of \((\id - K_\mu)\inv\) on \(\Hil_\mu / \R\) from~\cite[Thm.~4.3]{RGSD}, and that division by $\beta_\mu$ is an isometry between $\Hil_c$ and $\Hil_\mu$. Thus, using~\eqref{eq:Lipc_Hmu_bound}, 
we can choose a representative of $\psi$ with
\begin{equation*}
\| \psi \|_{\Lip_c} \lesssim \| \psi \|_{\Hil_\mu} =  \| \psi \|_{\Hil_\mu / \R} \lesssim \| (\beta_\nu - \beta_\mu) / \beta_\mu \|_{\Hil_{\mu}} = \| \beta_\nu - \beta_\mu \|_{\Hil_c} \lesssim \dS(\mu,\nu).
\end{equation*}
In particular $\| h \|_\infty = \calO(\dS(\mu,\nu))$.
Using the marginal properties~\eqref{eq:pihat_marginal} and \(\exp(h) - 1 = h + h^2/2 + \calO(h^3)\) gives
\begin{equation*}
\iint_{X \times X} (\exp(h) - 1) \diff \hat{\pi}_{\mu,\nu} = \int_X \psi \diff (\alpha_\nu \beta_\mu - \alpha_\mu \beta_\nu) + \frac{1}{2} \iint_{X \times X} ((-\psi) \oplus \psi)^2 \diff \hat{\pi}_{\mu,\nu} + \calO(\dS^3(\mu,\mu)).
\end{equation*}
Plugging this in the lower bound~\eqref{eq:lemma:Seps_KL_reformulation:lower_bound} and with the marginal properties of $\pi_{\mu,\nu}$, 
\begin{multline}
\label{eq:zz_lower_bound1st_step}
\frac{1}{\varepsilon} S_\varepsilon(\mu,\nu) \geq \frac{1}{2} \norm{\beta_\nu - \beta_\mu}_{\Hil_c}^2 \\ + \int_X \psi \diff (\nu - \mu + \alpha_\mu \beta_\nu - \alpha_\nu \beta_\mu)  - \frac{1}{2} \iint_{X \times X} ((-\psi) \oplus \psi)^2 \diff \hat{\pi}_{\mu,\nu} + \calO(\dS^3(\mu,\mu)).
\end{multline}
It is \(\mu - \nu + \alpha_\mu \beta_\nu - \alpha_\nu \beta_\mu = (\alpha_\mu + \alpha_\nu) (\beta_\nu - \beta_\mu) = 2 \alpha_\mu (\beta_\nu - \beta_\mu) + o(\dS(\mu,\nu))\) in \(\Lip_c(X)^*\) by~\eqref{eq:p:dS_asymp:control_reminder}. Integrating against $\psi$ whose Lipschitz norm is $\calO(\dS(\mu,\nu))$, and as $\alpha_\mu = \mu / \beta_\mu$ we have 
\begin{equation}
\label{eq:zz_lower_bound_1st_integral}
\int_X \psi \diff (\nu - \mu + \alpha_\mu \beta_\nu - \alpha_\nu \beta_\mu) = 2 \int_X \psi (\beta_\nu - \beta_\mu) / \beta_\mu \diff \mu + o(\dS^2(\mu,\nu)). 
\end{equation}
For the second integral, note that
the map \((\mu,\nu) \mapsto \hat{\pi}_{\mu,\nu}\) is \textweakstar-to-\textweakstar~continuous. Hence it is \(\norm{\hat{\pi}_{\mu,\nu} - \pi_{\mu,\mu}}_{\Lip_c(X \times X)^*} = o(1)\) as \(\dS(\mu,\nu) \to 0\) uniformly over \(\mu,\nu \in \prm(X)\). Moreover $\| ((-\psi) \oplus \psi)^2 \|_{\Lip_c} = \calO(\dS^2(\mu,\nu))$ by~\eqref{eq:Lipc_bound_product}. Thus we have 
\begin{align}
\notag
\frac{1}{2} \iint_{X \times X} ((-\psi) \oplus \psi)^2 \diff \hat{\pi}_{\mu,\nu} & = \frac{1}{2} \iint_{X \times X} ((-\psi) \oplus \psi)^2 \diff \pi_{\mu,\mu} + o(\dS^2(\mu,\nu)) \\
\label{eq:zz_lower_bound_2nd_integral}
& = \pair{\psi}{(\id - K_{\mu})[\psi]}_{L^2(X,\mu)} + o(\dS^2(\mu,\nu)), 
\end{align}
the last equality following from the binomial formula and the marginal condition of \(\pi_{\mu,\mu}\).
Plugging~\eqref{eq:zz_lower_bound_1st_integral} and~\eqref{eq:zz_lower_bound_2nd_integral} into~\eqref{eq:zz_lower_bound1st_step}, and given our explicit choice of $\psi$,
\begin{align*}
\frac{1}{\varepsilon} S_\varepsilon&(\mu,\mu) \\
& \geq \frac{1}{2} \norm{\beta_\nu - \beta_\mu}_{\Hil_c}^2 + 2 \int_X \psi (\beta_\nu - \beta_\mu) / \beta_\mu \diff \mu - \pair{\psi}{(\id - K_{\mu})[\psi]}_{L^2(X,\mu)} + o(\dS^2(\mu,\mu)) \\
&= \frac{1}{2} \norm{\beta_\nu - \beta_\mu}_{\Hil_c}^2 + \pair{(\id - K_{\mu})\inv [(\beta_\nu - \beta_\mu) / \beta_\mu]}{(\beta_\nu - \beta_\mu) / \beta_\mu}_{L^2(X,\mu)} + o(\dS^2(\mu,\mu)) \\ 
&= \frac{1}{\varepsilon} \tg_{\mu} (\beta_\nu - \beta_\mu, \beta_\nu - \beta_\mu) + o(\dS^2(\mu,\mu)).
\end{align*}
The conclusion follows as $\tg_{\mu} (\beta_\nu - \beta_\mu, \beta_\nu - \beta_\mu) = \dS^2(\mu,\mu) + o(\dS^2(\mu,\mu))$, see Theorem~\ref{theorem:expansion_dS}.
\end{proof}

\subsection{Approximate geodesics in our metric}

We take the occasion to comment more on the meaning of Theorem~\ref{theorem:expansion_dS} and in particular~\eqref{eq:approximate_geodesics}.
In general, take a Riemannian manifold \(\calX\) with metric \(d_\calX\) and metric tensor \(\g\). We call a class of paths \(\setgiven{(\gamma^{x,y}_t)_t}{x,y \in \calX}\) approximately geodesic if 
\begin{equation}
	d_\calX(x,y)^2 = \g_{\gamma^{x,y}_t}(\dot\gamma^{x,y}_t,\dot\gamma^{x,y}_t) + o(d_\calX(x,y)^2) \quad \text{as } d_\calX(x,y) \to 0
\end{equation}
uniformly in \((x,y) \in \calX \times \calX\), \(t \in [0,1]\). 
Our estimates in Theorem~\ref{theorem:expansion_dS} precisely show that arc interpolation and vertical interpolations generate approximate geodesics. Note that it does not imply that they \emph{look} like the exact geodesic, as illustrated in the following example. 

\begin{example}[Constant-speed translation]\label{ex:comparison_exact_geodesics}
	For the squared distance cost function on \(\R^d\), translations \(\mu_t = (\id + t v)_\# \mu\) for \(v \in \R^d\) are geodesics of \(\dS\) with \(\dS(\mu_0,\mu_1) = \norm{v}\) (\cite[Thm.~5.16]{RGSD}). Here $\id+tv$ denotes the map $x \mapsto x + tv$. In particular, the \(\dS\)-geodesic between the Dirac measures \(\mu = \delta_{0}\) and \(\nu = \delta_{r}\) is a point mass traveling along the straight line \(x_t = t r\), which equals the Wasserstein-2 geodesic. In contrast, the support of the vertical and arc interpolation is always supported on the union of the supports \(\supp(\mu) \cup \supp(\nu)\) (see~\eqref{eq:explicit_formula_arc} for the arc interpolation). Despite this qualitative difference, both arc and vertical interpolation are approximately geodesic for $\dS$ as $r \to 0$.
\end{example}

In the previous example the Wasserstein geodesic is the exact geodesic for $\dS$. This is not always the case: Wasserstein geodesics are not necessarily approximate geodesics for $\dS$, even for the quadratic cost, as shown in the following example.

\begin{figure}
	\begin{center}
		
		\includegraphics[scale=0.48]{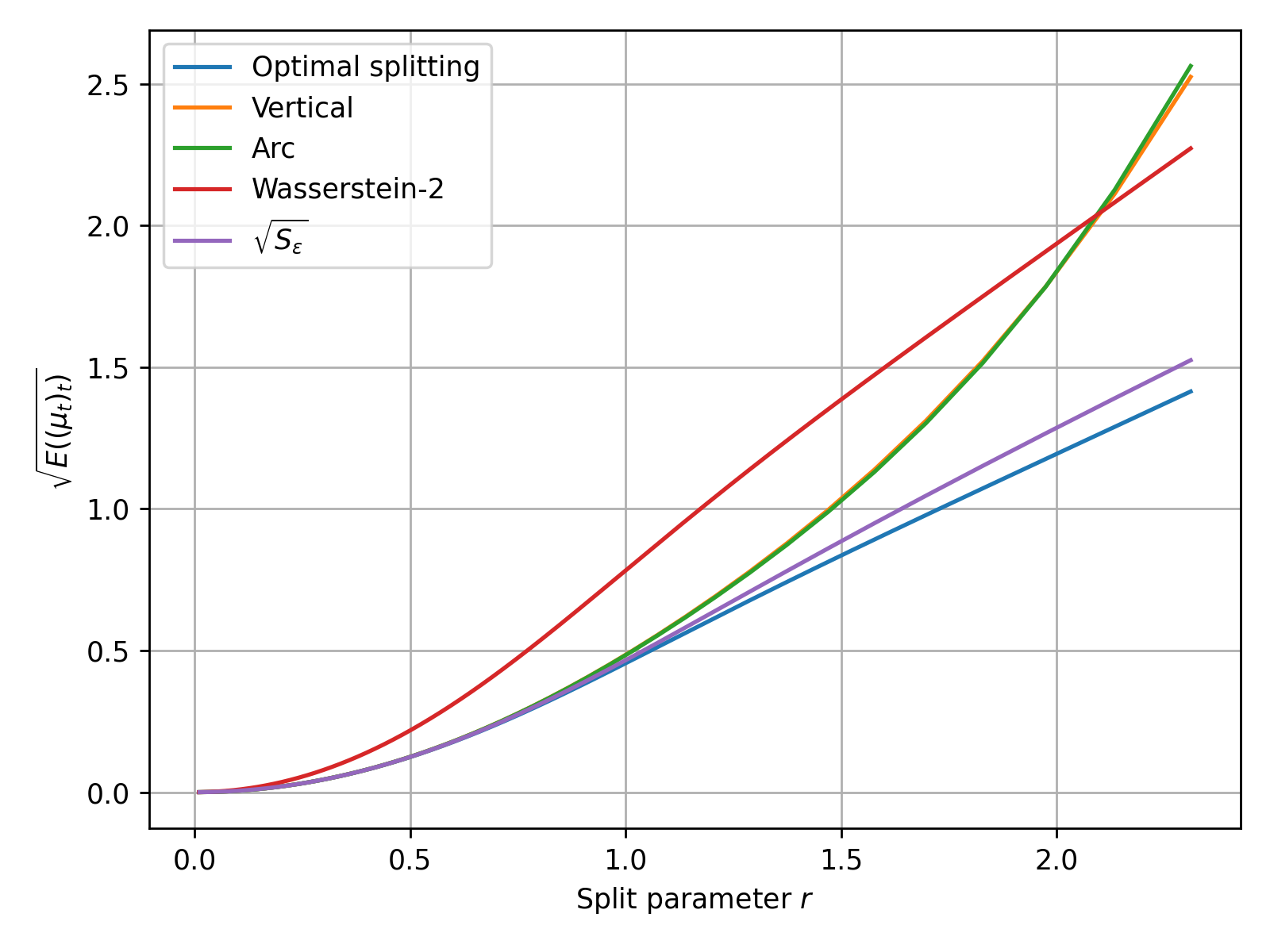}
		\includegraphics[scale=0.48]{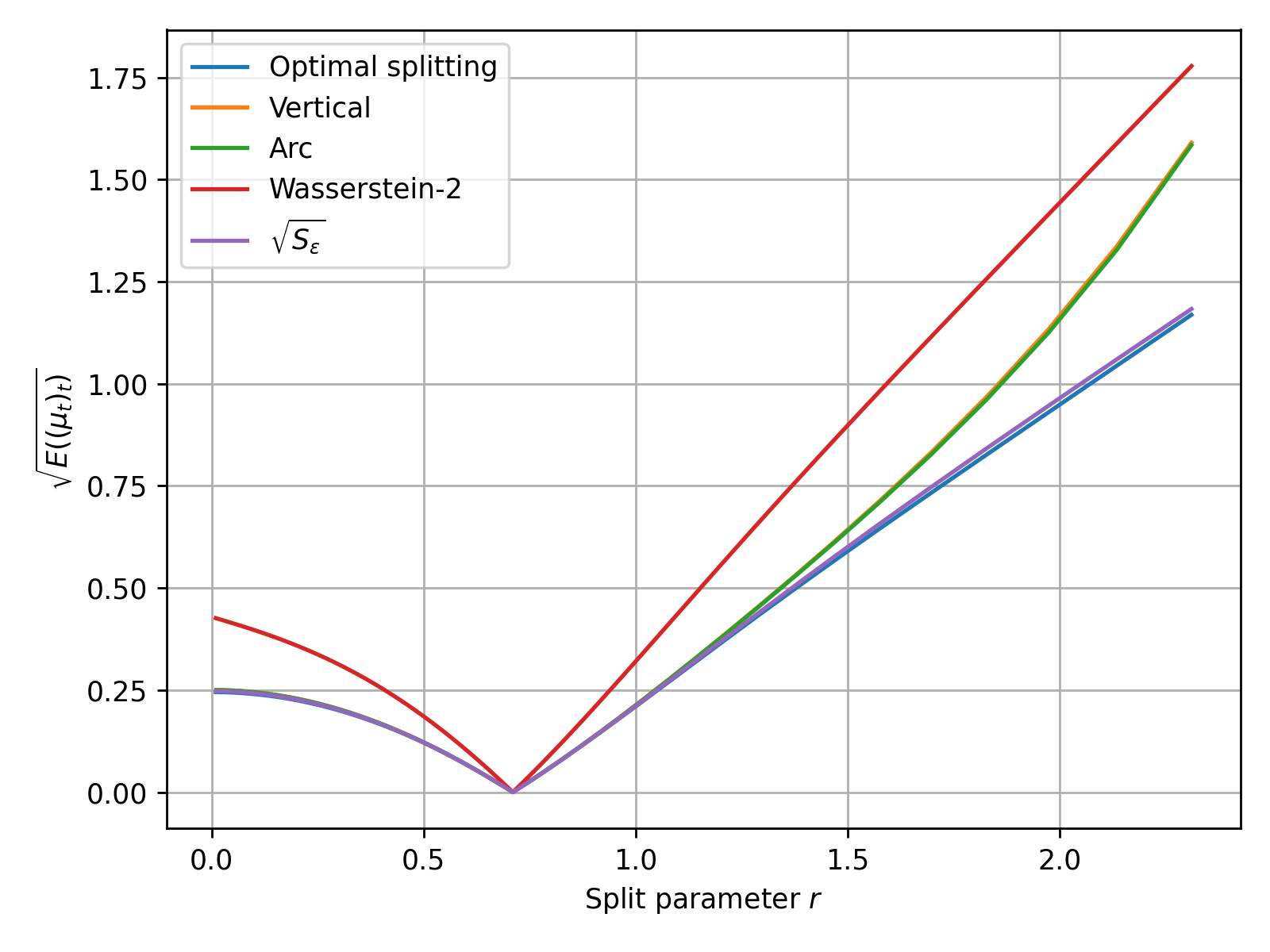}
		
		\caption{\label{fig:split_dirac_approx_energy} Splitting Dirac from Example~\ref{exm:W2_not_dS_geodesic_splitting_dirac}. Left: Upper bounds on distance by $\sqrt{E((\mu_t)_t)}$ between $\mu_0 = \delta_0$ and $\mu_1 = \nu_r = \frac{1}{2} (\delta_{-r} + \delta_r)$ for varying parameter $r > 0$ estimated from different paths for the cost function \(c(x,y) = \abs{x-y}^2\) and \(\varepsilon=1\). Blue: Optimal path of the form $\mu_t = \nu_{r(t)}$. Orange/green: Vertical/arc interpolation (almost perfect overlap). Red: Wasserstein-2 geodesic. Purple: Sinkhorn divergence.
			Right: Same comparison for $\mu_0 = \nu_{r_0}$ with $r_0 = 1/\sqrt{2}$.
			Vertical and arc interpolation give almost identical results, matching the Sinkhorn divergence and the optimal splitting curve asymptotically, while the Wasserstein-2 geodesic has a significantly higher energy. 
		}
	\end{center}
\end{figure}

\begin{example}[Splitting Dirac] \label{exm:W2_not_dS_geodesic_splitting_dirac}
	Consider \(X \subset \R\) with the cost \(c(x,y) = \abs{x-y}^2\). We take $\mu = \delta_0$ and $\nu_r \assign \frac{1}{2} (\delta_{-r} + \delta_r)$. It was already considered in \cite[Sec.~7.2]{RGSD}. In addition to $(\mu^A_t)_t$ and $(\mu^V_t)_t$ the arc and vertical interpolations, denote by $\mu^W_t = \nu_{r t}$ for $t \in  [0,1]$ the Wasserstein geodesic. We claim that as $r \to 0$ it is
	\begin{align}
		\label{eq:S_expansion_example}
		S_\varepsilon(\mu,\nu_r) & = \frac{1}{\varepsilon} \cdot r^4 + \calO(r^8), \\
		\label{eq:dS_expansion_example}
		\dS(\mu,\nu_r)^2 & = \frac{1}{\varepsilon} \cdot r^4 + o(r^4), \\
		\label{eq:Wasserstein_expansion_example}
		E((\mu^W_t)_t) & = \frac{4}{3 \varepsilon} \cdot r^4 +  \calO(r^6), \\
		\label{eq:Wc_example}
		W_c(\mu,\nu_r)^2 & = \frac{2}{\varepsilon} \cdot r^2 + \calO(r^4). 
	\end{align}
	It shows that $(\mu^W_t)_t$ is not an approximate geodesic as $r \to 0$. It also shows that we cannot expect a bound of the type $W_c \lesssim \dS$ to hold. 
	
	The computation of $S_\varepsilon(\nu_{s},\nu_r)$ can be found in \cite[Lem.~7.5]{RGSD}. When $s = 0$ it simplifies to 
	\begin{equation*}
		S_\varepsilon(\mu,\nu_r) = \frac{\varepsilon}{2} \log \cosh \left( \frac{2r^2}{\varepsilon} \right).
	\end{equation*}
	The expansion~\eqref{eq:S_expansion_example} follows. To prove~\eqref{eq:dS_expansion_example} we use Theorem~\ref{theorem:expansion_dS}. Indeed, we have $f_{\mu,\mu}(x) = x^2$ so that $k_{\mu,\mu}(x,y) = \exp(2 xy/\varepsilon)$. In particular $K_\mu[\phi] = \phi(0) \ones_X$, meaning $K_\mu = 0$ and $(\id - K_\mu^2)\inv = \id$ on $\Cont(X)/\R$. We deduce, as $r \to 0$ 
	\begin{align*}
		\g_\mu(\nu_r - \mu, \nu_r - \mu) & = \frac{\varepsilon}{2} \langle \nu_r - \mu, H_\mu[\nu_r - \mu] \rangle \\ 
		& = \frac{\varepsilon}{2} \iint_{X \times X} k_{\mu,\mu}(x,y) \diff (\nu_r(x) - \mu(x)) \diff (\nu_r(y) - \mu(y)) \\
		& = \frac{\varepsilon}{2} \left( \cosh \left( \frac{2 r^2}{\varepsilon} \right) - 1 \right).  
	\end{align*}
	With~\eqref{eq:expansion_dS_mu} we deduce the expansion~\eqref{eq:dS_expansion_example}.
	On the other hand consider the Wasserstein geodesic: it reads $\mu^W_t = \nu_{r t}$ for $t \in  [0,1]$. The evaluation of the metric tensor for this derivative was already done in \cite[7.3]{RGSD}: with \(\kappa = k_c(r,-r) = \exp(-4r^2/\varepsilon)\), we have 
	\begin{equation*}
		\g_{\nu_r} \left( \frac{1}{2} (\delta'_{r} - \delta'_{-r}), \frac{1}{2} (\delta'_{r} - \delta'_{-r}) \right) 
		= 1 + \frac{\kappa}{1+\kappa} \left( \frac{1}{1+\kappa}\frac{8 r^2}{\varepsilon} - 2 \right) = \frac{4 r^2}{\varepsilon} + \calO(r^4) \quad \text{as } r \to 0.
	\end{equation*}
	Thus as $\mudot^W_{t} = r \frac{1}{2} ( \delta_{rt}' - \delta_{-rt}')$ we have $\g_{\mu^W_t}(\mudot^W_{t},\mudot^W_{t}) = \frac{4}{ \varepsilon} \cdot r^4 t^2 + \calO(r^6)$. In particular the latter expression is not (approximately) constant in time, thus it cannot look like a geodesic. We confirm it by computing 
	\begin{equation*}
		E((\mu^W_t)_t) = \int_0^1 \g_{\mu^W_t}(\mudot^W_{t},\mudot^W_{t})  \diff t = \frac{4}{3 \varepsilon} \cdot r^4 +  \calO(r^6),
	\end{equation*}
	which is the third claim~\eqref{eq:Wasserstein_expansion_example}. The last claim~\eqref{eq:Wc_example} is immediate as we can use the Kantorovich formulation of the one-Wasserstein distance: $W_c(\mu,\nu_r) = d_c(r,0)$ with $d_c$ as in~\eqref{eq:def_dc}. 
	
	Interestingly, we can build a curve which still ``moves'' mass around and is asymptotically optimal for $\dS$. For $\gamma : [0,1] \to [0,1]$, consider $\mu^\gamma_t = \nu_{\gamma_t r}$: this is the Wasserstein geodesic traveled at a non-constant speed. The same computation as in the Wasserstein case leads to 
	\begin{equation*}
		E((\mu^\gamma_t)_t) = \int_0^1 \g_{\mu^\gamma_t}(\mudot^\gamma_{t},\mudot^\gamma_{t})  \diff t = \frac{4 r^4}{\varepsilon} \int_0^1 \gamma_t^2 \dot{\gamma}_t^2 \diff t +  \calO(r^6).
	\end{equation*}
	This quantity can be optimized in $\gamma$, leading to $\gamma_t = \sqrt{t}$. For this optimal $\gamma$, $E((\mu^\gamma_t)_t) = \frac{1}{\varepsilon} r^4 + \calO(r^6)$, which is equivalent to $\dS(\mu,\nu_r)^2$. 
	We refer to Figure~\ref{fig:split_dirac_approx_energy} for an illustration.
\end{example}

\section{Gamma convergence of Sinkhorn chain}  \label{sec:Gamma_convergence}

In this section we want to relate the functionals \eqref{eq:intro:dS} and \eqref{eq:intro:Sinkhorn_chain} and their minimizers,
\begin{equation*}
	E((\mu_t)_t) = \int_0^1 \g_{\mu_t}(\mudot_t,\mudot_t) \diff t, 
	\quad \tn{and} \quad \ol{E}^N(\mu_0,\mu_1,\dots,\mu_N) \assign N \cdot \sum_{k=0}^{N-1} S_\varepsilon(\mu_k, \mu_{k+1})
\end{equation*}
for a path \((\mu_t)_{t \in [0,1]} \in \Pall\) (as in Definition~\ref{def:distance}) and an \((N+1)\)-tuple \((\mu_k)_{k=0}^N\) in \(\prm(X)\). 

We rely on the framework of $\Gamma$-convergence to guarantee convergence of minima and minimizers of the functions, see~\cite{Braides2005} for a thorough introduction on the topic. We recall the setting: take $Z$ a first countable topological space and \(F, F^N \colon Z \to \R \cup \{\infty\}\) for \(N \in \N\). Then \(F^N \xrightarrow{\Gamma} F\) as \(N \to \infty\) if the following two conditions are satisfied:
\begin{enumerate}
	\item Whenever \(z^N \to z\) in \(Z\) as \(N \to \infty\), then \(F(z) \leq \liminf_{N\to\infty} F^N(z^N)\),
	\item For every \(z \in Z\) there is a sequence of \(z^N \in Z\), called the recovery sequence, with \(z^N \to z\) in \(Z\) as \(N \to \infty\) such that \(F(z) \geq \limsup_{N\to\infty} F^N(z^N)\).
\end{enumerate}    
The notion of \(\Gamma\)-convergence requires all functionals to be defined on the same space, which is a priori not the case for $E$ and $\ol{E}^N$. Therefore, for \((\mu_t)_t \in \Pall\) we write 
\begin{equation*}
	E^N((\mu_t)_t) \assign \ol{E}^N\left((\mu_\frac{k}{N})_{k=0}^N \right) = N \cdot \sum_{k=0}^{N-1} S_\varepsilon(\mu_{\frac{k}{N}}, \mu_{\frac{k+1}{N}})
\end{equation*}
when for all \(k = 0, \dots, N-1\) the segment \((\mu_t)_{t \in [k/N,(k+1)/N]}\) in \(\prm(X)\) corresponds to the arc interpolation on the unit sphere in \(\Hil_c\) between \(\beta_{k/N}\) and \(\beta_{(k+1)/N}\) (see Definition~\ref{def:vert_arc}), and we set \(E^N((\mu_t)_t) = \infty\) otherwise. Similarly, when \((\mu_t)_t \not\in \Pall\) we set \(E((\mu_t)_t) = \infty\).

\begin{theorem} \label{theorem:gamma_convergence}
	There holds \(E^N \xrightarrow{\Gamma} E\) on  \(\Cont([0,1];\prm(X))\). Furthermore, any sequence $(\mu^N_t)_t$ with $\sup_N E^N((\mu^N_t)_t) < \infty$ is relatively compact in \(\Cont([0,1];\prm(X))\), and the recovery sequence to $(\mu_t)_t \in \Pall$ can be chosen such that \(\mu^N_0 = \mu_0\) and \(\mu^N_1 = \mu_1\) for all \(N \in \N\).
\end{theorem}

\begin{remark}
	The topology on \(\Cont([0,1];\prm(X))\) is that of uniform convergence in \((\prm(X),\dS)\). Since \(\prm(X)\) is compact, all metrics metrizing the \textweakstar~topology, be it \(\dS\), or \(W_2\), are uniformly continuous with respect to one another and thus yield the same notion of uniform convergence. Equivalently, this is the compact-open topology. 
\end{remark}

\begin{remark}
	The restriction to piecewise arc interpolation in the definition of \(E^N\) is somewhat arbitrary. One could, for instance, also use piecewise geodesic interpolation (with respect to \(\dS\)), or vertical interpolation.
	The important point is that either of these forms of interpolations preserves uniform convergence. All of our results could therefore be expressed just the same with any of the listed forms of interpolation. One could even use piecewise constant interpolation with appropriate modification of the functional space where $E$ and $E^N$ are defined.
\end{remark}

\begin{corollary} \label{cor:chain_energy_approx}
	The Sinkhorn distance \(\dS\) can be obtained as the limit of the \(S_\varepsilon\)-chain approximation: for any \(\mu,\nu \in \prm(X)\)
	\begin{equation}
		\label{eq:chain_approx}
		\dS^2(\mu,\nu) = \lim_{N \to \infty} \inf \, \setgiven{N \cdot \sum_{k=0}^{N-1} S_\varepsilon(\mu_k,\mu_{k+1})}{\mu_0,\ldots,\mu_N \in \prm(X), \, \mu_0=\mu,\, \mu_N = \nu}
	\end{equation}
	Further, if for \(N \in \N\), \((\mu^N_k)_{k=0}^N\) are discrete paths with \(\mu^N_0 = \mu\), \(\mu^N_N = \nu\) and $\ol{E}^N((\mu^N_k)_k)$ converges to $\dS^2(\mu,\nu)$, 
	then the piecewise arc interpolation of \((\mu^N_k)_{k=0}^N\) converges (up to subsequence) in \(\Cont([0,1];\prm(X))\) to a path \((\mu_t)_t\) which is a geodesic between $\mu$ and $\nu$.
\end{corollary}

\begin{proof}
	As the recovery sequence in Theorem~\ref{theorem:gamma_convergence} can be chosen to have the prescribed start and end points, we also get $\Gamma$-convergence of $E^N$ to $E$ on the set $\Pall(\mu,\nu) \subseteq \Cont([0,1],\prm(X))$ (i.e.~those connecting \(\mu\) to \(\nu\)).
	
	From the $\Gamma$-limsup, we know that $\inf E^N$ is uniformly bounded from above. Thus, if $((\mu^N_t)_t)_N$ are curves with $E^N((\mu^N_t)_t) - \inf_{\Pall(\mu,\nu)} E^N \to 0$ as $N \to \infty$, from the relative compactness (again Theorem~\ref{theorem:gamma_convergence}) we know that we can extract a subsequence which converges. The theory of $\Gamma$-convergence implies both convergence of the minimal values (this is~\eqref{eq:chain_approx}) and that the limit of any converging subsequence of $(\mu^N_t)_t$ is a geodesic (proof of~\cite[Prop.~7.5]{santambrogio2023course}).  
\end{proof}

\begin{remark}
	Alternatively to the Sinkhorn chain \(E^N\), one could also work directly with metric tensors and define
	\begin{align*}
		\calE^N ( \mu_0, \mu_1, \dots, \mu_N ) & \assign N \cdot \sum_{k=0}^{N-1} \g_{\mu_k} ( \mu_{k+1} - \mu_k, \mu_{k+1} - \mu_k ), \\
		\tilde{\calE}^N ( \mu_0, \mu_1, \dots, \mu_N ) & \assign N \cdot \sum_{k=0}^{N-1} \tg_{\mu_k} ( \beta_{k+1} - \beta_k, \beta_{k+1} - \beta_k ).
	\end{align*}
	We believe that one can obtain \(\Gamma\)-convergence of (a continuous time version of) these functionals to \(E\) in the topology of \(\Cont([0,1];\prm(X))\) by mimicking the approach of~\cite[Thm.~4.8]{RumpfWirth2015}. The key point is that the remainders in Theorem~\ref{theorem:expansion_dS} are better behaved than the one in Theorem~\ref{theorem:estimates_with_Wc}, which would simplify the technical aspects. 
	However, likely $\calE$ and $\tilde{\calE}^N$ are numerically significantly less tractable than the Sinkhorn chain \(E^N\) due to the dependence of the metric tensors on the footpoint.
\end{remark}

Sections~\ref{sec:gamma_liminf} and~\ref{sec:gamma_limsup} are dedicated to the proof of this theorem. The \(\Gamma\)-liminf is shown via direct estimates, while the proof for the \(\Gamma\)-limsup relies on the spatial discretization step from Section~\ref{sec:projection} that allows us to reduce the problem to a finite base space.

\begin{remark}[Joint limit]
	\label{rmk:joint_limit}
	Interestingly our proof techniques enable us to prove the joint spatial and temporal limit without any constraint on the ratio between the spatial and temporal step size. Take $(x_m)_m$ a sequence dense in $X$ and set $X_M = \{ x_1, \ldots, x_M \}$. Consider $\mu^M$, $\nu^M$ probabilities over $X_M$ which converge to $\mu,\nu \in \prm(X)$ as $M \to \infty$. We consider the problem
	\begin{equation}
		\label{eq:joint_limit}
		\min \left\{ \left. N \cdot \sum_{k=0}^{N-1} S_\varepsilon(\mu_k,\mu_{k+1}) \; \right| \; \mu_0,\ldots,\mu_N \in \prm(X_M), \, \mu_0=\mu^M,\, \mu_N = \nu^M \right\},
	\end{equation}
	indexed by $M$ the spatial resolution and $N$ the temporal resolution. We claim that the value of the problem~\eqref{eq:joint_limit} converges to $\dS(\mu,\nu)$ as $N,M \to \infty$ jointly, and suitable interpolations of the minimizers of~\eqref{eq:joint_limit} converge to geodesics. 
	
	Indeed the $\Gamma$-liminf step is identical: Theorem~\ref{theorem:gamma_liminf} below still applies. For the $\Gamma$-limsup, assuming that we take the limit $N \to \infty$ and $M = M'_N \to \infty$ (at an arbitrary rate), then in the proof of Theorem~\ref{theorem:gamma_limsup} it is enough to take $M_N \leq M'_N$ (the sequence $(M_N)_N$ in this proof can grow at an arbitrarily slow rate) to guarantee that the recovery sequence is supported on $\prm(X_{M_N}) \subseteq \prm(X_{M'_N})$.   
\end{remark}

\subsection{Gamma-liminf} \label{sec:gamma_liminf}

\begin{theorem}[{Liminf in Theorem~\ref{theorem:gamma_convergence}}] \label{theorem:gamma_liminf}
	For any \(N \in \N\) let \((\mu^N_k)_k\) be in \(\prm(X)^{N+1}\) and assume that \(\sup_{N \in \N} \ol{E}^N((\mu^N_k)_k) < \infty\). Then, up to a subsequence, the piecewise arc interpolation of \((\mu^N_k)_k\) converges in \(\Cont([0,1];\prm(X))\) to a curve \((\mu_t)_t\) and 
	\begin{equation*}
		E((\mu_t)_t) \leq \liminf_{N \to \infty} \ol{E}^N((\mu^N_k)_k).
	\end{equation*}
\end{theorem}

\begin{proof}
Let \((\tilde{\mu}^N_t)_t\) be the arc interpolation of $\mu^N_k$. We use first the lower bound~\eqref{eq:Seps_asymp_lbound_dS} in Theorem~\ref{theorem:estimates_with_Wc}, followed by the estimate~\eqref{eq:approximate_geodesics} in Theorem~\ref{theorem:expansion_dS}, with the appropriate time-rescaling: we get two modulus of continuity \(\omega_1\), $\omega_2$ such that
\begin{align*}
\ol{E}^N((\mu^N_k)_k)  = N \sum_{k=0}^{N-1} S_\varepsilon(\mu^N_{k},\mu^N_{k+1})  & \geq \left( 1 - \omega_1 \left( \max_k \dS(\mu^N_{k},\mu^N_{k+1} ) \right) \right) N \sum_{k=0}^{N-1} \dS^2(\mu^N_{k},\mu^N_{k+1})  \\
& \geq \left( 1 - \omega_2 \left( \max_k \dS(\mu^N_{k},\mu^N_{k+1}) \right) \right) E((\tilde{\mu}^N_t)_t).
\end{align*}
Moreover, recalling the estimate~\eqref{eq:bounds_distance}, we have $\ol{E}^N((\mu^N_k)_k) \gtrsim N \dS(\mu^N_{k},\mu^N_{k+1})^2$ for any $k$. Thus $\max_k \dS(\mu^N_{k},\mu^N_{k+1}) \to 0$, so that  
\begin{equation*} 
\infty > \liminf_{N \to \infty} \ol{E}^N((\mu^N_k)_k) \geq  \liminf_{N \to \infty} E((\tilde{\mu}^N_t)_t).
\end{equation*}
Thus, up to a subsequence, Theorem~\ref{theorem:coercivity} shows that \((\tilde{\mu}^N_t)_t\) converges in \(\Cont([0,1];\prm(X))\) to some path \((\mu_t)_t\) and the corresponding path \((\tilde{\beta}^N_t)_t\) converges weakly in \(\Hil^1((0,1);\Hil_c)\) to the corresponding \((\beta_t)_t\). Weak lower semi-continuity of the energy functional \(E\) from Theorem~\ref{theorem:Energy_lsc} shows that the right hand side in the last display is greater than $E((\mu_t)_t)$.
\end{proof}

\subsection{Gamma-limsup} \label{sec:gamma_limsup}

\begin{theorem}[{Limsup in Theorem~\ref{theorem:gamma_convergence}}] \label{theorem:gamma_limsup}
	For any admissible path \((\mu_t)_t \in \Pall\), we can build discrete curves \((\mu^N_k)_k\) in \(\prm(X)^{N+1}\) for \(N \in \N\) with \(\mu^N_0 = \mu_0\), \(\mu^N_N = \mu_1\), such that the piecewise arc interpolation of \((\mu^N_k)_k\) converges in \(\Cont([0,1];\prm(X))\) and
	\begin{equation*}
		E((\mu_t)_t) \geq \lim_{N \to \infty} N \cdot \sum_{k=0}^{N-1} S_\varepsilon(\mu^N_k, \mu^N_{k+1}).
	\end{equation*} 
\end{theorem}

The proof is divided into two steps: we first show that it is the case if $(\mu_t)_t$ is supported on a discrete space \(X_M \subseteq X\). Section~\ref{sec:projection} provides a suitable way to project \(\prm(X)\) to \(\prm(X_M)\). The estimates from Section~\ref{sec:local_expansion} allow us to modify the discrete curves to preserve the start and end points.

\begin{lemma} \label{lemma:limsup_discrete}
	Suppose that \(X = \{x_1,\dots,x_M\}\) is finite and $(\mu_t)_t \in \Pall$. 
	Then it holds \(\lim_{N \to \infty} \ol{E}^N((\mu_{k/N})_k) = E((\mu_t)_t)\).
\end{lemma}

\begin{proof}
	Using the expansion of the Sinkhorn divergence in Theorem~\ref{theorem:estimates_with_Wc}, the equivalence of the metric tensor with $\dS$ in Theorem~\ref{theorem:expansion_dS}, and the comparison between distances in~\eqref{eq:bounds_distance},
	\begin{equation*}
		S_\varepsilon(\mu,\nu) = \dS^2(\mu,\nu) + o(W_c(\mu,\nu)^2)
	\end{equation*}
	as $W_c(\mu,\nu) \to 0$, with a small $o$ uniform over $\prm(X)^2$. Now $W_c \leq C \dS$ with $C$ depending on $c$, $X$ and $\varepsilon$. Indeed, in this case $W_c$ is bounded, up to a constant, by the total variation norm \cite[Thm.~6.15]{Villani2009}. Moreover, still in this case, the total variation norm is equivalent to the norm induced by the metric tensor, uniformly in the footpoint, and therefore also equivalent to the induced Riemannian distance $\dS$ (we have a finite-dimensional Riemannian manifold!). Thus, on this discrete space we have the stronger expansion $S_\varepsilon(\mu,\nu) = \dS^2(\mu,\nu) + o(\dS^2(\mu,\nu))$. We obtain 
	\begin{equation*} 
		N \sum_{k=0}^{N-1} S_\varepsilon(\mu_\frac{k}{N}, \mu_\frac{k+1}{N}) = N \sum_{k=0}^{N-1}  \dS^2(\mu_\frac{k}{N}, \mu_\frac{k+1}{N}) \left( 1   +     \omega(\dS(\mu_\frac{k}{N}, \mu_\frac{k+1}{N}) ) \right) 
	\end{equation*}
	with $\omega$ the rate function of the \(o\)-term. 
	As \(N \cdot \dS(\mu_\frac{k}{N}, \mu_\frac{k+1}{N})^2 \leq \int_{k/N}^{(k+1)/N} \g_{\mu_t}(\mudot_t,\mudot_t) \diff t\), summing over \(N\) yields
	\begin{equation*}
		N \sum_{k=0}^{N-1} S_\varepsilon(\mu_\frac{k}{N}, \mu_\frac{k+1}{N})  \leq \int_0^1 \g_{\mu_t}(\mudot_t,\mudot_t) \diff t \cdot \left(1 + \omega \left( \max_{\abs{t-s} \leq \frac{1}{N}} \dS(\mu_s,\mu_t)^2 \right) \right). 
	\end{equation*}
	One obtains \(\limsup_{N \to \infty} E^N((\mu_t)_t) \leq E((\mu_t)_t)\) from uniform continuity of the path \((\mu_t)_t\) with respect to \(\dS\). The \(\Gamma\)-liminf condition shown in Theorem~\ref{theorem:gamma_liminf} yields equality.
\end{proof}

\begin{proof}[{Proof of Theorem~\ref{theorem:gamma_limsup}}]
	Let \((\mu_t)_t \in \Pall\) with corresponding \((\beta_t)_t \in \Hil^1((0,1);\Hil_c)\). We construct the recovery sequence for the \(\Gamma\)-limsup condition. Take $(x_m)_m$ a sequence dense in $X$ and set $X_M = \{ x_1, \ldots, x_M \}$. 
	
	\emph{1st step: building the sequence.}
	Define the renormalized projected path $\mu^M_t = R_M(\mu_t)$ as in Section~\ref{sec:projection}.  
	The measures \(\mu^M_t\) are supported on \(X_M = \{x_1,x_2,\dots,x_M\}\).
	By defining 
	\begin{equation*}
		E_{N,M} \assign \ol{E}^{N}(\mu^M_0, \mu^M_{1/N}, \dots, \mu^M_{N/N}),    
	\end{equation*}
	Lemma~\ref{lemma:limsup_discrete} (for the limit in $N$) and Corollary~\ref{crl:cv_energy_M} (for the limit in $M$) show that
	\begin{equation}
		\label{eq:joint_limit_ENM}
		\lim_{M \to \infty} \lim_{N \to \infty} E_{N,M} = E((\mu_t)_t),
	\end{equation}
	where the limits are taken in this order. We modify the path to make the start and end points line up with \(\mu_0\) and \(\mu_1\), while preserving the convergence of the energy. We will do it by using a vertical interpolation near the endpoints. 
	For parameters $N$, $M$ and $K \leq N/2$ we define 
	\begin{equation} \label{eq:p:gamma_limsup_path}
		\begin{cases}
			\tilde{\mu}^{N,K,M}_k = \frac{K - k}{K} \mu_0 + \frac{k}{K} \mu^{M}_0 & \tn{for } 0 \leq k \leq K, \\[10pt]
			\tilde{\mu}^{N,K,M}_{k} = \mu^{M}_{(k - K) / (N - 2 K)} & \tn{for } K \leq k \leq N - K, \\[10pt]
			\tilde{\mu}^{N,K,M}_{k} = \frac{N - k}{K} \mu^{M}_1 + \frac{k - (N - K)}{K} \mu_1 & \tn{for } N - K \leq k \leq N.
		\end{cases}
	\end{equation}
	That is, we squeeze the spatially discrete curve $\mu^{M}$ only in $\{K,\ldots,N-K\}$, and use $\{1,\ldots,K\}$ and $\{N-K,\ldots,N\}$ to put a vertical interpolation between $\mu_0$ and $R_M(\mu_0)$ on one end, $R_M(\mu_1)$ and $\mu_1$ at the other end. 
	
	\emph{2nd step: Energy of the recovery sequence.}
	Splitting the sum defining the discrete energy \(\ol{E}^N(\tilde{\mu}^{N,K,M}_0,\dots,\tilde{\mu}^{N,K,M}_{N})\) according to~\eqref{eq:p:gamma_limsup_path}, we notice that the middle part recovers \(E^{N-2K,M}\) up to a scaling factor $N/(N-2K)$.
	For the endpoints let us write $w_M = \max \{W_c(\mu_0,\mu^M_0), W_c(\mu_1, \mu^M_1)\}$. By the convergence $\mu^M_0 \to \mu_0$ and $\mu^M_1 \to \mu_1$ in $\prm(X)$, we have $w_M \to 0$ as $M \to \infty$. We use the estimate \(W_c(\tilde{\mu}^{N,K,M}_k, \tilde{\mu}^{N,K,M}_{k+1}) =  W_c(\mu_0,\mu^{M}_0)/K \leq w_M / K\) which comes from the Kantorovich duality~\eqref{eq:Kantorovich_duality}. Thus with Theorem~\ref{theorem:estimates_with_Wc}, we can bound
	\begin{equation*}
		N \sum_{k=0}^{K - 1} S_\varepsilon( \tilde{\mu}^{N,K,M}_k, \tilde{\mu}^{N,K,M}_{k+1} ) \leq \frac{N w_M^2}{K} .
	\end{equation*}
	An analogous estimate holds at the other end.
	Putting the three pieces together, we have
	\begin{equation*}
		\ol{E}^N(\tilde{\mu}^{N,K,M}_0,\dots,\tilde{\mu}^{N,K,M}_{N}) \leq \frac{ 2N w_M^2}{K} + \frac{N}{N-2K} E_{N-2K,M}. 
	\end{equation*}
	We choose $M$ and $K$ as a function of $N$. Specifically, from~\eqref{eq:joint_limit_ENM} we can find $(M_N)_{N}$ a sequence diverging to $\infty$ slowly enough such that $E_{A_N,M_N} \to E((\mu_t)_t)$ as $N \to \infty$ for any sequence $(A_N)_N$ larger than $N/2$. Second, we take a sequence $(K_N)_{N}$ with $N w_{M_N}^2 \ll K \ll N$, e.g. $K_N = \lceil w_{M_N}  N \rceil$, which works (for $N$ large enough) as $w_{M_N} \to 0$. With this asymptotic behavior, as $A_N = N - 2K_N$ is larger than $N/2$ for $N$ large enough, in the previous estimate we obtain 
	\begin{equation}
		\label{eq:estimate_limsup_EN}
		\limsup_{N \to \infty} \ol{E}^N(\tilde{\mu}^{N,K_N,M_N}_0,\dots,\tilde{\mu}^{N,K_N,M_N}_{N}) \leq E((\mu_t)_t).
	\end{equation}
	The piecewise arc interpolation of $(\tilde{\mu}^{N,K_N,M_N}_k)_k$ is our candidate recovery sequence.
	
	\emph{3rd step: the recovery sequence converges to the original one.}
	Let us write $(\hat{\mu}^{N}_t)_t$ the piecewise arc interpolation of $(\tilde{\mu}^{N,K_N,M_N}_k)_k$. We need to show that it converges to $(\mu_t)_t$ in $\Cont([0,1];\prm(X))$ as $N \to \infty$. From the estimate~\eqref{eq:estimate_limsup_EN} and the coercivity result of the $\Gamma$-liminf (Theorem~\ref{theorem:gamma_liminf}), we know that $(\hat{\mu}^{N}_t)_t$ converges, up to extraction, to some limit $(\hat{\mu}_t)_t \in \Cont([0,1];\prm(X))$. It is enough to show that the limit coincides with the original path $(\mu_t)_t$. Fix $t \in (0,1)$, and consider $(k_N)_N$ such that $k_N/N \to t$. As $K_N / N \to 0$, we have $K_N\leq k_N \leq N- K_N$ for $N$ large enough. We write 
	\begin{multline*}
		\dS(\hat{\mu}_t,\mu_t) \leq \dS(\hat{\mu}_t,\hat{\mu}_{k_N/N}) + \dS(\hat{\mu}_{k_N/N},\hat{\mu}^N_{k_N/N} ) \\
		+ \dS(\hat{\mu}^N_{k_N/N}, \mu_{(k_N - K_N)/(N - 2K_N)} ) + \dS(\mu_{(k_N - K_N)/(N - 2K_N)}, \mu_t).
	\end{multline*}
	The first and last summands converge to $0$ by continuity (with respect to $t$) of $(\hat{\mu}_t)_t$ and $(\mu_t)_t$, respectively. The second summand converges to $0$ because $(\hat{\mu}^{N}_t)_t$ converges uniformly to $(\hat{\mu}_t)_t$. Eventually the third summand is by definition (see~\eqref{eq:p:gamma_limsup_path}) 
	\begin{equation*}
		\dS(\hat{\mu}^N_{k_N/N}, \mu_{(k_N - K_N)/(N - 2K_N)} ) = \dS(R_{M_N}(\mu_{(k_N - K_N)/(N - 2K_N)}), \mu_{(k_N - K_N)/(N - 2K_N)} ),
	\end{equation*}
	thus it also converges to $0$ as $N \to \infty$ by Theorem~\ref{theorem:unif_conv_RM}. We conclude that $\hat{\mu}_t = \mu_t$ for all $t \in (0,1)$, thus for all $t \in [0,1]$ by continuity. 
\end{proof}

\section{Numerical computation of \texorpdfstring{$\dS$}{the Sinkhorn distance}}

We discuss how to numerically optimize the Sinkhorn chain
\begin{equation*}
	\ol{E}^N({\mu}_0,{\mu}_1,\dots,{\mu}_N) \assign N \cdot \sum_{k=0}^{N-1} S_\varepsilon({\mu}_k, {\mu}_{k+1}),
\end{equation*}
as its minimizers approximate geodesics in $\dS$ (Corollary~\ref{cor:chain_energy_approx}). We provide some numerical examples to illustrate the behavior of the geometry of $\dS$. 

While $\ol{E}^N$ discretizes the temporal variable, we also have to discretize the measures spatially. For this, there are two established approaches. Both encode a measure as a collection of point masses \(\mu = \sum_{i=1}^M m_i \delta_{x_i}\). \emph{Eulerian} discretization uses a fixed discretized base space \(X_M = \{x_1, \dots, x_M\}\), where a measure \(\mu\) is given by its weights \(m_i \assign \mu(\{x_i\})\). Convergence of this approach in the limit of increasingly finer position sets is provided by the Gromov--Hausdorff convergence studied in Section~\ref{sec:Gromov_Hausdorff}. Alternatively, \emph{Lagrangian} discretization fixes the weights \(m_1, \dots, m_M\), typically \(m_i \assign 1/M\) for all \(1 \leq i \leq M\), and a measure is represented by the positions \(x_i \in \R^d\). Of course one may also optimize over both weights and positions but we only consider the two separate variants for simplicity. 

While Eulerian discretization has better convexity properties in the context of classical Wasserstein problems, it will be impractical on high-dimensional base spaces. Lagrangian discretization destroys this convexity structure and the convergence of the optimization algorithm is usually more challenging to analyze, but given an equal amount of points, it can in principle approximate a larger set of measures reasonably well. In practice one also observes that numerical optimization with Lagrangian methods is well behaved. For a brief discussion of both approaches and a theoretical result that explains the strong practical performance of Lagrangian discretization in the context of the classical Wasserstein distance, we refer to \cite[Section 4.3.3]{Feydy2020} and \cite{MerigotSantambrogioSarrazin}.

\subsection{Eulerian discretization: method}

We first discuss the Eulerian case, where the support of the measures is fixed and only the weights are optimized. 

\paragraph{Partial convexity and alternating minimization.}
The entire objective function \(\ol{E}^N({\mu}_0,{\mu}_1,\dots,{\mu}_N)\) is not convex in its input measures \cite[Sec.~7.2]{RGSD}. Recall, however, that the Sinkhorn divergence is convex in each of its arguments separately (Theorem~\ref{theorem:FeydyEtAl}). Here convexity is considered with respect to linear (vertical) interpolation in \(\prm(X)\). For any \(1 \leq i \leq N-1\) the contribution of \(\mu_i\) to the Sinkhorn chain score is given by
\begin{equation} \label{eq:chain_score_element}
	S_\varepsilon(\mu_{i-1},\mu_i) + S_\varepsilon(\mu_i,\mu_{i+1})
\end{equation}
which is convex in \(\mu_i\).

We see that from the perspective of the collection of measures $(\mu_{i})_{i\,:\,i\tn{ even}}$ for all even indices $i$, the functional $\ol{E}^N$ decomposes into a sum of independent convex functionals (and likewise for the odd indices). We can therefore perform alternating convex minimization on odd and even indices. This will monotonically decrease the objective and all fixed points are critical points of the energy.

\paragraph{The Sinkhorn barycenter problem.}
The barycenter problem of finding the minimizer \(\bar{\mu}_i\) of~\eqref{eq:chain_score_element} in Eulerian discretization has been studied in the literature. Using the primal formulation for the cross-transport terms \(\OT_\varepsilon(\mu_{i \pm 1},\mu_i)\) and the dual formulation for the subtracted self-transport term \(\OT_\varepsilon(\mu_i,\mu_i)\) provides a convex optimization problem that can be written as the sum of a \(\KL\) divergence and a kernel norm~\cite[Prop.~5]{Janati20}. Alternating minimization of these two parts leads to a Sinkhorn-like algorithm for the Sinkhorn barycenter problem~\cite[Alg.~1]{Janati20}. It is formulated with respect to the uniform reference measure, but since the Sinkhorn divergence is identical when defined with respect to the uniform or product reference measure (see~\cite[Eq.~9]{Janati20}), this does not affect the resulting barycenter. 

\paragraph{Refining the chain.}
Given fixed source and target measures \(\mu^\mathrm{S}, \mu^\mathrm{T}\), we can begin with \(N = 2\), set \(\mu_0 = \mu^\mathrm{S}\), \(\mu_2 = \mu^\mathrm{T}\) and optimize \(\mu_1\) to approach the barycenter. We can then refine the temporal discretization inserting new elements within each time step. This leads to a new chain with \(2 N\) many steps. We then apply alternating optimization on the refined chain, where we begin by optimizing over the newly added measures. This scheme of refinement and alternating optimization can be repeated until a desired level of temporal resolution is reached. We conjecture that incremental refinement of the chain will converge to high-quality candidates faster, since in a poorly initialized chain of high temporal resolution it will take a long time until information from the end points propagates through the whole chain by alternating optimization. Eventually, when the temporal resolution is sufficiently fine, by Theorem~\ref{theorem:expansion_dS} one can expect vertical interpolation to give a good initialization for the barycenter problem of the newly added chain elements.

\begin{algorithm}
	\caption{Sinkhorn chain optimization} \label{alg:chain_alternating}
	\begin{algorithmic}
		\State \textbf{Input:} Source $\mu^\mathrm{S}$, target $\mu^\mathrm{T}$, resolution $N = 2^R$ 
		\State \textbf{Output:} Chain $(\mu_k)_{k=0}^N$ 
		\State Initialize $\mu^0_0 = \mu^\mathrm{S}, \mu^0_1 = \mu^\mathrm{T}$
		\For{$r = 1, \dots, R$}
		\For{$k = 0, \dots, 2^{r-1}$}
		\Comment{Initialize the refined chain}
		\State $\mu^r_{2k} \gets \mu^{r-1}_k$
		\EndFor
		\Repeat
		\Comment{Alternating minimization}
		\For{$k = 0, \dots, 2^{r-1}-1$}
		\Comment{over new chain links}
		\State $\mu^r_{2k+1} \gets \mathrm{Sinkhorn\_Barycenter}(\mu^{r}_{2k},\mu^{r}_{2k+2})$
		\EndFor
		\For{$k = 1, \dots, 2^{r-1}-2$}
		\Comment{and old chain links}
		\State $\mu^r_{2k} \gets \mathrm{Sinkhorn\_Barycenter}(\mu^{r}_{2k-1},\mu^{r}_{2k+1})$
		\EndFor
		\Until{convergence}
		\EndFor
	\end{algorithmic}
\end{algorithm}

\begin{figure}
	\begin{center}
		
		\includegraphics[scale=0.45]{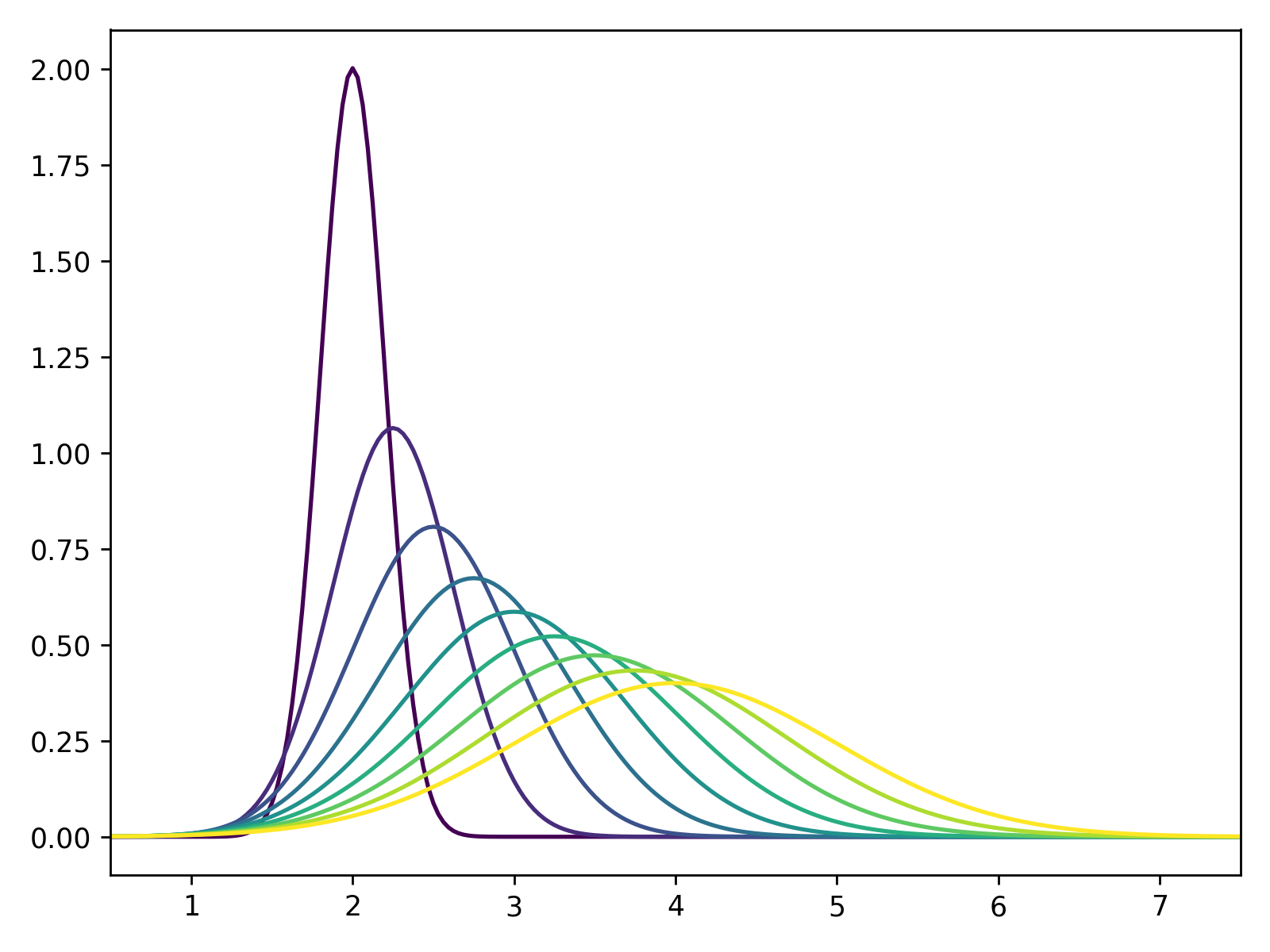}
		\includegraphics[scale=0.6]{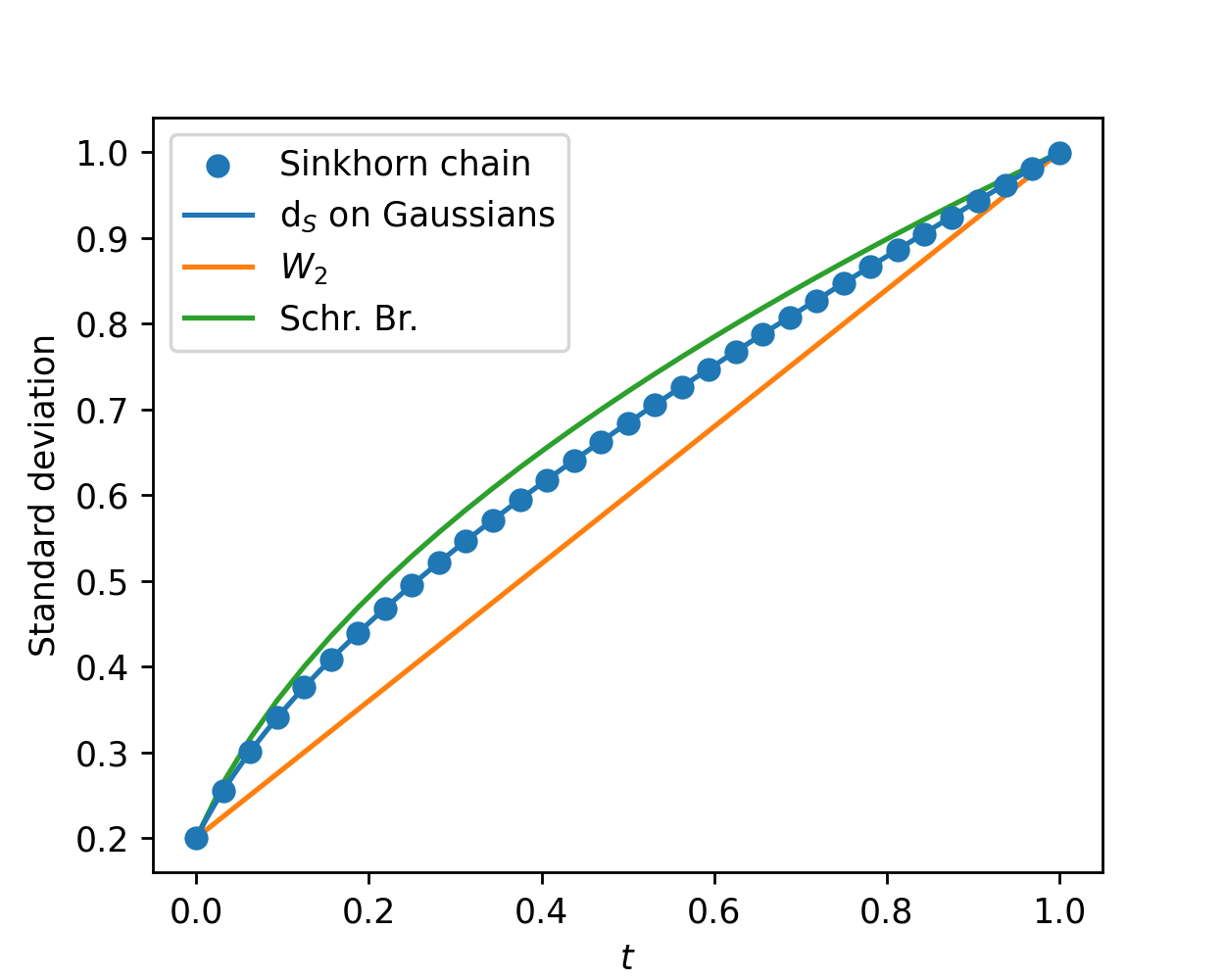}
		
		\caption{\label{fig:Gaussian} Left: Lebesgue density of the Eulerian chain between two Gaussians with standard deviations \(0.2\) and \(1\) in \(\R\) for the cost function \(c(x,y) = \frac{1}{2} \abs{x-y}^2\) with \(\varepsilon=1\). Computed in Eulerian discretization of \(X = [0,8]\) with \(M = 257\) equi-spaced points with \(N = 32\) time steps (showing eight). The movement decomposes into a constant shift of the mean and the change of the variance. Right: Corresponding standard deviations. Wasserstein-2 geodesic and Schrödinger bridge (also with $\varepsilon=1$) for comparison. Numerical results align with the theoretical prediction from~\cite[Thm.~6.2]{RGSD}. }
	\end{center}
\end{figure}

\paragraph{Numerical implementation.}
Our implementation follows the described alternating minimization scheme. 
We iterate the barycenter problem until the marginal conditions and the optimality condition between the self-transport potentials and the inner cross-transport potentials are satisfied up to some threshold. After the refinement step, the Sinkhorn barycenter algorithm~\cite[Alg.~1]{Janati20} is initialized at the uniform measure. In subsequent iterations we perform a warm start with the previous barycenter and Sinkhorn scaling factors. We alternate between even and odd chain links until the same consistency conditions are satisfied on the entire chain, not only the part that was last updated (necessarily with a slightly looser threshold), or a maximum number of alternating steps is reached.

\subsection{Lagrangian discretization: method}

Recall that Lagrangian discretization has measures given as \(\mu = \sum_{i=1}^{M} m_i \delta_{x_i}\) with varying positions \(x_i \in \R^d\) and fixed masses \(m_i\). In order to be able to numerically capture mass splitting, we can encode each particle of the source and target measures by multiple lighter particles (with a very small amount of spatial noise for symmetry breaking). 
For computations we lose the partial convexity of each chain link. Hence there is no clear benefit to the alternating optimization from Algorithm~\ref{alg:chain_alternating} and instead we minimize the entire chain at once using gradient descent. 

\paragraph{Computing the gradients.}
Gradients of the entropic optimal transport score can be obtained efficiently without needing a backward pass of an automatic differentiation engine through the entire Sinkhorn algorithm~\cite{Ablin2020}, since the derivative of the score function (defined as a supremum over functions of \((\mu,\nu)\) in its dual formulation~\eqref{eq:entropic_dual}) equals the derivative of the maximizing function, resulting in \(\nabla \OT_\varepsilon(\mu,\nu) = (\fmunu,\gmunu)\). This can be implemented by first running the Sinkhorn loop until convergence with automatic differentiation disabled, and then performing one last Sinkhorn step that sets the gradients. As moving points encounter spatial gradients of the potentials in the gradient descent formula, we can only expect this approach to work well when the cost function is at least $C^1$-differentiable.

\paragraph{Encoding and refining the chain.}
This can also be done in the Lagrangian approach. We consider two strategies of encoding the positions of the measures \((\mu_{k})_{k=1}^N-1\) in between \(\mu^\mathrm{S} = \sum_{i=1}^{M} m_i \delta_{x^\mathrm{S}_i}\) and \(\mu^\mathrm{S} = \sum_{i=1}^{M} m_i \delta_{x^\mathrm{S}_i}\). We can directly encode \(\mu^k\) by its positions \(x^k_i \in \R^d\) (and then initialize these randomly, for instance). The Sinkhorn chain at any finite temporal resolution is agnostic as to which particle in one time step corresponds to which particle in the next time step, hence there is no clear identification of \(x^k_i\) and \(x^l_i\) for \(k \neq l\). Not even the number of particles of \(\mu^k\) needs to be fixed. 

However, we conjecture that in the continuous time limit of the restriction to measures with a fixed number of particles \(M\), each particle will trace a clear path. Thus we may also encode \(\mu_k\) by velocities as \(x^k_i = x^\mathrm{S}_i + \frac{1}{N} \sum_{l=1}^k v^l_i\) optimizing over \(v^l_i \in \R^d\) (initialized as zero). In our experiments, this leads to more visually consistent paths when the number of particles is relatively low. We can also use this to easily initialize the new measures during refinement of the chain by placing particles at the midpoints \(x^k_i + \frac{1}{2 N} v^{k+1}_i\).

\begin{figure}
	\begin{center}
		
		\includegraphics[scale=0.4]{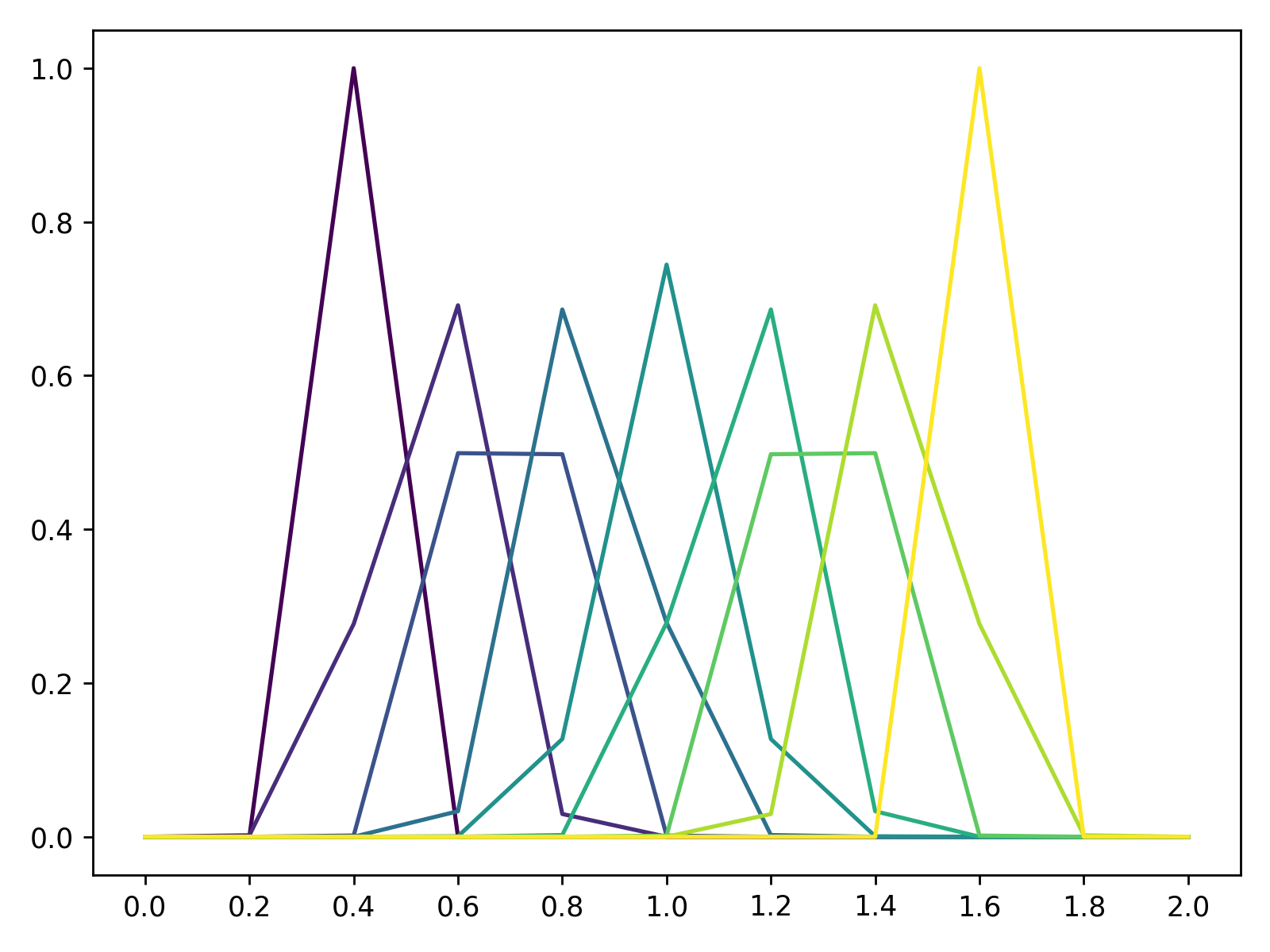}
		\includegraphics[scale=0.44]{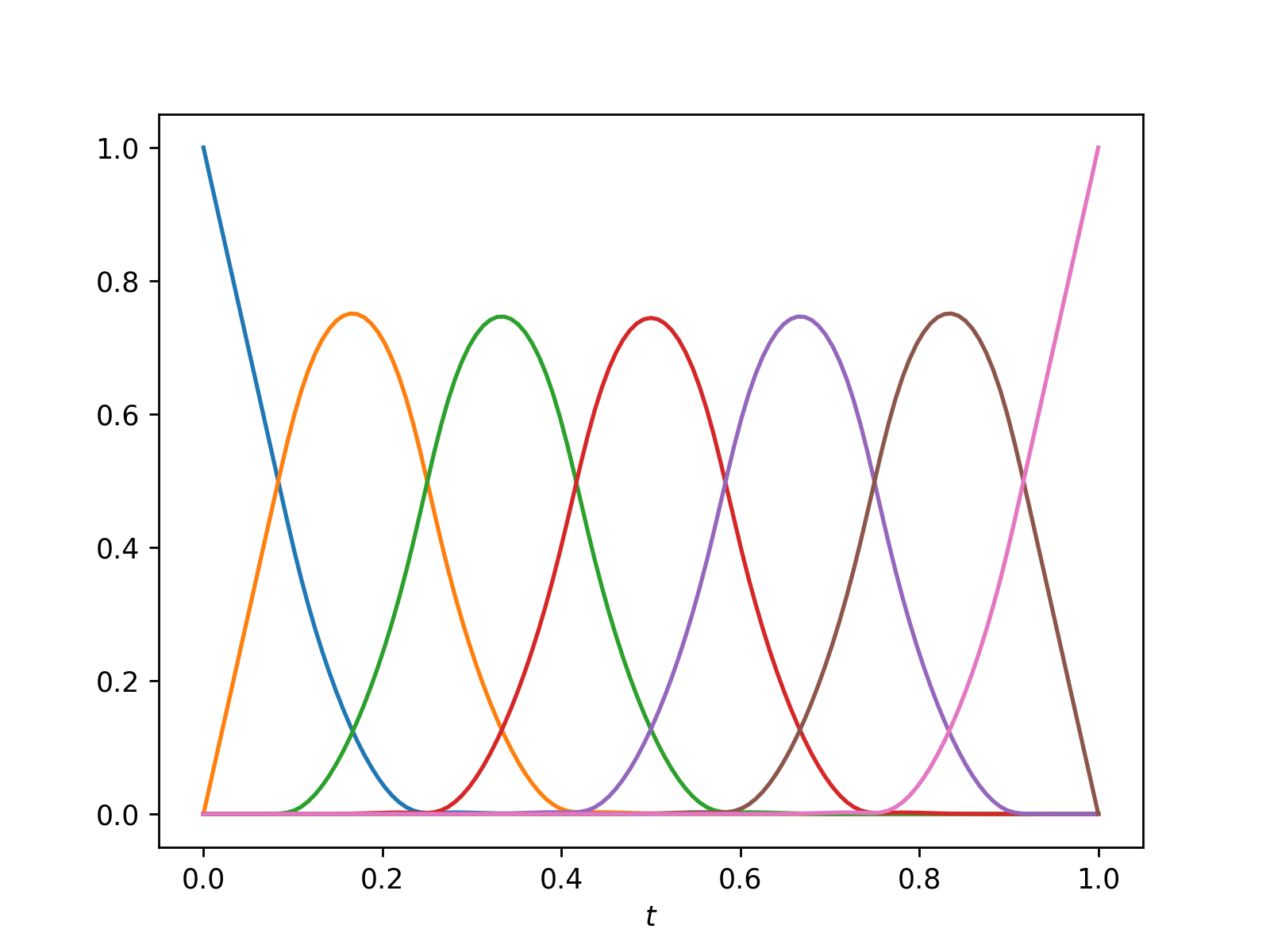}
		
		\includegraphics[scale=0.4]{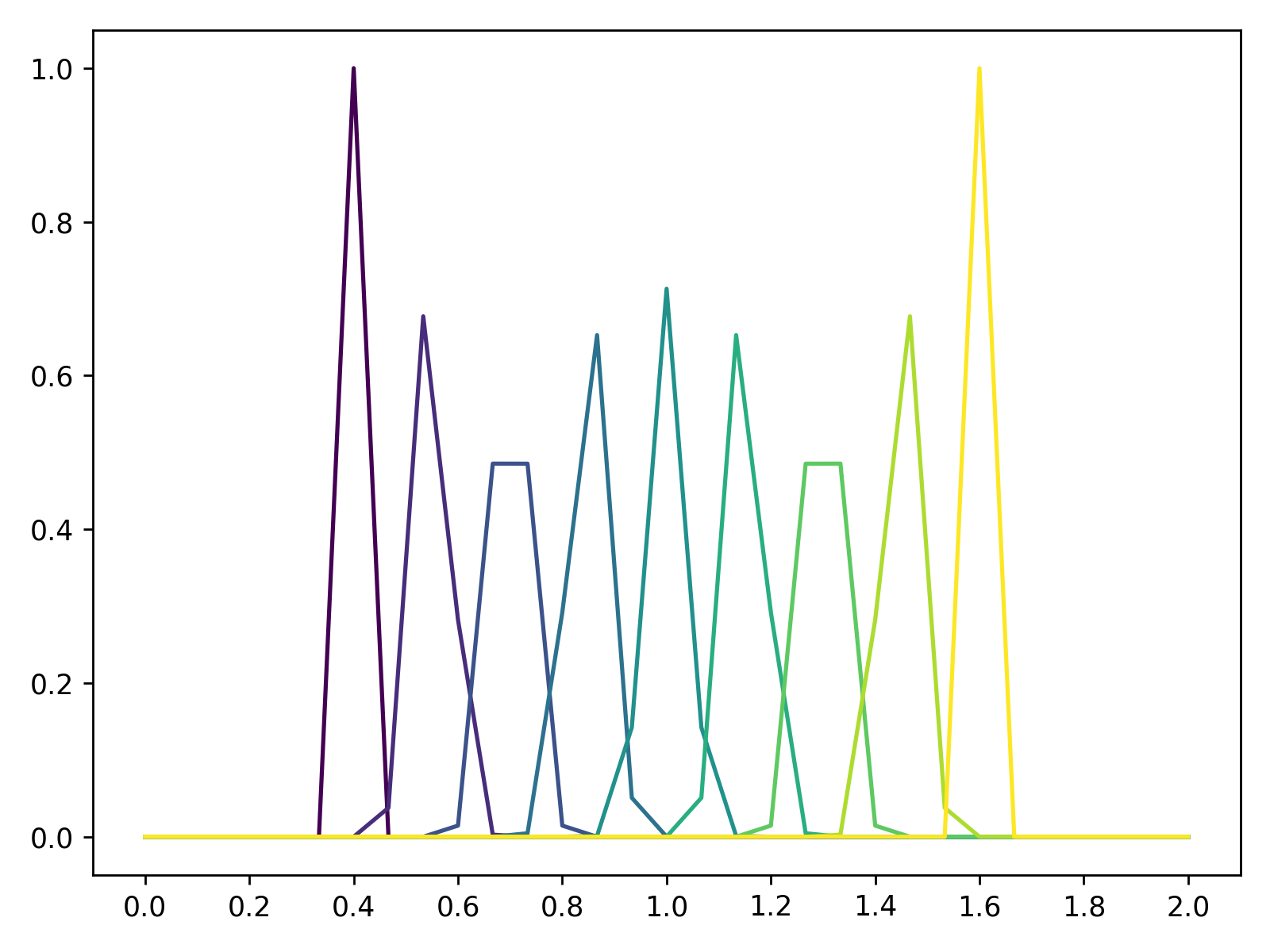}
		\includegraphics[scale=0.44  ]{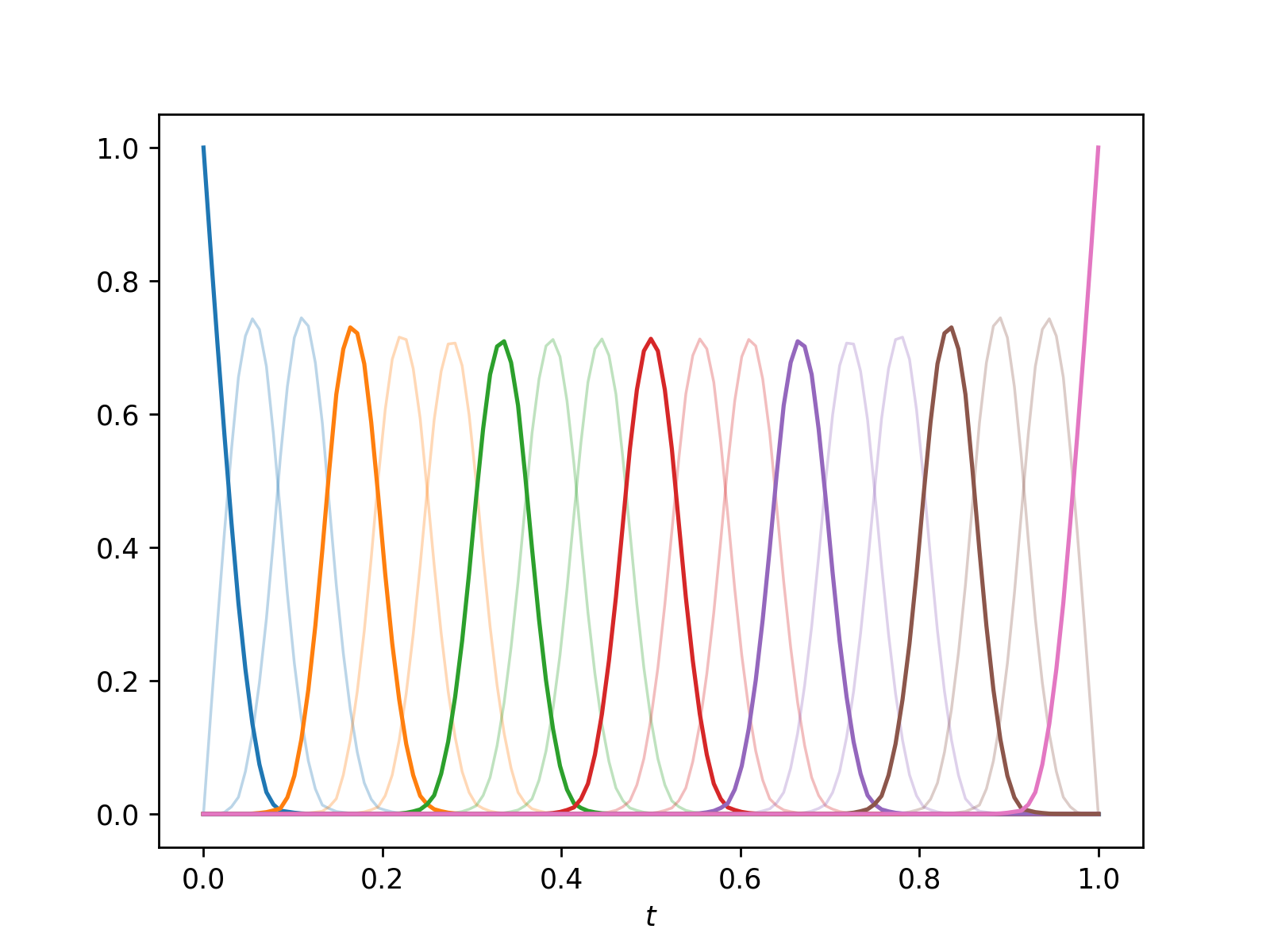}
		
		\caption{\label{fig:Eulerian_Dirac} Left: Eulerian chain between Dirac measures at \(x=0.4\) and \(x=1.6\) with the cost function \(\frac{1}{2} \abs{x-y}^2\) with \(\varepsilon=1\). Computed in Eulerian equi-spaced discretization of \([0,2]\) to \(M=11\) (top), or \(M=31\) (bottom) points, and \(N = 128\) time steps (showing eight). 
			Right: Mass at the discrete positions over time (blue at \(x=0.4\) to pink at \(x=1.6\)).
		}
	\end{center}
\end{figure}

\subsection{Examples} 

\begin{figure}
	\begin{center}
		
		\includegraphics[width=0.49\textwidth]{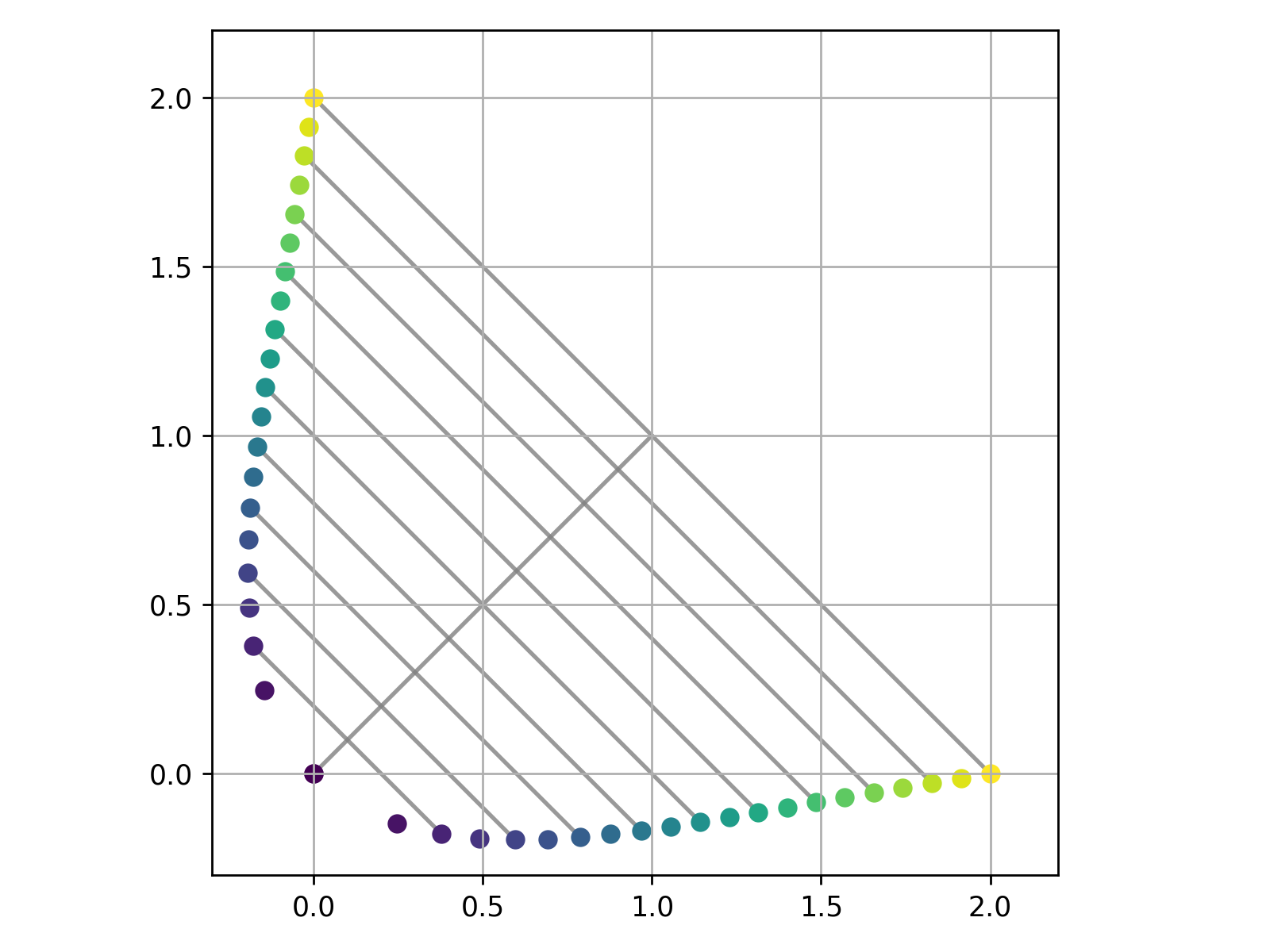}
		\includegraphics[width=0.49\textwidth]{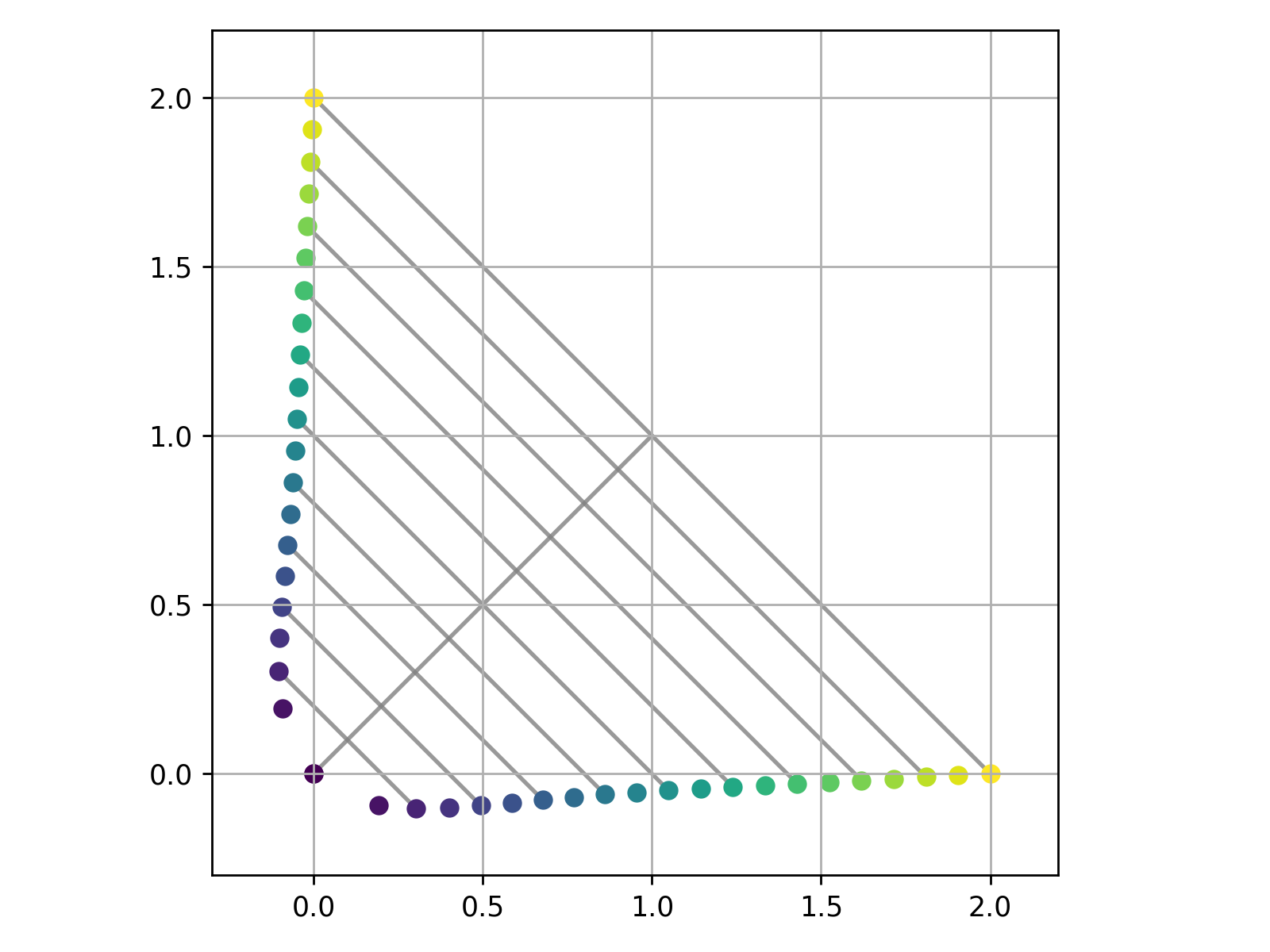}

		\caption{\label{fig:split_dirac_2d} Lagrangian chain between $\mu^\mathrm{S} = \delta_{(0,0)}$ and $\mu^\mathrm{S} = \frac{1}{2} (\delta_{(2,0)} + \delta_{(0,2)})$ in \(\R^2\) for the cost function \(c(x,y) = \frac{1}{2} \abs{x-y}^2\) and \(\varepsilon = 1\) (left), or \(\varepsilon=0.2\) (right). The movement decomposes into a constant shift of the mean along the diagonal and the one dimensional splitting Dirac along the antidiagonal (visualized by grey auxiliary lines). For smaller $\varepsilon$ the effect of rapid initial splitting is less pronounced and the curve is more similar to the classical Wasserstein geodesic.}
	\end{center}
\end{figure} 

\paragraph{Geodesics between Gaussians.} 
Due to their non-compact support, normal distributions are not within the technical scope of assumptions of \cite{RGSD} and the present article. Nevertheless, \cite[Section 6.1]{RGSD} formally derived the expression for the metric tensor on the subset of one-dimensional Gaussian distributions and studied the corresponding geodesic restriction of the Sinkhorn distance.

Figure~\ref{fig:Gaussian} shows an approximated geodesic between two normal distributions in Eulerian discretization and finds that the evolution of the variance agrees with the behaviour derived in \cite[Section 6.1]{RGSD}. 

This provides some numerical evidence that the set of one-dimensional normal distributions might be geodesically convex with respect to the Sinkhorn distance. Note that \cite[Thm.~3]{Janati20} shows that the Sinkhorn barycenter between two normal distributions is itself a normal distribution. Hence our Algorithm~\ref{alg:chain_alternating} will always retain a purely Gaussian chain when initialized as such. However, it is unclear whether this is the minimum, or merely a saddle point of the energy functional \(\bar{E}^N\). 
Further research is required to settle this question.

\paragraph{Discretization of translations.} 
By \cite[Thm.~5.16]{RGSD} constant-speed translations of measures on bounded subsets of $\R^d$ are geodesic with respect to the Sinkhorn distance. By our Gromov--Hausdorff convergence result (Theorem \ref{theorem:unif_conv_RM}) this behaviour should be approximately recovered on fine Eulerian discretizations.
This is shown in Figure~\ref{fig:Eulerian_Dirac} for a geodesic between two Dirac masses on \(X = [0,2]\). 
As the grid is refined, one seems to recover the limiting curve of a moving Dirac mass. Note however, that the support of the intermediate measures always seems to be more than one point.
Notably, even though the middle position \(x=1\) is part of the discretized space, the midpoint measure along the chain is not solely supported on $\{1\}$.

\paragraph{Mass splitting.} 
In Example~\ref{exm:W2_not_dS_geodesic_splitting_dirac} we studied splitting Dirac measures, paths between source $\mu^\mathrm{S} = \delta_0$ and target $\mu^\mathrm{S} = \frac{1}{2} (\delta_r + \delta_{-r})$.
The optimal path from \(\mu^\mathrm{S}\) to \(\mu^\mathrm{S}\) along measures of this form shoots apart very quickly initially, settling into more linear motion later on. See also Figure~\ref{fig:split_dirac_2d} for a numerical approximation in Lagrangian discretization. The same effect also happens between measures with densities, as can be seen in Figure~\ref{fig:Gaussian_split} (left). Near $t=0$, quickly shooting mass to both sides equally is so cheap in the Sinkhorn metric tensor, that the initial split happens rapidly, forming two spikes of high concentration.
This is because the spatial smoothing built into the Sinkhorn metric tensor approximately cancels these two opposite contributions.
Once these spikes are further apart than the blur scale \(\sqrt{\varepsilon}\) the cancellation effects subside and the movement becomes more gradual.

\subsection{Comments on the optimization}

\paragraph{Barycenter spikes.} 
In our numerical experiments we frequently observe that Sinkhorn barycenters between non-Gaussian measures exhibit strong concentrated spikes, approximately separated by distance \(\sqrt{\varepsilon}\), Figure~\ref{fig:Gaussian_split} (right) for \(N=2\).
As the Sinkhorn chain is refined, this effect seems to become weaker, Figure~\ref{fig:Gaussian_split} (right).
We also observe however, that minima of the chain energy are very shallow as we discuss below and the emergence or disappearance of spikes is associated with minuscule changes in the energy and therefore an immensely large number of iterations may be required to properly resolve the appearance of the true geodesic.

\begin{figure}
	\begin{center}
		
		\includegraphics[height=0.42\textwidth]{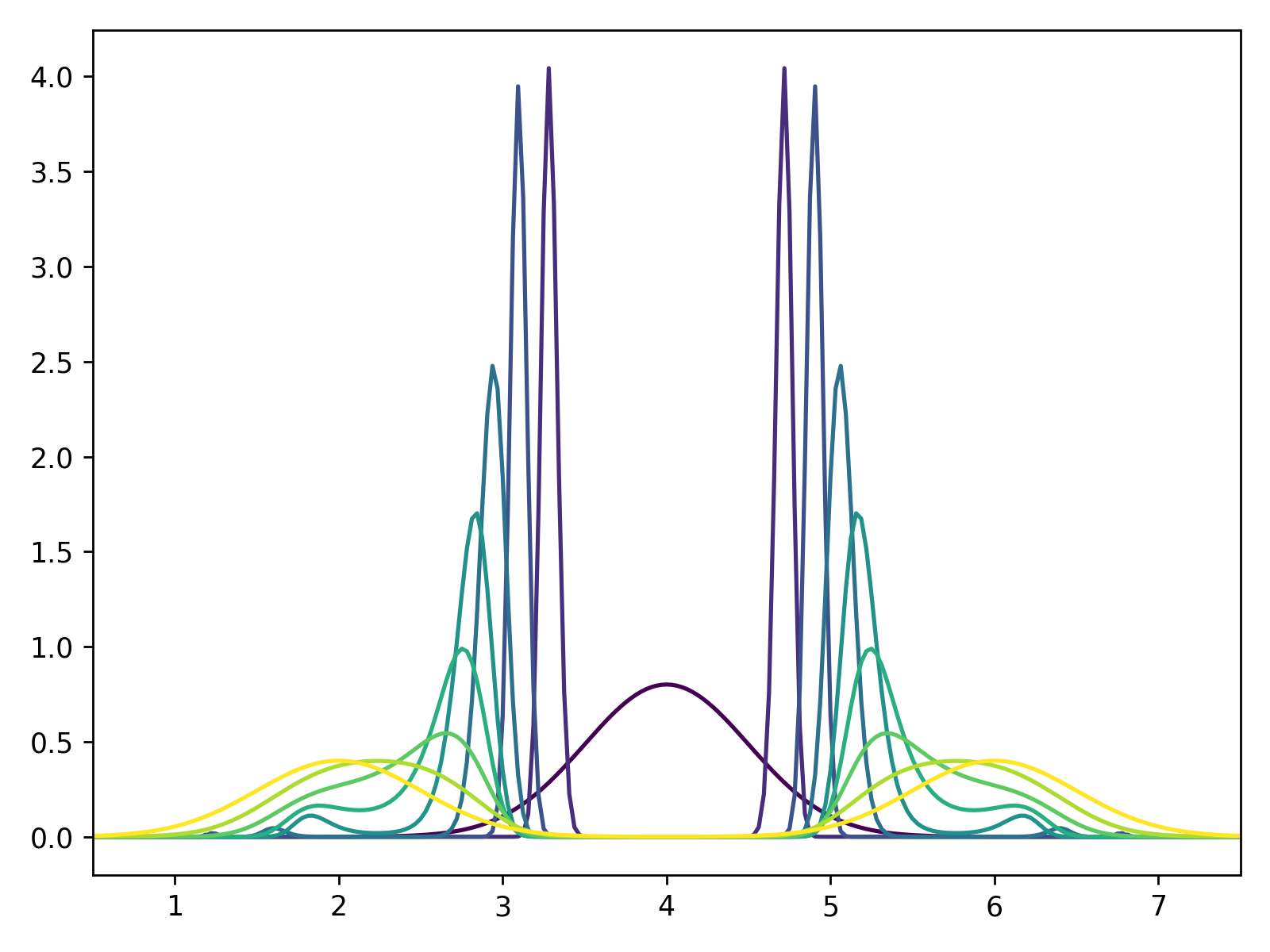}
		\includegraphics[height=0.42\textwidth]{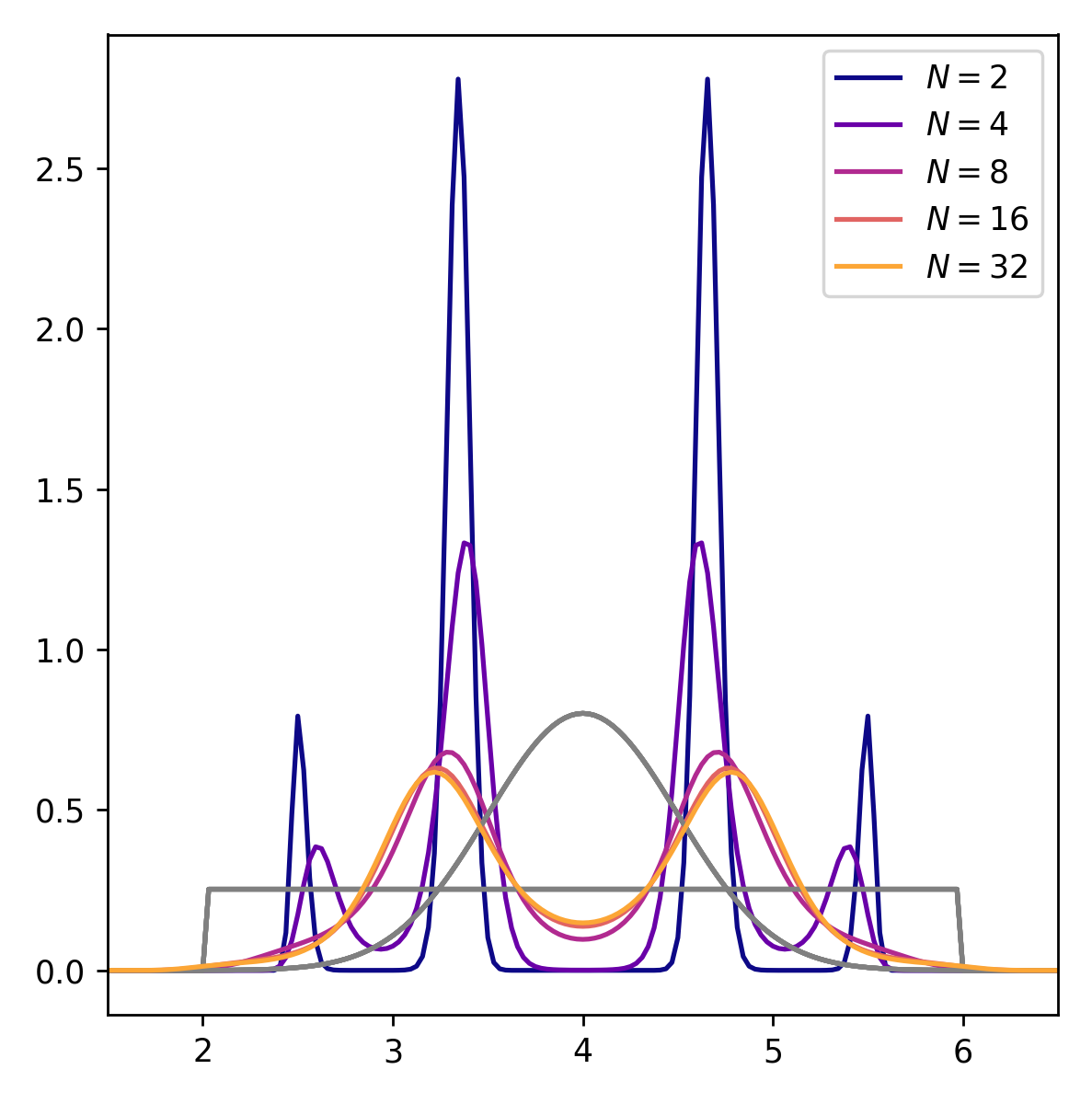}
		
		\caption{\label{fig:Gaussian_split} 
			Lebesgue densities along the Sinkhorn chain in \(\R\) for the cost function \(c(x,y) = \frac{1}{2} \abs{x-y}^2\) with \(\varepsilon=1\). Computed in Eulerian discretization of \(X = [0,8]\) with \(M = 257\) equi-spaced points.
			Left: Chain between a Gaussian with standard deviation \(0.5\) (purple) and a mixture of two Gaussians (yellow) with \(N = 32\) time steps (showing eight). We observe behaviour similar to the splitting Dirac from Figure~\ref{fig:split_dirac_2d}. The mass shoots apart rapidly in the beginning, forming two concentrated spikes before relaxing to the target configuration. Also notice that at intermediate times additional weaker modes form.
			Right: Midpoint along the chain of temporal resolution \(N\) between the same Gaussian and the uniform measure on an interval (both endpoints in grey). For \(N=2\) we get the Sinkhorn barycenter with multiple strong spikes. As \(N\) increases, the spikes disappear.
		}
	\end{center}
\end{figure}

\begin{figure}[t]
	\begin{center}
		
		\includegraphics[scale=0.7]{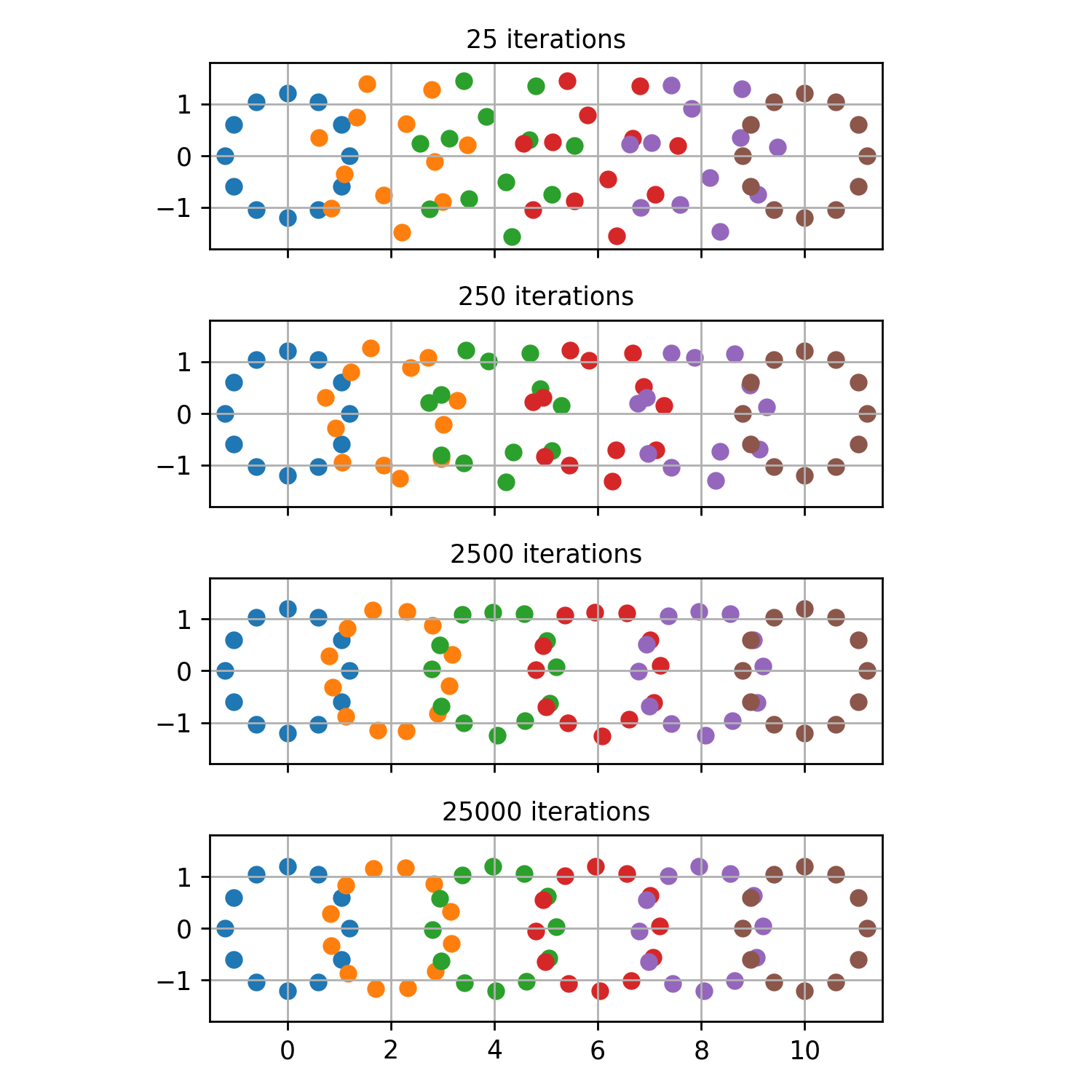}
		
		\caption{\label{fig:circle_alignment} Lagrangian chain with four intermediate measures between two circles of radius \(1.2\) discretized as \(M=12\) equispaced points for the cost function \(c(x,y) = \frac{1}{2} \abs{x-y}^2\) and \(\varepsilon=1\). Convergence of gradient descent over time with step size parameter \(1\) for position parametrization. The exact geodesic is known to be the constant speed translation, giving \(\dS(\mu^\mathrm{S},\mu^\mathrm{S})^2 = 50\). This figure demonstrates that the minimum is very shallow. Already after 25 iterations the score is \(\approx 50.027\), yet the points are ``far'' away from the optimal translation, for instance in terms of the unregularized Wasserstein metric.}
	\end{center}
\end{figure} 

	\paragraph{Shallow minima.} 
	In our experiments the objective value of the Sinkhorn chain often approaches the minimal value well before the chain visually stabilizes. In particular, it takes many iterations for the chain to align to finer details of the source and target measures when initializing randomly in Lagrangian discretization with position parametrization.
	This suggests that minima of the path energy are rather shallow.
	This is because the Sinkhorn distance does not discriminate strongly between a measure and a slightly blurred version of it: in the setting of Remark~\ref{rmk:diff_setting}, in particular for $c(x,y) = \|x-y \|^2$, 
	we have $\dS(\mu,\nu) \leq C_m \| \mu - \nu \|^2_{\Cont^m(X)^*}$~\cite[Prop.~5.9]{RGSD} for any $m \geq 1$. Any perturbation which is small in the space $\Cont^m(X)^*$ does not influence the energy much, in particular it applies to blurring, de-blurring, or splitting in opposite directions.
	The fact that splitting mass in opposing directions is so cheap leads to the fast initial splitting in Figures \ref{fig:Gaussian_split} and \ref{fig:split_dirac_2d} (see also Example \ref{exm:W2_not_dS_geodesic_splitting_dirac}).
	
	Figure~\ref{fig:circle_alignment} illustrates the shallow nature of the minimum of the Sinkhorn chain energy. Here again, the energy approaches the true squared distance quickly, while the points take many more iterations to align to a coherent path. Notice though, that even after 25000 iterations, the orange circle is still rotated by half a discretization angle, something which the Sinkhorn divergence (at the given blur scale) barely notices.
	
	With few particles moving around, we obtain clearly traced paths of each particle, even when starting from pure noise in the position parametrization. 
	In cases with more particles moving around however, this alignment is more difficult (see Figure~\ref{fig:circle_alignment}). The entropic regularization breaks the clear correspondence between particles that one gets for instance in Wasserstein-2 transport and which leads to its Benamou--Brenier formulation~\eqref{eq:Benamou_Brenier}.
	
	Performing the same experiment as in Figure~\ref{fig:circle_alignment} in velocity parametrization quickly recovers the geodesic given by the constant speed translation. Also in more complex settings, it allows gradient information to pass through the entire chain at once, not only within one link. In our experiments it always yielded equal or better results than position encoding. The faster convergence seems to outweigh possible bias towards clearly traced paths.

	\section*{Acknowledgments}
    J.L.~and B.S.~were supported by the Deutsche Forschungsgemeinschaft (DFG) through the Emmy Noether Programme (project 403056140) and project 569461875.
    No AI tool was used for the conception and mathematical content of this project, except for fixing a few English and typographic errors.
	
	\bibliographystyle{abbrv}
	{\small
		\bibliography{references}}    
\end{document}